\documentclass[a4paper,10pt,oneside]{articleHJ}
\usepackage{amsfonts,amssymb,amsmath,amsthm}
\usepackage[a4paper, total={6in,9in}]{geometry}
\usepackage[utf8]{inputenc}
\usepackage{graphicx}

\usepackage{subcaption}
\usepackage{tikz}
\usepackage{enumerate}
\usepackage{cases}
\usepackage{multicol}
\usepackage{sidecap}
\usepackage{float}
\usepackage{xcolor}
\definecolor{blue2}{rgb}{0.67, 0.9, 0.93}
\usepackage[color=blue2]{todonotes}
\usepackage{cases}
\usepackage{comment}
\normalfont
\usepackage[T1]{fontenc}
\usepackage[utf8]{inputenc}	
\usepackage[normalem]{ulem}
\usepackage[hidelinks=true]{hyperref}

\numberwithin{equation}{section}

\newtheorem{thm}{Theorem}[section]
\newtheorem{lemma}[thm]{Lemma}
\newtheorem{prop}[thm]{Proposition}
\newtheorem{cor}[thm]{Corollary}
	
\theoremstyle{definition}

\newcommand{\N}{{\mathbb N}}
\newcommand{\R}{{\mathbb R}}

\newcommand{\Z}{\mathbb{Z}}

\usepackage{setspace}

\newcommand\norm[1]{\left|\left|{#1}\right|\right|}

\begin{document}

\begin{frontmatter}
\title{Dynamics of curved travelling fronts for the discrete Allen-Cahn equation on a two-dimensional lattice}
\journal{...}
\author[LD1]{M. Juki\'c\corauthref{coraut}},
\author[LD2]{H. J. Hupkes},
\corauth[coraut]{Corresponding author. }
\address[LD1]{
  Mathematisch Instituut - Universiteit Leiden \\
  P.O. Box 9512; 2300 RA Leiden; The Netherlands \\ Email:  {\normalfont{\texttt{m.jukic@math.leidenuniv.nl}}}
}
\address[LD2]{
  Mathematisch Instituut - Universiteit Leiden \\
  P.O. Box 9512; 2300 RA Leiden; The Netherlands \\ Email:  {\normalfont{\texttt{hhupkes@math.leidenuniv.nl}}}
}

\date{\today}

\begin{abstract}
\singlespacing
In this paper we 
consider the discrete Allen-Cahn equation posed on a two-dimensional rectangular lattice. We analyze the large-time behaviour of solutions that start
as bounded perturbations
to the well-known planar front solution that travels in the horizontal direction.
In particular,
we construct an asymptotic phase function $\gamma_j(t)$ and show that for each vertical coordinate $j$ the 
corresponding horizontal slice of the solution converges to the planar front
shifted by $\gamma_j(t)$. We exploit the comparison principle to show that the evolution of these phase variables
can be approximated by an appropriate discretization of the mean
curvature flow with a direction-dependent drift term. This generalizes the results obtained in \cite{Matano} for the spatially continuous setting. 
Finally, we prove that the horizontal planar wave is nonlinearly stable
with respect to perturbations that are asymptotically periodic in the vertical direction.

\end{abstract}

\begin{subjclass}
\singlespacing
34K31 \sep 37L15.
\end{subjclass}

\begin{keyword}
\singlespacing
Travelling waves, 
bistable reaction-diffusion systems,
spatial discretizations,
discrete curvature flow,
nonlinear stability,
modified Bessel functions of the first kind.
\end{keyword}

\end{frontmatter}

\section{Introduction}
\label{sec:int}

Our main aim in this paper is to explore the large time behaviour of the  Allen-Cahn lattice differential equation
(LDE)
\begin{equation}\label{eqn:intro:Allen_Cahn_LDE}
\dot{u}_{i,j}= u_{i+1,j} + u_{i-1,j} + u_{i, j+1} + u_{i,j-1} - 4u_{i,j} + g(u_{i,j};a)
\end{equation}
posed on the planar lattice $(i,j) \in \Z^2$. The nonlinearity $g(\cdot;a)\in C^2(\R)$ is of bistable type, in the sense that it has two stable equilibria at $u=0$ and $u=1$ and one unstable equilibrium at $u=a\in (0,1)$. The prototypical  example is the cubic 
\begin{equation}
\label{eq:int:def:cubic}
    g_{\mathrm{cub}}(u;a) = u(1-u)(u-a).
\end{equation}

We are interested in the stability properties of curved versions of the horizontal travelling front
\begin{equation}
\label{eq:int:trv:planar:wave:lde}
    u_{i,j}(t) = \Phi( i  - ct),
    \qquad \Phi(-\infty) = 0,
    \qquad \Phi(+\infty) = 1
\end{equation}
in the case where $c \neq 0$. In particular, for initial conditions that are $j$-uniformly `front-like' in the sense
\begin{equation}
\label{eq:int:limits:init:cond:lde}
    \limsup_{i\to-\infty}\sup_{j\in \R} u_{i,j}(0) < a, \qquad \qquad \liminf_{i\to\infty}\inf_{j\in \R} u_{i,j}(0) > a,
\end{equation}
we establish the uniform convergence
\begin{equation}
\label{eq:int:unif:cnv:to:phase:front}
    u_{i,j}(t) \to 
    \Phi\big( i - \gamma_j(t) \big),
    \qquad t \to \infty,
\end{equation}
for some appropriately constructed
transverse phase variables 
$\gamma_j(t)$.
In addition, we show that the evolution of these phases can be approximated by a discrete version of the mean curvature flow. 

After adding further restrictions
to \eqref{eq:int:limits:init:cond:lde}, a detailed analysis of this 
curvature flow allows us to establish the convergence $\gamma_j(t) \to ct + \mu$. In fact, it turns out that 
the set of initial conditions covered
by this result is significantly broader than the sets considered in earlier work \cite{hoffman2017entire,hoffman2015multi}. As a consequence, we widen the known basin of attraction for the planar horizontal wave \eqref{eq:int:trv:planar:wave:lde}. \\ 

\noindent \textbf{Continuous setting} The LDE~\eqref{eqn:intro:Allen_Cahn_LDE} can be seen as a discrete analogue of the two-dimensional Allen-Cahn PDE
\begin{equation}\label{eqn:intro:PDE}
u_t = u_{xx} + u_{yy} + g(u;a).
\end{equation}
Our primary interest here is in planar travelling travelling front solutions 
\begin{equation}\label{eqn:intro:PDE:wave:sol}
u(x,y,t) = \Phi(x\cos\theta + y\sin \theta - ct)
\end{equation}
that connect the two stable equilibria, in the sense that the waveprofile $\Phi$ satisfies
\begin{equation}
    \lim_{\xi \to -\infty} \Phi(\xi) = 0,
    \qquad 
    \qquad
    \lim_{\xi \to \infty} \Phi(\xi) = 1.
\end{equation}
Direct substitution shows that the wave $(\Phi,c)$ must satisfy the
$\theta$-independent ODE
 \begin{equation}\label{eqn:intro:tw_ODE_cont}
     - c\Phi'(\xi) = \Phi''(\xi) + g\big(\Phi(\xi);a\big),
 \end{equation}
 reflecting the rotational symmetry of \eqref{eqn:intro:PDE}. Indeed,
 \eqref{eqn:intro:tw_ODE_cont} also arises as the wave ODE for the
 one-dimensional counterpart
 \begin{equation}\label{eqn:intro:PDE:1d}
u_t = u_{xx} + g(u;a).
\end{equation}
of \eqref{eqn:intro:PDE}.
The existence of solutions to \eqref{eqn:intro:tw_ODE_cont}
can be obtained via phase-plane analysis \cite{fife2013mathematical} for any parameter $a\in (0,1)$. Moreover, the pair $(\Phi,c)$ is unique up to translations, depends smoothly on the parameter $a$, 
and admits the strict monotonicity $\Phi' > 0$.
\paragraph{Modelling background} 
Reaction-diffusion equations 
have been used as  modelling tools in many different fields. For example,
the classical papers \cite{aronson1975nonlinear,aronson1978multidimensional} use both one- and
multi-dimensional versions of such equations to describe the 
expression of genes throughout a population. Bistable nonlinearities
such as \eqref{eq:int:def:cubic}
are typically used to model the 
strong Allee effect - a biological phenomenon which arises in the field of the population dynamics \cite{taylor2005allee}.
Indeed, the parameter $a$
can be seen as a type of minimum viability threshold that a population needs to reach in order to grow,
in contrast to the standard logistic
dynamics. Adding the ability for the population to diffuse throughout its spatial habitat results in systems such as \eqref{eqn:intro:PDE} \cite{Sun2016}.
In this setting, travelling waves provide a mechanism by which species
can invade (or withdraw from) the spatial domain.

In many applications 
this spatial domain has a discrete structure, in which case it is more natural to consider the LDE \eqref{eqn:intro:Allen_Cahn_LDE}.
For example, in \cite{levin1976population,keitt2001allee}
the authors 
use this LDE to 
study populations in patchy landscapes.
This allows them to 
describe and analyze a so-called 
`invasion pinning' scenario,
wherein a species fails to propagate
as a direct consequence of the spatial discreteness.

By now, models involving LDEs have appeared in many other scientific and technological fields. For example, they have been used to describe
phase transitions in Ising models
\cite{bates1999discrete},
nerve pulse propagation in myelinated axons \cite{bell1981some,bell1984threshold,keener1987propagation,keener1991effects}, calcium channels dynamics \cite{bar2000discrete}, 
crystal growth in materials \cite{cahn1960theory}
and wave propagation through semiconductors \cite{carpio2001motion}.
For a more extensive list we recommend the surveys \cite{chow1996dynamics,chow1995pattern,hupkestraveling}.

\paragraph{Stability of PDE waves}
The first stability result for 
the wave \eqref{eqn:intro:PDE:wave:sol} in the one-dimensional
setting of \eqref{eqn:intro:PDE:1d} 
was established by Fife and McLeod in~\cite{fife1977approach}. In particular,
they showed that this wave (and its translates) attracts all solutions
$u$ with initial conditions that satisfy
\begin{equation}
\label{eq:int:bnds:init:cond:1d}
    \limsup_{x\to-\infty} u(x, 0) < a,
    \qquad \qquad
    \liminf_{x\to+\infty} u(x, 0) > a,
\end{equation}
together with $u(\cdot, 0) \in [0,1]$. This latter restriction was
later weakened to $u(\cdot, 0) \in L^\infty(\R)$  in \cite{fife1979long}.
Both these proofs rely  
on the construction of super- and sub-solutions for~\eqref{eqn:intro:PDE:1d} in order to exploit the comparison principle for parabolic equations.  More recently, similar
large-basin stability results have been obtained using variational methods
that do not appeal to the comparison principle \cite{gallay2007variational,risler2008global}.


In \cite{kapitula1997multidimensional}, Kapitula established the  multidimensional stability of traveling waves in $H^k(\R^n)$, for $n \ge 2$ and $ k\geq \lfloor \frac{n+1}{2} \rfloor $.
These results were recently extended
by Zeng \cite{Zeng2014stability},
who considered perturbations in 
$L^\infty(\R^n)$. 
An alternate stability proof exploiting the comparison principle can be found in
the seminal paper \cite{berestycki2009bistable},
where the authors study the interaction of travelling fronts with compact obstacles.
Let us also mention
the pioneering works \cite{XIN1992,LEVXIN1992} 
which contain the first stability results
for $n \ge 4$  together with partial results for $n=2,3$.

Based on the techniques developed by Kapitula, Roussier \cite{ROUSSIER2003}
was able to consider `asymptotically spherical' waves and establish their
stability under spherically symmetric perturbations. Such solutions behave as
\begin{equation}
    u(x,y,t) \to \Phi\Big( \sqrt{x^2 + y^2} - c t - c^{-1} \ln t \Big),
    \qquad 
    \qquad
    t \to \infty
\end{equation}
and were first studied by Uchiyama and Jones \cite{jones1983spherically, uchiyama1985asymptotic}. Note that the extra time dependence highlights
the important role 
that
curvature-driven effects have to play.

\paragraph{Curved PDE fronts}
Our work in the present paper is inspired heavily by the results
for \eqref{eqn:intro:PDE}
obtained by Matano and Nara in \cite{Matano}. They considered
bounded initial conditions satisfying the limits
\begin{equation}
\label{eq:int:front:like:pde:init:cond}
    \limsup_{x\to-\infty}\sup_{y\in \R} u_0(x,y) < a, \quad  \liminf_{x\to\infty}\inf_{y\in \R} u_0(x,y) > a,
\end{equation}
which form the natural two-dimensional generalization
of \eqref{eq:int:bnds:init:cond:1d}. They show that
eventually horizontal cross-sections of $u$ become sufficiently monotonic
to allow a phase $\gamma = \gamma(y,t)$ 
to be uniquely defined by the requirement
\begin{equation}
u\big( \gamma(y, t), y , t \big) = \Phi(0).
\end{equation}

These phase variables can be used to characterize the asymptotic behaviour of $u$. In particular, the authors establish the limit
\begin{equation}\label{eqn:intro:conv_gamma_PDE}
    \lim_{t\to\infty}\sup_{(x,y)\in \R^2}|u(x, y, t) - \Phi\big(x-\gamma(y,t)\big)|  = 0
\end{equation}
and show that - asymptotically - the phase $\gamma$ closely tracks
solutions $\Gamma$ to the PDE
\begin{equation}\label{eqn:intro:mcf}
 \dfrac{\Gamma_t}{\sqrt{1+\Gamma_y^2}} 
 = \dfrac{\Gamma_{yy}}
 {(1+\Gamma_y^2)^{3/2}}
 + c.
 %
\end{equation}
Upon supplementing \eqref{eq:int:front:like:pde:init:cond}
with the requirement that the initial condition
$u(\cdot, \cdot, 0)$ is uniquely ergodic in the $x$-direction, 
a careful analysis of \eqref{eqn:intro:mcf}
can be used to show that $\gamma(y,t) \to ct + \mu$
for some $\mu \in \R$. This can hence be interpreted
as a stability result for the planar waves \eqref{eqn:intro:PDE:wave:sol}
under a large class of non-localized perturbations. Note however
that no information is provided on the rate at which the convergence
takes place. Very recently - and simultaneously with our analysis here - Matano,
Mori and Nara generalized this approach to consider radially expanding surfaces in anistropic continuous media \cite{matano2019asymptotic}.

\paragraph{Mean curvature flow} 
In order to interpret the PDE \eqref{eqn:intro:mcf}, we consider
the interfacial graph
$G(t):=\left\{(\Gamma(y, t), y):y\in \R \right\}$. Writing $\nu(y,t)$ for the rightward-pointing normal vector,
$V(y,t)$ for the horizontal velocity vector
and $H(y,t)$ for the curvature
at the point $(\Gamma(y,t), y)$,
we obtain
\begin{equation}
    \nu = \big[ 1 + \Gamma_y^2 \big]^{-1/2} (1 , - \Gamma_y),
    \qquad \qquad
    V = (\Gamma_t, 0 \big),
    \qquad \qquad
    H = \big[ 1 + \Gamma_y^2 \big]^{-3/2} \Gamma_{y y}.
\end{equation}
In particular, \eqref{eqn:intro:mcf}
can be written in the form
\begin{equation}
  \label{eq:int:simpl:mean:curv:PDE}
    V\cdot \nu = H + c,
\end{equation}
which can be interpreted as a mean curvature flow with an additional normal drift of size $c$. It is no coincidence that this drift does not depend on $\nu$: it reflects the fact that the speed of
the planar waves \eqref{eqn:intro:PDE:wave:sol} does not depend on the angle $\theta$.

In a sense, it is not too surprising
that the mean curvature flow plays a role in the asymptotic dynamics of wave interfaces. Indeed, one of the main historical reasons for considering the Allen-Cahn PDE is that it actually desingularizes this flow by smoothing out the transition region \cite{ALLEN19791085,deckelnick2005computation}. However, from a technical point of view, its role in \cite{Matano} is actually rather minor.

Instead, the main PDE used to capture the behaviour of the phase $\gamma$ is 
the nonlinear heat equation
\begin{equation}\label{eqn:intro:nonlinear_eqn_for_V}
    V_t = V_{yy} + \dfrac{c}{2}V_y^2 + c. 
\end{equation}
This PDE can be reformulated as a standard linear heat equation by a Cole-Hopf transformation and hence explicitly solved.
These solutions can subsequently be used
to construct super- and sub-solutions 
to~\eqref{eqn:intro:PDE} of
the form
\begin{equation}\label{eqn:intro:supersols}
    u^\pm(x,y,t) = \Phi\left(\dfrac{x- V(y,t)}{\sqrt{1+V_y^2}} \pm q(t)\right) \pm p(t),
\end{equation}
in which $q$ and $p$ are small correction terms that allow spatially homogeneous perturbations at $t =0 $ to be traded
off for phase-shifts as $t \to \infty$.

Using the comparison principle, one can use the functions \eqref{eqn:intro:supersols} to show that the phase $\gamma$ can be approximated asymptotically by $V$.
A second comparison principle argument
subsequently shows that $V$ can be used to track the solution $\Gamma$ of~\eqref{eqn:intro:mcf}. It therefore
plays a crucial role as an intermediary to obtain the desired relation between $\gamma$ and $\Gamma$.

\paragraph{Spatially discrete travelling waves} Plugging the  travelling wave ansatz
\begin{equation}\label{eqn:intro:LDE:wave:sol}
u_{ij}(t) = \Phi(i\cos\theta + j\sin \theta - ct),
\qquad \qquad
\Phi(-\infty) = 0,
\qquad
\Phi(+\infty) = 1
\end{equation}
into the Allen-Cahn LDE~\eqref{eqn:intro:Allen_Cahn_LDE},
we obtain the functional differential equation of mixed type (MFDE)
\begin{equation}\label{eqn:intro:MFDE}
    -c\Phi'(\xi) = \Phi(\xi + \cos \theta) + \Phi(\xi - \cos \theta) + \Phi(\xi + \sin \theta) + \Phi(\xi - \sin \theta) - 4\Phi(\xi) + g\big(\Phi(\xi);a\big).
\end{equation}
The existence of such waves
$(\Phi_\theta, c_\theta)$
was first obtained
for the horizontal direction $\theta = 0$
\cite{Hankerson1993,zinner1992existence}
and subsequently generalized to arbitrary directions \cite{Mallet-Paret1999}.
This $\theta$-dependence is a direct consequence of the anistropy of the lattice,
which breaks the rotational symmetry of the PDE \eqref{eqn:intro:PDE}. 

A second important difference between
\eqref{eqn:intro:tw_ODE_cont} and \eqref{eqn:intro:MFDE} is 
that the character of the latter 
system depends crucially on the speed $c$, which depends uniquely but intricately on the parameters $(\theta, a)$.
When $c \neq 0$ the associated waveprofile
is unique up to translation and satisfies $\Phi' > 0$. When $c = 0$ however,
one loses the uniqueness and smoothness of waveprofiles. In addition, monotonic and non-monotonic profiles typically coexist. This behaviour is a direct consequence
of the fact that \eqref{eqn:intro:MFDE}
reduces to a difference equation, posed
on a discrete ($\tan \theta \in \mathbb{Q}$)
or dense ($\tan \theta \notin \mathbb{Q}$) subset of $\R$. The transition between these two regimes is a highly interesting and widely studied topic, focusing on themes such as propagation failure \cite{HJH2011, HMP10, keener1987propagation},  crystallographic pinning  \cite{MPCP, HMP10} and frictionless kink propagation \cite{barashenkov2005translationally, DMIKEVYOS2005}; see \cite{hupkestraveling} for an overview. 

For the remainder of the present paper we only consider
the case $c \neq 0$ and shift our attention to the stability properties of the associated waves. In one spatial dimension
Zinner obtained the first stability 
result \cite{zinner1991stability}, 
which was followed by the development
of a diverse set of tools exploiting
either the comparison principle \cite{ChenGuoWu2008},
monodromy operators \cite{CMPS98} or spatial-temporal Green's functions \cite{Beck2010nonlinear, HJHFHNINFRANGE}.
The first stability result in two spatial dimensions
was obtained in \cite{hoffman2015multi}
for waves travelling in arbitrary rational
($\tan \theta \in \mathbb{Q}$) directions. Taking $\theta = 0$ here for presentation purposes,
the authors consider initial
conditions of the form
\begin{equation}
\label{eq:int:init:cond:2d}
u_{i,j}(0) = \Phi(i) + v^0_{i,j}
\end{equation}
and show that $u$ converges algebraically
to the horizontal wave $\Phi( i -ct)$.
%
Here the initial perturbation $v^0$
is taken to be sufficiently small in $\ell^\infty\big(\Z ; \ell^1(\Z;\R)\big)$.
In particular, the perturbation $v^0$ is only required to be localized in the direction perpendicular to the wave propagation.  

The restriction $\tan \theta \in \mathbb{Q}$ was removed in the sequel
paper \cite{hoffman2017entire},
where the initial perturbation
$v^0$ in \eqref{eq:int:init:cond:2d}
can be of arbitrary size as long as it is localized in the sense that
\begin{equation}
\label{eq:int:localization:v:zero}
    \lim_{|i| + |j| \to \infty} |v^0_{i,j}| = 0.
\end{equation}
The proof relies on the construction of explicit sub- and super-solutions to the LDE \eqref{eqn:intro:Allen_Cahn_LDE}, 
generalizing the PDE constructions from \cite{BHM}. This construction is especially delicate for the cases $\theta \notin  \frac{\pi}{4} \mathbb{Z}$, where the disalignment with the lattice directions causes slowly decaying modes that need to be carefully controlled.

\paragraph{Curved LDE fronts} 
In order to avoid the problematic slowly decaying terms discussed above,
we restrict ourselves to 
the horizontal 
waves \eqref{eq:int:trv:planar:wave:lde}
throughout the remainder of the paper.
The novelty is that we allow general bounded
initial conditions that satisfy the
limits
\eqref{eq:int:limits:init:cond:lde}. 
To compare this with the discussion above, we note that this class includes initial conditions
of the form
\begin{equation}
\label{eq:int:init:cond:kappa:v0:decomp}
    u_{i,j}(0) = \Phi\big(i - \kappa_j \big) + v^0_{i,j},
\end{equation}
in which $\kappa$ is an arbitrary bounded
sequence and $v^0$ is allowed to be
small in $\ell^\infty\big(\Z; \ell^1(\Z;\R)\big)$ or to satisfy the localization condition \eqref{eq:int:localization:v:zero}. In particular, we significantly expand the set of initial conditions that were
considered in \cite{hoffman2015multi,hoffman2017entire}.

Our main aim is to follow the program of \cite{Matano} that we outlined above as closely as possible. However, the first obstacle already arises when one attempts to define appropriate phase coordinates $\gamma_j(t)$ for $t \gg 1$. Indeed,
it no longer makes sense to define the interface of $u(t)$ as the set of points where $u_{i,j}(t) = \Phi(0)$, since this solution set can behave highly erratically due to the discreteness of the spatial variables. To resolve this, we establish an asymptotic monotonicity result in the interfacial region where $u_{i,j}(t) \approx \Phi(0)$. This allows us to `fill' the troublesome gaps between lattice points by performing a spatial interpolation based on the shape of $\Phi$; see Fig.~\ref{fig:intro:phase_gamma}.
\begin{figure}[t]
\centering
  \includegraphics[width=\linewidth]{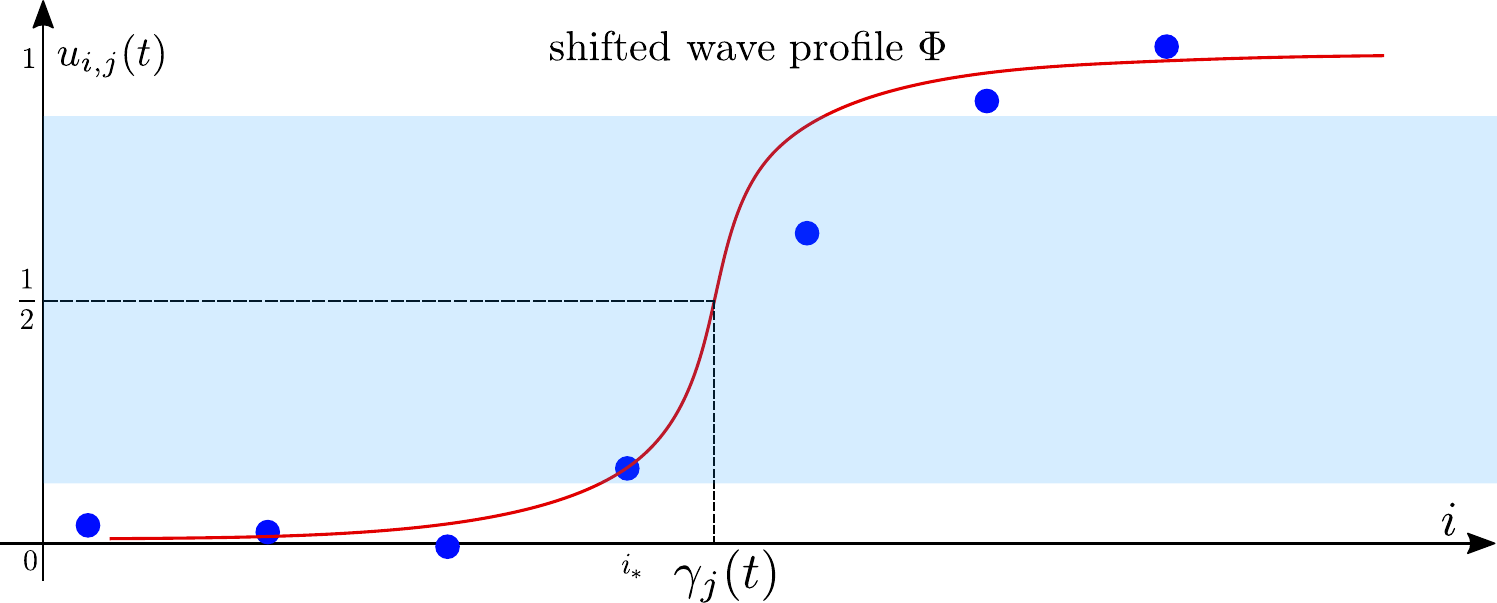}
  \caption{
  In~{\S \ref{sec:large_time_behaviour}} we show that for each $j\in \Z$ and $t\gg 0$, the function $i\mapsto u_{i,j}(t)$ is monotonic inside an interfacial region $I$ that is depicted in light blue. The dark blue dots represent the horizontal solution slice $i\mapsto u_{i,j}(t)$. Since $u$ is monotonic inside $I$, we can find an unique value $i_*$  for which $ u_{i_*, j}(t) \leq 1/2 < u_{i_*+1, j}(t)$.  We subsequently shift the travelling wave profile $\Phi$ in such a way that it matches the solution slice at $i_*$. The  phase $\gamma_j(t)$ is then defined as the argument where this shifted profile equals one half. }
  \label{fig:intro:phase_gamma}
\end{figure}

This fundamental problem
of not being able to move continuously
between lattice points occurs in many other parts of our analysis. For example,
we need to construct so-called $\omega$-limit points of solution sequences in order to establish the uniform
convergence \eqref{eq:int:unif:cnv:to:phase:front}.
In \cite{Matano} this is achieved by
passing to a new coordinate $x' = x - ct$ that `freezes' the wave at the cost of an extra convective term in the PDE \eqref{eqn:intro:PDE}. Such a coordinate transformation does not exist in the discrete case, forcing us to use a more involved discontinuous version of this freezing process.
\paragraph{Discrete curvature flow}
We remark that it is by no means a-priori clear how the mean curvature PDE
\eqref{eqn:intro:mcf} should be discretized in order to track the discrete phase coordinates $\gamma_j(t)$. For example, there is more than one reasonable way 
to define geometric notions such as normal vectors and curvature in discrete settings \cite{crane2018discrete}. On the other hand, the discussion above shows that 
there may be range of `suitable' choices,
as we only desire the tracking to be approximate.

Introducing the convenient notation
\begin{equation}
    [\beta_{\Gamma}]_j = \sqrt{1+\dfrac{(\Gamma_{j+1} - \Gamma_j)^2 + (\Gamma_{j-1} - \Gamma_j)^2 }{2} },
    \qquad\qquad
    [\partial^{(2)} \Gamma]_j
    = \Gamma_{j+1} + \Gamma_{j-1} - 2 \Gamma_j,
\end{equation}
we will use the standard symmetric discretizations
\begin{equation}
    V \cdot \nu \mapsto \beta_{\Gamma}^{-1} \dot{\Gamma},
    \qquad \qquad 
    H \mapsto   \beta_{\Gamma}^{-3} \partial^{(2)} \Gamma
\end{equation}
for the normal velocity and curvature terms in \eqref{eq:int:simpl:mean:curv:PDE}. However, the remaining normal drift term requires more care to account for the direction dependence of the planar front speeds. In particular, it seems natural make the replacement
\begin{equation}\label{eqn:intro:cpm}
    c \mapsto \frac{1}{2} \big( c_{\theta^+} + c_{\theta^-} \big),
\end{equation}
in which 
the angles 
\begin{equation}\label{eqn:intro:theta_pm}
    \theta^- = \arctan{(\Gamma_j - \Gamma_{j-1})},
    \qquad \qquad
    \theta^+ = \arctan{(\Gamma_{j+1} - \Gamma_{j})}
\end{equation}
measure the orientation of the normal vectors
for the lower and upper segments of the interface at $(\Gamma_j, j)$;
see Fig. \ref{fig:intro:drift_term}.
\begin{figure}[t]
\centering
  \includegraphics[width=\linewidth]{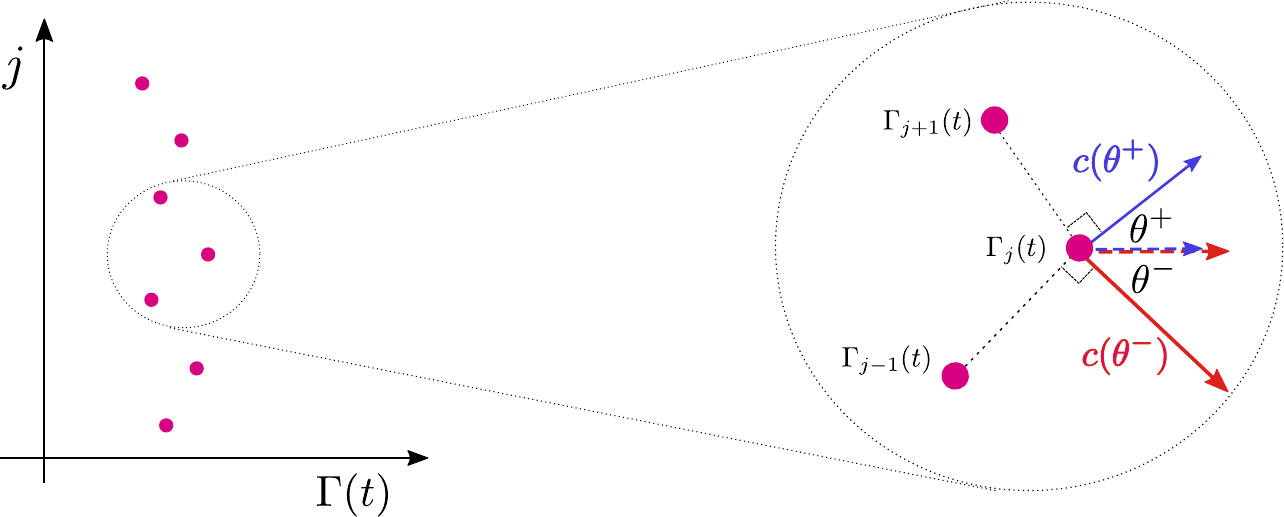}
  \caption{The panel on the left represents a graph $j\mapsto \Gamma_j(t)$ at a fixed time $t$. The right panel zooms in on three nodes
  of this graph to illustrate the identities~\eqref{eqn:intro:cpm} and~\eqref{eqn:intro:theta_pm} that underpin the drift term in our discrete curvature flow. }
  \label{fig:intro:drift_term}
\end{figure}

In order to make this more explicit, we use the
identity $[\partial_\theta c_{\theta}]_{\theta=0} = 0$
derived in \cite[Lem. 2.2]{hupkes2019travelling}
to obtain the expansions
\begin{equation}
    c_{\theta_-} \sim c+ \dfrac{1}{2}[\partial_\theta^2 c_\theta]_{\theta=0} (\Gamma_j - \Gamma_{j-1})^2,
    \qquad \qquad 
    c_{\theta_+} \sim 
    c+ \dfrac{1}{2}[\partial_\theta^2 c_\theta]_{\theta=0} (\Gamma_{j+1} - \Gamma_{j})^2 ,
\end{equation}
which suggests the replacement
\begin{equation}
\label{eq:int:drift:ajd:ii}
    c \mapsto c + \dfrac{1}{2}[\partial_\theta^2 c_\theta]_{\theta=0}
    \big( \beta_{\Gamma}^2 - 1 \big).
\end{equation}
In order to prevent the quadratic growth in this term, we make the final adjustment
\begin{equation}
\label{eq:int:drift:ajd:iii}
    c \mapsto c + [\partial_\theta^2 c_\theta]_{\theta=0} \big( 1 - \beta_{\Gamma}^{-1}  \big),
\end{equation}
which agrees with \eqref{eq:int:drift:ajd:ii} up to second order
in the differences $\Gamma_{j \pm 1} - \Gamma_j$. 

All in all,
the discrete mean curvature flow that we use in this paper
to approximate the phases $\gamma_j$ can be written as
\begin{equation}
\label{eq:int:curv:flow:discrete}
    \beta_{\Gamma}^{-1} \dot{\Gamma}
    =   \beta_{\Gamma}^{-3} \partial^{(2)} \Gamma
    + c + [\partial_\theta^2 c_\theta]_{\theta=0} \big( 1 - \beta_{\Gamma}^{-1}  \big).
\end{equation}
While this justification appears to be rather ad-hoc, it turns out that
our approximation procedure is not sensitive to 
$O\big( ( \Gamma_{j \pm 1} - \Gamma_j )^3 \big)$-correction terms.
In addition, we explain below how the crucial lower order terms can be recovered by independent technical considerations.

\paragraph{Super- and sub-solutions}
The technical heart of this paper is formed by our construction of 
suitable spatially discrete versions of the sub- and super-solutions
\eqref{eqn:intro:supersols}. The correct generalization of
\eqref{eqn:intro:nonlinear_eqn_for_V} that preserves the Cole-Hopf structure
turns out to be
\begin{equation}\label{eqn:intro:nonlinear_V_exp}
    \dot V_j = \dfrac{1}{d}\left(e^{d(V_{j+1}- V_j)} - 2 + e^{d(V_{j-1} - V_j)}\right) + c,
\end{equation}
in which we are still free to pick the coefficient $d$. Indeed,
this LDE reduces to the discrete heat equation upon picking
$h(t) = e^{d(V-ct)}$.

However, the discrete Laplacian
spawns terms proportional to $\Phi'' (\beta_V^2 - 1)$
if one simply 
substitutes a direct discretization of the PDE super-solution
\eqref{eqn:intro:supersols} with \eqref{eqn:intro:nonlinear_V_exp}
into \eqref{eqn:intro:Allen_Cahn_LDE}.
These terms decay as $O(t^{-1})$ and hence cannot be integrated 
and absorbed into the phaseshift $q(t)$.

Similar difficulties were also encountered in \cite{hoffman2017entire}. The novelty here is that this troublesome behaviour occurs even for the horizontal direction $\theta = 0$, which is completely aligned with the lattice.
Inspired by the normal form approach developed in \cite{hoffman2017entire},
we therefore set out to construct sub- and super-solutions of the form
\begin{equation}
    u^{\pm}_{i,j}(t) = \Phi\big(i-V_j(t) \pm q(t)\big)  +  r\big(i-V_j(t) \pm q(t)\big)([\beta_V]_j^2 - 1) \pm p(t),
\end{equation}
using the extra residual function $r$ to neutralize the slowly decaying terms. Working through the computations, it turns out the relevant condition on 
the pair $(r,d)$ can be formulated as
\begin{equation}
    \mathcal{L}_{\mathrm{tw}} r + d\Phi' = -\Phi'',
\end{equation}
in which the Fredholm operator $\mathcal{L}_{\mathrm{tw}}$ encodes the linearization of the wave MFDE \eqref{eqn:intro:MFDE} around $\Phi$;
see {\S}\ref{sec:sub:sup}. Using the Fredholm theory for MFDEs developed in \cite{Mallet-Paret1999_Fredholm,Mallet-Paret1999} together with the computations in {\S}\ref{sec:asymp} and \cite[{\S}2]{hupkes2019travelling}, it turns out that $d$ must be given by
\begin{equation}
    d = \frac{1}{2}c + \frac{1}{2}[\partial_\theta^2 c_\theta]_{\theta=0} = \frac{1}{2}\big[\partial^2_\theta \mathcal{D}(\theta) \big]_{\theta = 0},
\end{equation}
in which the quantity
\begin{equation}
    \mathcal{D}(\theta) = \frac{c_{\theta}}{\cos \theta}
\end{equation}
is referred to as the directional dispersion. This quantity measures the horizontal speed of waves travelling in the direction $\theta$,
which also plays an important role in the construction of travelling corner solutions to \eqref{eqn:intro:Allen_Cahn_LDE}.

\begin{figure}[t]
    \centering
        \begin{subfigure}{0.5\textwidth}
\includegraphics[width=1\columnwidth]{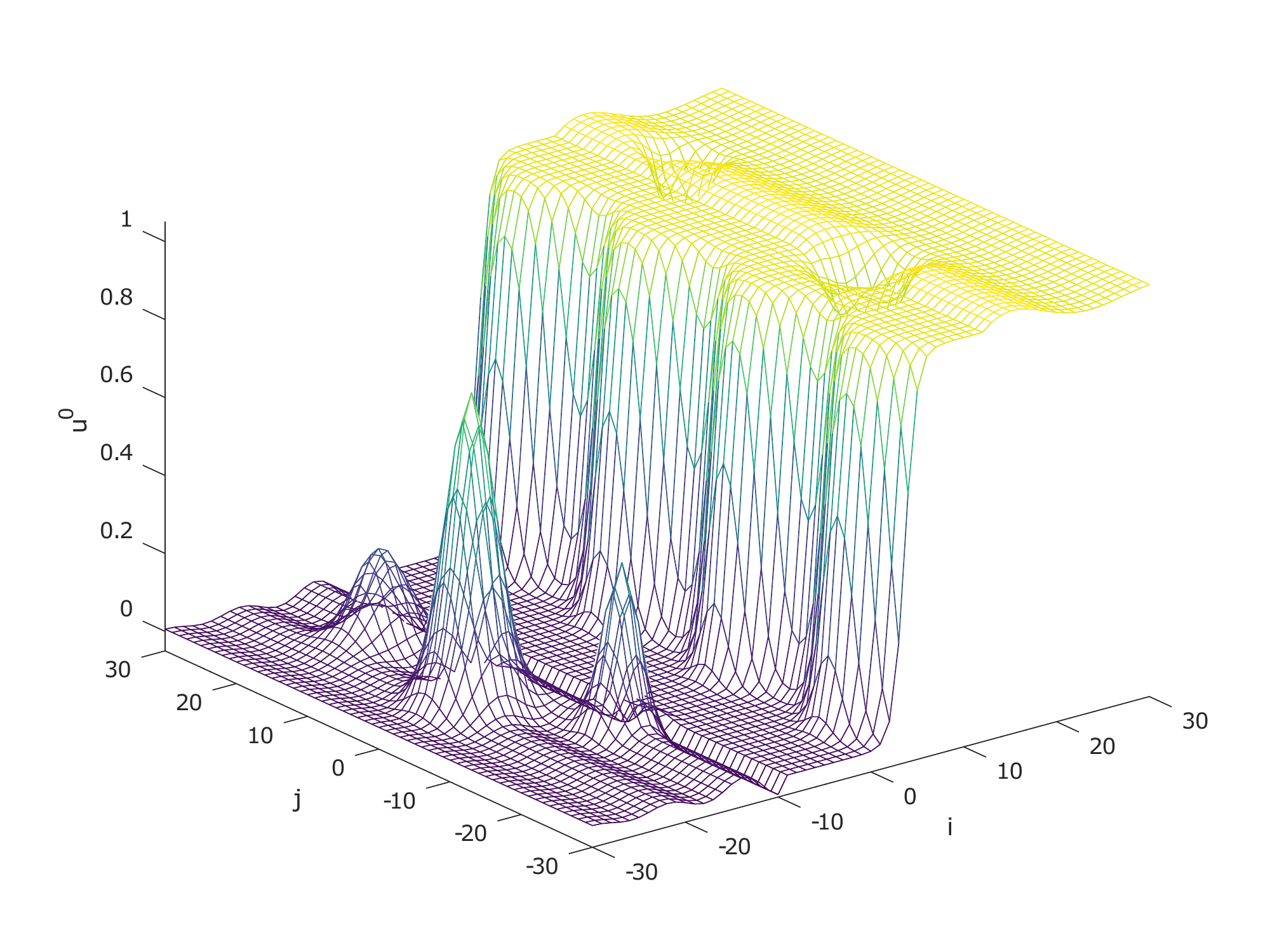}
        \caption{}
            \label{fig:intro:initial_conditions:a}
    \end{subfigure}%
    \begin{subfigure}{0.5\textwidth}
\includegraphics[width=1\columnwidth]{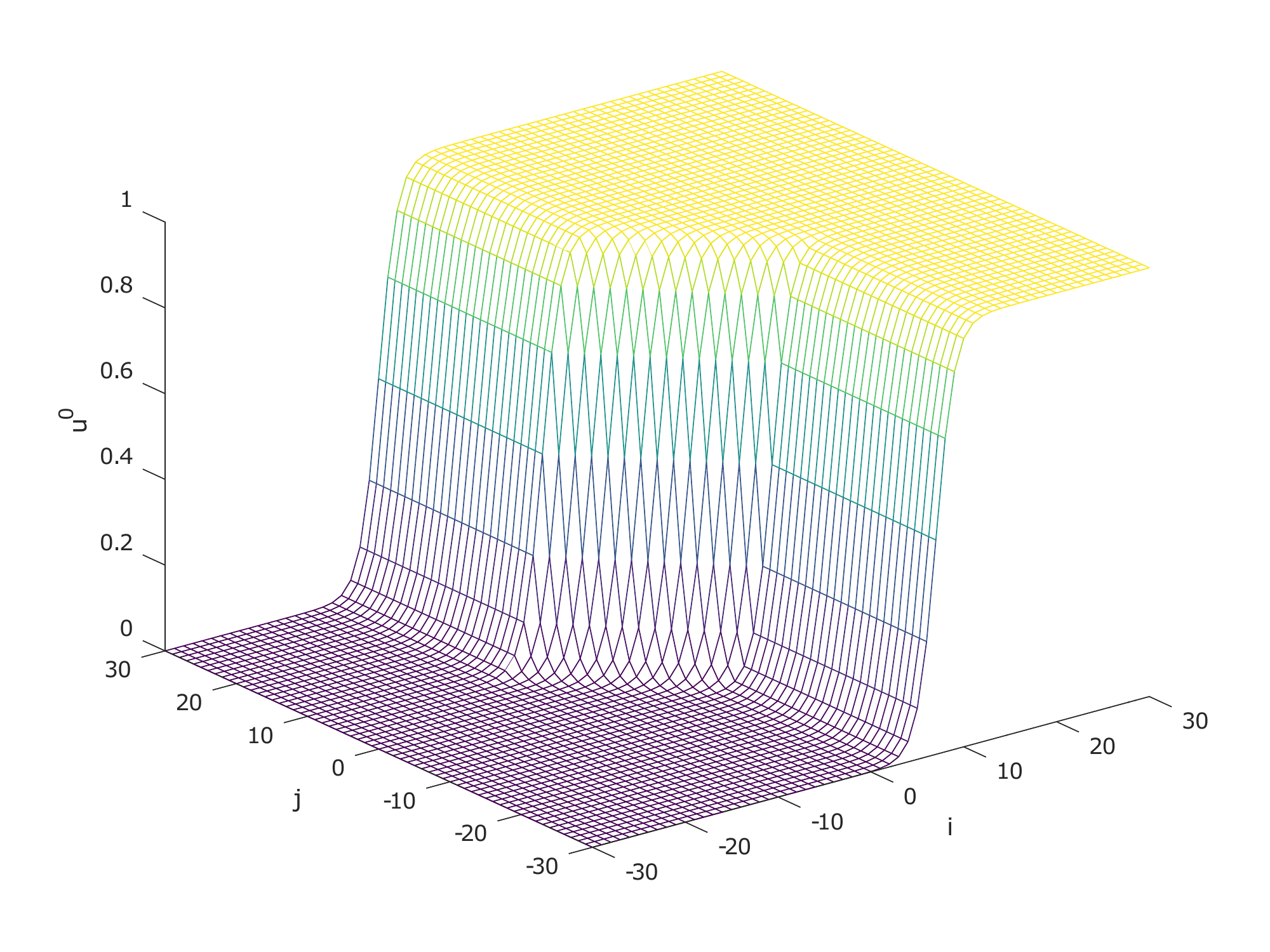}
        \caption{}
            \label{fig:intro:initial_conditions:b}    
    \end{subfigure}
    \caption{
    Both panels illustrate front-like initial conditions
  that satisfy 
  \eqref{eq:int:limits:init:cond:lde} and hence fall within the framework of this paper.  Panel a) provides an example of an initial perturbation that  converges uniformly to a traveling front. On the contrary, the initial perturbation in b) does not uniformly converge to a traveling planar front, but the evolution of the interface is described asymptotically by \eqref{eq:int:curv:flow:discrete}.
    }
    \label{fig:intro:initial_conditions}
\end{figure}


\paragraph{Stability results}
As a by-product of our analysis, we are able to extend
the stability results obtained previously in \cite{hoffman2015multi,hoffman2017entire}. For example, if the phase
sequence $\kappa$ appearing in the initial condition \eqref{eq:int:init:cond:kappa:v0:decomp}
is periodic (see e.g. 
Fig.~\ref{fig:intro:initial_conditions:a}),
we show that there exists an asymptotic phase $\mu \in \R$ for which we have the convergence $\gamma(t) \to ct+\mu$ as $t \to \infty$. In particular, the horizontal planar wave retains its stability under such perturbations, provided we allow for a phase-shift. 

In order to prove this result,
we first analyze the behaviour of 
\eqref{eqn:intro:Allen_Cahn_LDE}
and
\eqref{eqn:intro:nonlinear_V_exp}
when applied to $j$-periodic sequences. We subsequently add a localized initial perturbation and show that the effects remain localized in some sense. Since the heat-equation eventually eliminates such localized perturbations, the desired asymptotic convergence persists.
We remark that our stability result is slightly less general than its continuous counterpart from \cite{Matano}, since it is not yet clear to us how ergodicity properties can be transferred to our discrete setting.

We emphasize that this stability result does not hold for arbitrary bounded $\kappa$ in \eqref{eq:int:init:cond:kappa:v0:decomp}. For example, if there exist $\kappa^-$ and $\kappa^+$ for which we have the limits
\begin{equation}
    \lim_{j \to - \infty} \kappa_j = \kappa^-, \qquad \qquad
    \lim_{j \to + \infty} \kappa_j = \kappa^+
\end{equation}
(see e.g. Fig.~\ref{fig:intro:initial_conditions:b}), 
then the results in {\S}\ref{sec:per}
imply that for every $t >0$ we have the convergence
\begin{equation}
\label{eq:int:kappa:pm:limits}
    u_{i,j}(t) \to  \Phi(i-ct - \kappa^{\pm}) \quad \  \text{as}\ j\to\pm\infty,
\end{equation}
uniformly in $i$.
In particular, the interface $\gamma(t)$ describes the phase transition between $\kappa^-$ and $\kappa^+$, which is asymptotically captured by
\eqref{eq:int:curv:flow:discrete}.

\paragraph{Organization}
After formulating our assumptions and main results in {\S}\ref{sec:mr}, we transfer the standard $\omega$-limit point constructions for the PDE \eqref{eqn:intro:PDE} to our discrete setting in {\S}\ref{sec:omega}. In {\S}\ref{sec:Trapped entire solutions}
we (partially) generalize the results from \cite{berestycki2007generalized} concerning trapped
entire solutions to the setting of \eqref{eqn:intro:Allen_Cahn_LDE}. In particular, we prove that every entire solution of the Allen-Cahn LDE trapped between two traveling waves is a traveling wave itself. 
In {\S}\ref{sec:large_time_behaviour} we focus on the large-time behaviour of the solution $u$ and establish
the discrete counterpart of \eqref{eqn:intro:conv_gamma_PDE}.
We move on in {\S}\ref{sec:dht} to obtain decay estimates for 
discrete gradients of solutions to the discrete heat equation. We exploit these in {\S}\ref{sec:sub:sup} to construct super- and sub-solutions, which we use in~{\S}\ref{sec:asymp} to approximate the phase $\gamma$ with the solution of 
the  discrete mean curvature flow
\eqref{eq:int:curv:flow:discrete}.
Finally, in {\S}\ref{sec:per} we establish the stability results discussed above for the horizontal planar travelling wave. 

\paragraph{Acknowledgments}
Both authors acknowledge support from the Netherlands Organization for Scientific Research (NWO) (grant 639.032.612).

\section{Main results}
\label{sec:mr}

Our principal interest in this paper is the 
discrete  Allen-Cahn equation 
\begin{align}
\label{eqn:main_results:discrete AC}
&\dot{u}_{i,j} =  (\Delta^+u)_{i,j} +  g(u_{i,j}) 
\end{align}
posed on the planar lattice
$(i,j) \in \Z^2$.
The discrete Laplacian $\Delta^+:\ell^\infty(\Z^2)\to \ell^\infty(\Z^2)$ is defined as
\begin{equation}
 (\Delta^+u)_{i,j}=   u_{i+1,j} + u_{i,j+1} + u_{i-1, j} + u_{i,j-1} - 4u_{i,j},
\end{equation}
while the nonlinearity is 
assumed to satisfy the following
bistability condition.

\begin{itemize}
    \item[(Hg)]
 The nonlinear function $g:\R\to\R$ is $C^2$-smooth and there exists $a\in (0,1)$ for which we have
\begin{equation}\label{assumption:main_results:conditions on g:a}
g(0)=g(a) = g(1)=0, \qquad g'(0) < 0, \qquad g'(1) <0. 
\end{equation}
In addition, we have the inequalities
\begin{equation}
\label{assumption:main_results:conditions on g:b}
g> 0 \text{ on } (-\infty, 0)\cup (a, 1),
\qquad \qquad
g< 0 \text{ on } (0,a) \cup (1, \infty).
\end{equation}
\end{itemize}

Existence results for planar traveling wave solutions of \eqref{eqn:main_results:discrete AC}
were established in \cite{Mallet-Paret1999}. More precisely, if we pick an arbitrary angle $\theta \in [0, 2\pi)$,
then \eqref{eqn:main_results:discrete AC}
admits a solution
of the form
\begin{equation}
\label{eq:mr:wave:ansatz}
    u_{ij}(t) = \Phi_\theta(i \cos \theta + j \sin \theta - c_\theta t),
\end{equation}
for some wave speed $c_\theta \in \R$
and wave profile $\Phi_\theta : \R \to \R$
that satisfies 
the boundary conditions
\begin{equation}
\label{eqn:main_results:boundary:cond:for:phi}
\lim_{\xi\to-\infty} \Phi_\theta(\xi) = 0, \quad \lim_{\xi\to+\infty} \Phi_\theta(\xi) = 1.
\end{equation}
Substituting the Ansatz
\eqref{eq:mr:wave:ansatz}
into \eqref{eqn:main_results:discrete AC},
we see that the 
the pair $(\Phi_\theta,c_\theta)$ must satisfy
the MFDE
\begin{equation}\label{eqn:main_results:MFDE}
    - c_\theta \Phi_\theta'(\xi) = 
    \Phi_\theta(\xi + \cos \theta ) + \Phi_\theta(\xi + \sin \theta ) +\Phi_\theta(\xi - \cos \theta ) + \Phi_\theta(\xi - \sin \theta )- 4 \Phi_\theta(\xi) + g\big(\Phi_\theta(\xi) 
    \big).
\end{equation}

The results in \cite{Mallet-Paret1999}
state that $c_\theta$ is unique.
In addition, when $c_\theta \neq 0$,
the wave profile $\Phi_\theta$ is unique up to translation and satisfies $\Phi'_\theta > 0$.
In this paper, we are interested
in planar waves that travel in the horizontal direction $\theta =0$.
Since we rely on smoothness properties of the wave profile, we demand 
that $c_0 \neq 0$.

\begin{itemize}
\item[(H$\Phi$)] There exists 
a non-zero speed $c\neq 0$ and a wave profile $\Phi\in C^1(\R, \R)$ 
so that the  pair $(\Phi,c)$
satisfies the
boundary conditions
\eqref{eqn:main_results:boundary:cond:for:phi} and the
MFDE \eqref{eqn:main_results:MFDE} 
for the horizontal direction $\theta = 0$. 
In addition, we have the normalization
$\Phi(0) = \frac{1}{2}$.
\end{itemize}

Our main results concern the
Cauchy problem for the Allen-Cahn LDE. In particular, we look for 
functions
\begin{equation}
\label{eq:mr:def:sol:u}
    u \in C^1\big([0, \infty); \ell^\infty(\Z^2) \big)
\end{equation}
that satisfy
the LDE \eqref{eqn:main_results:discrete AC} 
for $t > 0$ together with the initial
condition
\begin{align}
&u_{i,j}(0) = u^0_{i,j} \label{eqn:main_results:initial condition}
\end{align}
for some $u^0 \in \ell^\infty(\Z^2)$.
Observe that the comparison principle together with
the bistable structure of $g$ imply
that such solutions are unique
and exist globally. We impose
the following structural condition
on $u^0$.

\begin{itemize}
\item[(H0)] The initial condition 
$u^0 \in \ell^\infty(\Z^2)$ satisfies
the inequalities
\begin{equation}\label{assumption:main_reusults:condition on u0}
    \limsup_{i\to-\infty} \sup_{j\in\Z} u^0_{i,j} < a, 
    \quad \qquad
    \liminf_{i\to\infty} \inf_{j\in\Z} u^0_{i,j} > a .
\end{equation}
\end{itemize}
Notice that we do not impose the usual assumption $0\leq u^0\leq 1$
or any kind of decay
in the spatial limits. As explained in detail in {\S}\ref{sec:int},
this condition is less restrictive than its counterparts
from \cite{hoffman2015multi,hoffman2017entire} and includes the
general class \eqref{eq:int:init:cond:kappa:v0:decomp}.


\subsection{Interface formation}

Our first goal is to 
find a link between the solution \eqref{eq:mr:def:sol:u}
of the general Cauchy problem
for \eqref{eqn:main_results:discrete AC} 
and the planar travelling
wave $(\Phi, c)$. The result below
provides a key tool for this purpose
when $t \gg 1$.
In particular, it establishes
that for each fixed $j \in \Z$,
the horizontal slice $i \mapsto u_{ij}(t)$ `crosses through' the value $u =\frac{1}{2}$ in a monotonic fashion. 

\begin{prop}[see {\S}\ref{sec:large_time_behaviour}]
\label{prp:main_results_existence of gamma}
There exists a time $T>0$ such that for every $j\in\Z$ and $t\geq T$ there exists a unique $i_*=i_*(j,t)$ with the property 
\begin{equation} \label{eqn:main_results:i(j,t))}
0 < u_{i_*,j}(t) \leq \dfrac{1}{2}, \quad u_{i_*+1, j}(t) >    \dfrac{1}{2}. \end{equation}
\end{prop}

These functions $i(j,t)$
can be used to define
a set of phases $\big(\gamma_j(t) \big)_{j \in \Z}$ that measure
in some sense where the value
$u=\frac{1}{2}$ is `crossed'.
More precisely,
we define a function $ \gamma:[T, \infty)\to \ell^\infty(\Z) $ 
that acts as
\begin{equation}\label{eqn:main_results:def_of_gamma} 
\gamma_j(t) = i_*(j,t) - \Phi^{-1}\big(u_{i_*(j,t),j}(t)\big);
\end{equation}
see Fig.~\ref{fig:intro:phase_gamma}.
The motivation behind the second
term on the right is our desire to recover
the traditional phase when $u$ is itself a travelling wave. Indeed, in the special
case that
\begin{equation}
u_{i,j}(t) = \Phi(i-ct-\mu)
\end{equation}
for some $\mu \in \R$, the
phase condition $\Phi(0) = \frac{1}{2}$
implies that 
\begin{equation}
    i_*(j,t) = \lfloor ct + \mu \rfloor.
\end{equation}
In particular, we obtain
\begin{equation}
    \gamma_j(t) =  ct+\mu,
\end{equation}
which allows us to write
\begin{equation}
\label{eq:mr:char:trv:wave:ex}
    u_{i,j}(t) = \Phi\big( i - \gamma_j(t) \big).
\end{equation}
The drawback of this relatively straightforward construction is that the phases $\gamma_j(t)$ will in general admit discontinuities. However, the size of these jumps will tend to zero as $t \to \infty$, which suffices for our asymptotic purposes.

Our main result here is that this 
phase description \eqref{eq:mr:char:trv:wave:ex}
holds asymptotically for any initial condition $u^0$ that satisfies (H0).
In particular, for large time, the dynamics of the full solution $u$ can be 
approximated by the behaviour
of the phase coordinates $\gamma(t)$.

 \begin{thm}[{see \S\ref{sec:large_time_behaviour}}]\label{thm:main results:gamma_approximates_u}
    Suppose that (Hg), $(H\Phi)$ and (H0)
    are satisfied and consider the
    solution $u$ of the discrete Allen-Cahn equation~\eqref{eqn:main_results:discrete AC} with the initial condition~\eqref{eqn:main_results:initial condition}. Then we have the limit
	\begin{equation}\label{eqn:main_results:front}
	\lim_{t\to\infty} \sup_{(i,j) \in \Z^2} \big|u_{i,j}(t) - \Phi\big(i-\gamma_j(t)\big)\big|=0. 
	\end{equation}
\end{thm}
 
\subsection{Interface evolution}

Our second main goal is to
uncover the long-term dynamics
of the phase $\gamma$
defined in \eqref{eqn:main_results:def_of_gamma}. In particular, we show
that this evolution can be approximated
by a discrete version of the 
mean curvature flow
with an appropriate drift term.

In order to formulate this equation,
we pick a sequence
$\Gamma \in \ell^\infty(\Z)$
and introduce the discrete derivatives
\begin{equation}
\begin{array}{lcl}
 [\partial^+ \Gamma]_j &=& \Gamma_{j+1} - \Gamma_j, \\[0.2cm]
[\partial^- \Gamma]_j &=& \Gamma_{j} - \Gamma_{j-1}, \\[0.2cm]
[\partial^{(2)}\Gamma]_j & =& \Gamma_{j+1} - 2\Gamma_j + \Gamma_{j-1}, 
\end{array}
\end{equation}
together with the sequence
\begin{equation}
    [\beta_\Gamma]_j = \sqrt{1+ \dfrac{1}{2} (\partial^+ \Gamma)^2_j + \dfrac{1}{2} (\partial^- \Gamma)^2_j}.
\end{equation}
As explained in {\S}\ref{sec:int}, the driving
force 
in \eqref{eq:mr:discrete:MCF} below
is not a constant as in the PDE case. Instead, it features
additional terms that arise due to the underlying anisotropy of the lattice.

\begin{thm}[{see {\S}\ref{sec:asymp}}]
\label{thm:main_results:approx_of_gamma} Suppose that (Hg), $(H\Phi)$
and (H0) are all satisfied,
consider the solution
$u$ of the LDE
\eqref{eqn:main_results:discrete AC} with the initial condition~\eqref{eqn:main_results:initial condition}
and recall the phase $\gamma$
defined in  \eqref{eqn:main_results:def_of_gamma}.
Then for every 
$\epsilon >0$, there exists $\tau_\epsilon \geq T $ so that for any $\tau \ge \tau_{\epsilon}$,
the solution
\begin{equation}
\Gamma: [\tau, \infty)
\to \ell^\infty(\Z)
\end{equation}
to the initial value problem
\begin{equation}
\label{eq:mr:discrete:MCF}
\left\{
\begin{array}{lcl}
     \beta_{\Gamma}^{-1} \dot{\Gamma} 
     &=& \beta_{\Gamma}^{-3} \partial^{(2)} \Gamma   +  \left(c +[\partial_\theta^2 c_\theta]_{\theta=0}\right) -  \beta_{\Gamma}^{-1}  [\partial_\theta^2 c_\theta]_{\theta =0} 
     \\[0.2cm]
     
      \Gamma(\tau)  &=& \gamma(\tau)
\end{array}
\right.
\end{equation}
 satisfies the estimate
 \begin{equation}
     \sup_{t\geq \tau} \norm{\Gamma(t) - \gamma(t) }_{\ell^\infty} < \epsilon.  
 \end{equation}
 \end{thm}

 Our final result provides more detailed information
 on the asymptotics of $\gamma$ in the special case
 that the initial condition $u^0$ is a localized perturbation
 from a front-like background state that is periodic in $j$. Indeed,
 this provides sufficient control on \eqref{eq:mr:discrete:MCF}
 to show
 that the corresponding solution converges to a planar travelling front. We emphasize that the case $P =1$ encompasses the stability results
 from \cite{hoffman2015multi,hoffman2017entire}, 
 albeit only for horizontal waves.

 \begin{thm}[{see {\S}\ref{sec:per}}]
\label{thm:mr:periodicity+decay:stability}
Suppose that (Hg), (H$\Phi$) and (H0) are
satisfied and consider the solution $u$ of the discrete
Allen-Cahn equation \eqref{eqn:main_results:discrete AC}
with the initial condition \eqref{eqn:main_results:initial condition}. Suppose furthermore that there exists a 
sequence $u^{0;\mathrm{per}} \in \ell^\infty(\Z^2)$
so that the following two properties hold.
\begin{itemize}
    \item[(a)]{
      We have the limit
      \begin{equation}
      \label{eq:mr:lim:per:j:pm:infty}
    u^0_{i,j} - u^{0;\mathrm{per}}_{i,j} \to 0, \quad \text{as}\quad  |i|+|j|\to\infty .
\end{equation}
    }
    \item[(b)]{
      There exists an integer $P \ge 1$ so that
      \begin{equation}
          u^{0;\mathrm{per}}_{i, j+P}
          =  u^{0;\mathrm{per}}_{i, j},
          \qquad \qquad
          \hbox{ for all }  (i,j) \in \Z^2.
      \end{equation}
     }
\end{itemize}
Then there exists a constant $\mu\in \R$ for which
we have the limit
\begin{equation}
  \label{eqn:periodicity+decay:stability}
    \lim_{t\to\infty} \sup_{(i,j)\in \Z^2} |u_{i,j}(t) - \Phi(i-ct-\mu)| = 0. 
\end{equation}
\end{thm}

\section{Omega limit points}
\label{sec:omega}

The techniques used in \cite{Matano}
relied heavily upon the ability
to construct so-called omega limit points.
More specifically, consider a
solution $u: \R^2 \times [0, \infty) \to \R$ to the PDE \eqref{eqn:intro:PDE}
together with an unbounded
sequence $0 < t_1 < t_2 < \ldots $
and a set of vertical shifts
$(y_k) \subset \R$.
One can then establish \cite{Matano}
the existence of an
entire solution $\omega$
to \eqref{eqn:intro:PDE}
for which the convergence
\begin{equation}\label{eqn:omega limit points:def_in_cont_case}
u(x + c  t_k, y+y_k, t+ t_k) \to \omega(x ,y,t) \quad \text{ in } C_\text{loc}^{2,1}(\R^2\times\R)
\end{equation}
holds as $k \to \infty$,
possibly after passing to a subsequence. This 
can be achieved efficiently by
replacing $x$ with 
the travelling wave coordinate $x - ct$.

Any direct attempt to generalize this
procedure to the LDE setting
will fail on account of the
fact that $i-ct$ is not necessarily an integer. Indeed, this prevents
us from introducing a well-defined
co-moving frame. Our approach
here to handle this is rather crude:
we simply round the horizontal
shifts upward towards the nearest
integer. 

To illustrate this,
let us consider the planar wave
solution
\begin{equation}
    u_{ij}(t) = \Phi(i-ct)
\end{equation}
together with an unbounded sequence
$0 < t_1 < t_2 < \ldots$ and
a set of vertical
shifts $(j_k) \subset \Z$.
Possibly taking a subsequence,
we obtain the convergence
\begin{equation}
    [0,1] \ni \lceil ct_k \rceil - ct_k \to \theta_{\omega}
\end{equation}
as $k \to \infty$, which means that
\begin{equation}
    u_{i+\lceil ct_k \rceil, j+j_k}(t+t_k)
    = 
    \Phi(i + \lceil ct_k \rceil -ct -ct_k   )
    \to \Phi(i-ct + \theta_{\omega})  
\end{equation}
as $k\to \infty$.  
In particular, we do still
recover an entire solution,
at the price of a small phase-shift
that would not occur in the continuous
framework.
As we will see throughout the following sections, this phase-shift does not cause any qualitative difficulties. 

Our main result confirms that our
procedure indeed generates
$\omega$-limit points. In addition,
it states
that such limits are trapped
between two travelling waves,
which turns out to be a crucial
point in our analysis. The consequences
of this fact 
will be discussed in greater depth in 
{\S}\ref{sec:Trapped entire solutions}.
\begin{prop}
\label{lemma:omega_limit_points:construction of limit points}
Suppose that (Hg), $(H\Phi)$ and (H0) are satisfied. Let $ u \in C^1\big([0,\infty); \ell^\infty(\Z^2)\big)$ be a solution of the LDE \eqref{eqn:main_results:discrete AC}. Then for any sequence $ (j_k, t_k) $ in $ \Z\times [0, \infty ) $  with $ 0<t_1<t_2  < \dots \to \infty $, there exists a subsequence $ (j_{n_k}, t_{n_k}) $ and a function $ \omega \in C^1\big(\R; \ell^\infty(\Z^2)\big)$ with the following properties. \begin{enumerate}[(i)]
    \item 
    We have the convergence \begin{equation}\label{eqn:omega_limit_points:existence}
	    u_{i+ \lceil c t_{n_k}\rceil,j+ j_{n_k}} (t+ t_{n_k}) \to \omega_{i,j}(t) \quad  \text{ in } C_{\mathrm{loc}}(\Z^2\times \R)
	\end{equation}
	as $k \to \infty$.
    \item The limit $\omega$ 
    satisfies
    the discrete Allen-Cahn equation  \eqref{eqn:main_results:discrete AC} on $\Z^2\times\R$.
    \item There exists a constant $\theta \in \R$ such that
	\begin{equation}\label{eqn:omega_limit_points:omega_limit_between_two_waves}
    \Phi(i-ct-\theta) \leq \omega_{i,j} (t)\leq \Phi(i-ct + \theta),  \quad  \text{for all } i  \in \Z \text{ and } t \in \R. 
    \end{equation}
\end{enumerate}
\end{prop}

We refer to such a function $ \omega $ as an $ \omega $-limit point of the solution $ u $. 
The proof of the bounds
\eqref{eqn:omega_limit_points:omega_limit_between_two_waves} relies on the 
fact that the LDE \eqref{eqn:main_results:discrete AC}
admits a comparison principle;
see \cite[Prop. 3.1]{hoffman2017entire}.
In order
to exploit this,
we introduce the
residual
\begin{equation}\label{eqn:omega_limit_points:residual}
    \mathcal{J}[u]
     = \dot{u} - \Delta^+ u - f(u)
\end{equation}
and recall that a
function
\begin{equation}
    u \in C^1\big( [0,\infty); \ell^\infty(\mathbb{Z}^2)
    \big)
\end{equation}
is referred
to as a sub-
or super-solution to the discrete Allen-Cahn equation
\eqref{eqn:main_results:discrete AC} if $\mathcal{J}[u]_{i,j}(t) \le 0$
respectively
$\mathcal{J}[u]_{i,j}(t) \ge 0$
holds for all $t \ge 0$
and $(i,j) \in \mathbb{Z}^2$.
Our first result
describes a standard pair of such solutions, using the well-known principle that uniform perturbations
to the travelling wave $\Phi$ at $t = 0$ can be traded off
for phase-shifts at $t = \infty$. 
\begin{lemma}
\label{lemma:analysis_of_u:super_and_sub_solution}
Assume that (H$g$) and (H$\Phi$) are satisfied. Then 
for any  $q_0\in(0,a)$ and $q_1\in(0, 1-a)$, 
there exist constants $\mu > 0$ and  $C\geq 1$ 
so that
the functions
\begin{align}
    u^+_{i,j}(t) &= \Phi\left(i-ct + C q_0(1-e^{-\mu t})\right) + q_0 e^{-\mu t}, \label{eqn:analysis_of_u:supersolution} \\
    u^-_{i,j}(t) &=  \Phi\left(i-ct - C q_1(1-e^{-\mu t})\right) - q_1 e^{-\mu t} \label{eqn:analysis_of_u:subsolution}
\end{align}
are a super- 
respectively sub-solution of the discrete Allen-Cahn equation \eqref{eqn:main_results:discrete AC}. 
\end{lemma}
\begin{proof}
The arguments from Lemma 4.1 in \cite{fife1977approach} can be copied almost verbatim; see for example~\cite{ChenGuoWu2008}.
\end{proof}
We now turn to
the solution $u$
of the LDE~\eqref{eqn:main_results:discrete AC} with the initial condition
\eqref{eqn:main_results:initial condition}. Using two a-priori estimates
we will show that $u$ can eventually  be controlled
by time translates of $u^+$ and $u^-$. By exploiting
the divergence $t_k \to \infty$ of the time-shifts
for the $\omega$-limit point,
we can subsequently eliminate the 
uniform additive
terms in \eqref{eqn:analysis_of_u:supersolution}-\eqref{eqn:analysis_of_u:subsolution}
and recover the
phase-shifts in 
\eqref{eqn:omega_limit_points:omega_limit_between_two_waves}.

\begin{lemma}\label{lemma:analysis_of_u:liminfu_leq_q0}
Assume that (H$g$) and (H0) are satisfied. Pick $q_0\in(0,a) $ in such a way that the initial condition $u^0$ satisfies $$\limsup_{i\to-\infty} \sup_{j\in \Z} u^0_{i,j} < q_0.$$ Then for every $t>0$ we have the bound
\begin{equation}\label{eqn:analysis_of_u:liminf_u}
    \limsup_{i\to-\infty} \sup_{j\in \Z} u_{i,j}(t) < q_0. 
\end{equation}
\end{lemma}

\begin{proof}
First, we find a constant $d\in(0, q_0)$ for which
\begin{equation}
    \limsup_{i\to-\infty} \sup_{j\in \Z} u^0_{i,j} < d. 
\end{equation}
Next, we pick a constant $M$ in such a way that 
\begin{equation}
    u^0_{ i,j} \leq  d+Me^{i|c|}, \quad \text{for every }(i,j)\in \Z^2.   
\end{equation}
Writing $K>0$ for the maximum value of the function $g$ on the interval $[a,1]$, we choose $\alpha > 0$ sufficiently large to have
\begin{equation}\label{eqn:analysis_of_u:alpha}
    \alpha |c| - \dfrac{c^4}{12} \cosh |c| \geq \dfrac{2K}{a-d}.
\end{equation}
 
 We now claim that the 
$j$-independent function
\begin{equation}
    w_{i,j}(t) =  d + Me^{|c|(i+|c|t + \alpha t)}
\end{equation}
is a super-solution to \eqref{eqn:main_results:discrete AC}. To see this, we compute
\begin{align*}
\mathcal{J}[w]_{i,j}(t) 
	&=  Me^{|c|(i+|c|t+\alpha t)} \left(c^2+\alpha |c| - e^{-|c|} - e^{|c|} +2\right) - g\big(w_{i,j}(t) \big) \\ 
	&=  Me^{|c|(i+|c|t+\alpha t)} \left(\alpha |c|-\dfrac{c^4}{12} \cosh \tilde{c}\right) - g\big(w_{i,j}(t) \big)\\
	&\geq \big(w_{i,j}(t) - d\big) \frac{2K}{a -d} - g\big(w_{i,j}(t) \big), 
\end{align*}
where $\tilde{c} $ is a number between $0$ and $|c|$.
For $  w_{i,j}(t) \in [0, a] \cup [1, \infty)$, we have $ g\big(w_{i,j}(t) \big) \leq 0 $, which immediately gives $\mathcal{J}[w]_{i,j}(t) \geq 0 $.
On the other hand, for $ w_{i,j}(t) \in [a, 1] $ our choice for $K$ yields
$$
\mathcal{J}[w]_{i,j}(t) \geq (a-d)\dfrac{2K}{a-d} - K \geq K > 0.  
$$

Applying the comparison principle we conclude 
\begin{equation}
    u_{i,j}(t) \leq w_{i,j}(t) = d+ Me^{|c|(i+|c|t + \alpha t)}, \quad 
\end{equation}
$\text{for every } t\geq 0 \ \text{and } (i,j)\in \Z^2$. Taking the supremum over $j\in \Z$ and sending $i$ to $-\infty$ we obtain the desired inequality \eqref{eqn:analysis_of_u:liminf_u}.
\end{proof}

\begin{lemma}\label{lemma:analysis_of_u:u_leq_1}
Suppose that (Hg), $(H\Phi)$ and (H0) are satisfied. Let $u$ be the solution of the discrete Allen-Cahn equation \eqref{eqn:main_results:discrete AC} with the initial condition \eqref{eqn:main_results:initial condition}. Then for every $q_0>0$ there exists $T>0$ so that
\begin{equation}\label{eqn:analysis_of_u:u_leq_1}
u_{i,j}(t) \leq 1+\dfrac{q_0}{2}
\end{equation}
holds for every $t \geq T$
and $(i,j) \in \mathbb{Z}^2$.
\end{lemma}

\begin{proof}
 Let $ \tilde{u} $ be the solution to 
 the scalar initial value problem
	\begin{equation}
	\begin{cases}
	\tilde{u}_t &= g(\tilde{u}) , \qquad t>0 \\
	\tilde{u}(0) &= \norm{u^0}_{\ell^\infty(\Z^2)}. \\
	\end{cases}
	\end{equation}
	Since $g(u) < 0$ for all $u > 1$, 
	there exists $ T>0 $ such that $ \tilde{u}(t) \leq 1+ \frac{q_0}{2} $ 
	for all $ t\geq T $.
	Exploiting
	the fact that
	$ \tilde{u} $ is also a
	spatially homogeneous
	solution to
	\eqref{eqn:main_results:discrete AC},
	the comparison principle
	yields $ u_{i,j}(t) \leq \tilde{u}(t)$ for all $ t\geq 0 $ and $(i,j) \in \mathbb{Z}^2$. Combining
	these observations
	leads directly to
	\eqref{eqn:analysis_of_u:u_leq_1}.
	%
%
\end{proof}

\begin{lemma}\label{lemma:super_and_sub_solutions_for_the_discrete_AC:upper and lower bounds}
Assume that (Hg), $(H\Phi)$ and (H0) are satisfied. 
Then there exists a time $T>0$ together with constants 
\begin{equation}
    q_0\in(0,a), \qquad q_1\in(0,1-a), 
    \qquad  \theta_0\in \R , \qquad \theta_1\in \R,
    \qquad \mu > 0, \qquad C > 0
\end{equation} 
so that the solution $u$ to \eqref{eqn:main_results:discrete AC} with the initial condition \eqref{eqn:main_results:initial condition} satisfies the estimates
\begin{align}
u_{i,j}(t) \leq \Phi\left(i+ \theta_0 - c(t-T) + Cq_0\big(1-e^{-\mu (t-T)}\big)\right) + q_0 e^{-\mu (t-T)}, \quad \forall t\geq T, \label{eqn:super_and_sub_solutions_for_the_discrete_AC:upper_bound} \\
u_{i,j}(t) \geq \Phi\left(i-\theta_1 - c(t-T) - Cq_1\big(1-e^{-\mu (t-T)}\big)\right) - q_1 e^{-\mu (t-T)}, \quad \forall t\geq T. \label{eqn:super_and_sub_solutions_for_the_discrete_AC:lower_bound}
\end{align}
\end{lemma}

\begin{proof}
We first choose $q_0\in (0,a)$ 
in such a way that
\begin{equation}\label{eqn:analysis_of_u:liminf_u0_leq_q0}
    \limsup_{i\to -\infty} \sup_{j\in \Z} u^0_{i,j} < q_0.
\end{equation}
Using Lemma \ref{lemma:analysis_of_u:u_leq_1}, we obtain  $T>0$ for which 
\begin{equation}\label{eqn:analysis_of_u:u_leq_1_in_proof}
    u_{i,j} (T) \leq 1+\dfrac{q_0}{2} \quad \text{for every }(i,j)\in \Z^2.
\end{equation}
On the other hand,
Lemma~\ref{lemma:analysis_of_u:liminfu_leq_q0} 
allows us to find
$\vartheta_a\in \Z$ so  that
\begin{equation}\label{eqn:analysis_of_u:u(T)_leq_q0_in_proof}
    u_{i,j}(T) \leq q_0, \quad \text{for } i\leq \vartheta_a \text{ and}\ j\in \Z. 
\end{equation}
Finally, in view of the limits~ \eqref{eqn:main_results:boundary:cond:for:phi} there exists $\vartheta_b \in \Z$ for which
$$ \Phi(i)\geq1-\frac{q_0}{2}, \quad \text{for every } i\geq \vartheta_b. $$
Combining these inequalities
and recalling
the definition
\eqref{eqn:analysis_of_u:supersolution},
we obtain
\begin{equation}
	u_{i,j}(T)  \leq \Phi(i-\vartheta_a + \vartheta_b) + q_0 = u_{i-\vartheta_a +\vartheta_b }^+(0)  
	\end{equation} 
for all $i \in \Z$.
The desired upper bound
\eqref{eqn:super_and_sub_solutions_for_the_discrete_AC:upper_bound} with $\theta_0 = \vartheta_b - \vartheta_a$
now follows from Lemma
\ref{lemma:analysis_of_u:super_and_sub_solution} and the comparison principle. The lower bound can be obtained
in a similar fashion.

\end{proof}

\begin{proof}[Proof of Proposition~\ref{lemma:omega_limit_points:construction of limit points}]
Fix an integer $L\in \N$ and consider the functions $$u^k\in C\big([-L,L]; \R^{(2L+1)\times (2L+1)}\big)$$ that are defined by 
$$ u^k_{i,j}(t) = u_{i+\lceil c t_{n_k}\rceil, j+j_{n_k}}(t+ t_{n_k}), \quad \quad  (i,j,t) \in \{-L, \ldots, L\}^2\times[-L,L]$$
for all sufficiently large $k$.  Lemma~\ref{lemma:analysis_of_u:u_leq_1} implies that the solution $u$ and hence  the functions $u^k $ are globally bounded. Since the derivative $\dot{u}$ satisfies \eqref{eqn:main_results:discrete AC}, it follows that  $\dot{u}^k$ is also a globally bounded sequence.  Hence, Ascoli-Arzela implies that the sequence $ u^k $ is relatively compact. By using 
 a standard diagonalization argument together with \eqref{eqn:main_results:discrete AC}, we obtain a subsequence $ u^{n_k} $ and a function $ \omega:\R\to \ell^\infty(\Z^2) $ so that
	$$ \sup_{(i,j,t)\in K}|u^{n_k}_{i,j}(t) - \omega_{i,j}(t)|
	+ |\dot{u}^{n_k}_{i,j}(t) - \dot{\omega}_{i,j}(t)| \to 0,$$
	%
    for every compact $K\subset \Z^2 \times \R$. This immediately implies \textit{(i)} and \textit{(ii)}. The bounds  \eqref{eqn:omega_limit_points:omega_limit_between_two_waves} follow directly from Lemma~\ref{lemma:super_and_sub_solutions_for_the_discrete_AC:upper and lower bounds}.
    \end{proof}

\section{Trapped entire solutions}
\label{sec:Trapped entire solutions}

The main point of this section is to prove that every entire solution that is trapped between two traveling waves is a traveling wave itself. This is a very 
useful result 
when combined 
with Proposition ~\ref{lemma:omega_limit_points:construction of limit points}, 
since it implies that every $\omega$-limit point of the solution $u$ is a traveling wave. This will turn
out to be a crucial tool 
during our analysis of the large time behaviour of $u$.

\begin{prop}\label{thm:trapped_entire_sol:every_trapped_entire_solution_is_a_traveling_wave}
	Assume that (H$g$) and (H$\Phi$) are satisfied
	and consider a
	function $\omega \in C^1\big(\R; \ell^\infty(\Z^2)
	\big)$
	that satisfies
	the Allen-Cahn LDE~\eqref{eqn:main_results:discrete AC}
	 for all $t \in \R$.
	%
	Assume furthermore that there exists a constant $\theta$ for which the bounds
	\begin{equation}\label{eqn:trapped_entire_sol:function_trapped_between_two_tw}
	 \Phi(i-ct-\theta) \leq \omega_{i,j} (t) \leq \Phi(i-ct+\theta) 
	\end{equation}
	hold for 
	all $(i,j) \in \Z^2$
	and $t\in \R$.
	Then there exists a constant $ \theta_0\in [-\theta, \theta] $  so that
	$$\omega_{i,j}(t) = \Phi(i-ct-\theta_0), \quad \hbox{for all } (i,j)\in \Z^2, \, t\in \R.$$
\end{prop} 
This result is a 
generalization
of \cite[Thm. 3.1]{berestycki2007generalized}
 to the current spatially discrete setting. The main complication lies in the fact that the LDE~\eqref{eqn:main_results:discrete AC} is a nonlocal equation, as opposed to the PDE~\eqref{eqn:intro:PDE}. For example, if a smooth  function $f:E\subset \R^2\to\R$ attains a local minimum at some point $x_0$, then 
 we automatically have
 $\Delta f (x_0)\geq 0$. This is an important ingredient for the arguments in \cite{berestycki2007generalized},
 but fails to hold in our spatially discrete setting.
 
 Indeed, if 
 $v\in \ell^\infty (\Z^2)$  attains a minimum in $E\subset \Z^2 $ at some point $(i,j) \in E$,
 it does not automatically follow that the discrete Laplacian satisfies $(\Delta^+ v)_{i,j}\geq 0$.
 This conclusion can only be obtained
 if one can verify that
 the nearest neighbours of $(i,j)$
 are also contained in $E$.
  This is the key purpose of our first technical result.

 \begin{lemma}\label{lemma:trapped_entire_sol:first_lemma}
Consider the setting of Proposition~\ref{thm:trapped_entire_sol:every_trapped_entire_solution_is_a_traveling_wave} and pick a  sufficiently small $\delta>0$. 
Choose a pair $(I,J)\in \Z^2$ together with a constant $\sigma\in \R$. Suppose for some $\kappa \in \Z$ that the function 
\begin{equation}\label{eqn:trapped_entire_sols:v_sigma}
    v^\sigma_{i,j}(t) = \omega_{i+I, j+J}\left(t+\dfrac{I}{c} + \dfrac{\sigma}{c}\right)
\end{equation}
satisfies the inequality
\begin{equation}
\label{eq:trp:ineq:v:sig:omega}
v^\sigma_{i,j}(t) \leq \omega_{i,j}(t)
\end{equation}
whenever  $i-ct\in [\kappa , \kappa +1]$. 
Then the  
following claims holds true. 
\begin{enumerate}[(i)]
    \item If  $\omega_{i,j}(t) \geq 1-\delta$ 
    whenever $i-ct\geq \kappa $, then 
    in fact \eqref{eq:trp:ineq:v:sig:omega}
    holds for all $i-ct\geq \kappa $.
    %
    \item If $v^\sigma_{i,j}(t) \leq \delta $ whenever $i-ct \leq \kappa+1 $, then 
    in fact \eqref{eq:trp:ineq:v:sig:omega}
    holds for all  $i-ct \leq \kappa +1$.
\end{enumerate}
\end{lemma}
\begin{proof}
Starting with (i), we define the set 
$$
    E:=\left\{(i,j,t) \in \Z^2\times\R: i-ct \geq \kappa  \right\}.
$$
Since both functions $\omega$ and $v^\sigma$ are globally bounded, the quantity
$$\epsilon^* = \inf\left\{\epsilon>0: 
v^\sigma \leq \omega + \epsilon   \text{ in } E \right\}$$
is finite. In addition, by continuity we have 
\begin{equation}\label{eqn:trapped_entire_sols:vsigma<omega+eps}
    v^{\sigma} \leq \omega + \epsilon^* \quad \text{in }E. 
\end{equation}
To prove the claim, it  suffices to show that
$\epsilon^* = 0$. Assuming to the contrary
that $\epsilon^*>0$, we can
find  sequences $\epsilon_n \nearrow \epsilon^*$ and $(i_n, j_n, t_n) $ in $E$ with the property 
that
\begin{equation}
    \omega_{i_n, j_n}(t_n) + \epsilon_n < v^{\sigma}_{i_n, j_n}(t_n)\leq \omega_{i_n, j_n}(t_n) + \epsilon^* \quad \text{for each } n\in \N. 
\end{equation}
Sending $n\to\infty$ we conclude that
\begin{equation} \label{eqn:trapped_entire_sol_:limit in trapped wave}
    \lim_{n\to\infty} \omega_{i_n, j_n}(t_n) -  v^{\sigma}_{i_n, j_n}(t_n) + \epsilon^* = 0.
\end{equation}
Now, notice that the assumption \eqref{eqn:trapped_entire_sol:function_trapped_between_two_tw} and the inequality $\epsilon^* > 0$ imply  
that the sequence $l_n:=i_n-c t_n $ is bounded.
In addition, our assumption
\eqref{eq:trp:ineq:v:sig:omega} implies
that $l_n > \kappa + 1$. In particular, we can assume that the bounded sequence $i_n-\lceil c t_n \rceil$ is equal to an integer $L\geq \kappa$.

Applying Proposition 3.1 to the function $\omega$ and
the sequence $(j_n, t_n)$, we obtain a 
limiting function $\omega^\infty$ for which we have
\begin{equation}\label{eqn:trapped_entire_sols:limit}
    \lim_{n\to\infty} \omega_{i+\lceil ct_n \rceil, j+j_n}(t+ t_n) = \omega^{\infty}_{i, j}(t), 
\end{equation}
for each $(i,j,t)\in\Z^2\times\R$.
By construction it also holds that
\begin{equation}
    \lim_{n\to\infty} v^\sigma_{i+\lceil ct_n \rceil, j+j_n}(t+ t_n) = \omega^{\infty}_{i+I, j+J}(t+\dfrac{I}{c} + \dfrac{\sigma}{c}). 
\end{equation}
 Next we define the function $ z= z_{i,j}(t)$ as
\begin{equation}
    z_{i,j}(t) = \omega^{\infty}_{i, j}(t) - \omega^{\infty}_{i+I, j+J}(t+\dfrac{I}{c} + \dfrac{\sigma}{c}) + \epsilon^*.
\end{equation}
    For $(i,j,t) \in E$ we have $(i+\lceil ct_n\rceil, j+j_n, t+t_n)\in E$. Combining this with the fact that the inequality~\eqref{eqn:trapped_entire_sols:vsigma<omega+eps} survives  the limit~\eqref{eqn:trapped_entire_sols:limit}, we have  $z_{i,j}(t)\geq 0$ in $E$. By \eqref{eqn:trapped_entire_sol_:limit in trapped wave} we obtain $z_{L, 0}(0) = 0$.
    Also, for $i-ct=\kappa$, we have $(i+\lceil ct_n\rceil, j+j_n, t+t_n) \in [\kappa, \kappa+1]$. In particular, we find
    \begin{equation}\label{eqn:trapped_entire_sols:z>eps}z_{i,j}(t) \geq \epsilon^*>0, \quad \text{for } i-ct=\kappa.
    \end{equation}
    Therefore, it must hold that  $L\geq \kappa+1$. 
    
    We pick $\delta$ to be small enough so that $g$ is non-increasing on $[1-\delta, 1]$. Since $ \omega^\infty \in [ 1-\delta, 1] $ and $ g $ is locally Lipschitz continuous on $ E $, there exists $B>0$ so that
	\begin{equation}
	\begin{aligned}
	\dot{z}_{i,j}(t) - (\Delta^+ z)_{i,j}(t) &= g\big(\omega^\infty_{i,j}(t) \big) - g\big(\omega^\infty_{i+I,j+J}(t+\dfrac{I}{c}  +\dfrac{\sigma}{c})\big) \\
	&\geq g(\omega^\infty_{i,j}(t) + \epsilon^*) - g\big(\omega^\infty_{i+I,j+J}(t
	+\dfrac{I}{c} + \dfrac{\sigma}{c})\big) \\
	&\geq -B z_{i,j}(t) 
	\end{aligned}
	\end{equation}
	for all $(i,j,t) \in E$.
Since $z$ attains its minimum at the point $(L, 0, 0) \in E$ with $L\geq \kappa +1$, we have $\dot{z}_{L,0}(0) = 0$. In addition,
the inequality  $(\Delta^+z)_{L, 0}(0) \geq 0 $ holds since all the nearest neighbours of $(L, 0, 0)$ are contained in $E$. In particular, we compute
     \begin{equation}
     0 \leq \dot{z}_{L,0}(0) - ( \Delta^+z)_{L, 0}(0)  + B z_{L,0} (0) = - ( \Delta^+z)_{L, 0} (0)\leq 0.
     \end{equation}
Therefore, $ (\Delta^+z)(0)_{L, 0}= 0 $ must hold, which implies that $ z_{0, L-1 }(0) = 0 $. 

If $ L=\kappa  + 1 $ then we are done,  since $ z\geq \epsilon^* > 0 $ for $ i-ct=\kappa $
which contradicts \eqref{eqn:trapped_entire_sols:z>eps}.
%
On the other hand, if $ L - 1 \geq \kappa +1  $ we can  iteratively decrease $L$ using this procedure until we reach the desired contradiction. 
Statement (ii) can be obtained in a similar fashion using $\lfloor ct_n \rfloor$ instead of $\lceil ct_n \rceil$. 
\end{proof}

\smallskip
	
\begin{lemma}\label{lemma:trapped_entire_sol:sigma leq 0}
 Consider the setting of Propostion~\ref{thm:trapped_entire_sol:every_trapped_entire_solution_is_a_traveling_wave},
 fix an arbitrary pair $(I,J)\in \Z^2$
 and recall the functions $v^\sigma$ defined in \eqref{eqn:trapped_entire_sols:v_sigma}
Then the quantity 
   \begin{equation}
       \sigma_*:= \inf \left\{\sigma\in \R: v^{\tilde{\sigma}}\leq \omega  \text{ in } \Z^2\times\R  \text{ for all }  \tilde{\sigma} \geq \sigma    \right\}
   \end{equation}
satisfies 
$\sigma_*\leq 0$. 
\end{lemma}
   
\begin{proof} First we show that $\sigma_*<\infty.$  
    Without loss we may assume that $0<\delta<1/2 $ holds for the constant defined in Lemma~\ref{lemma:trapped_entire_sol:first_lemma}. The inequalities \eqref{eqn:trapped_entire_sol:function_trapped_between_two_tw} allow $\kappa \in \N$ such that 
    \begin{equation}
      \label{eqn:trapped_entire_sol:ineqaulity for w in -infty}
	   \begin{array}{lcll}
	    \omega_{i,j}(t) &\geq &1-\delta, \quad &i-ct\geq \kappa,\\[0.2cm]
	    \omega_{i,j}(t) &\leq &\delta,\quad  &i-ct\leq -\kappa.
	   \end{array}
	 \end{equation}

    For $\sigma \geq 2\kappa  + 1$ and $i-ct \in \le \kappa +1$ one has $i-ct-\sigma \leq -\kappa $. It follows from \eqref{eqn:trapped_entire_sol:ineqaulity for w in -infty} that $v^\sigma \leq \delta $ on $i-ct \le \kappa +1$. Using  $\delta\leq 1-\delta$ we have $v^\sigma \leq \omega$ on $i-ct\in[\kappa, \kappa+1]$. Hence, both items~\textit{(i)} and \textit{(ii)}  of  Lemma~\ref{lemma:trapped_entire_sol:first_lemma} are satisfied and the bound $v^\sigma\leq \omega$ on $\R$ follows immediately. Since $\sigma\geq2\kappa  + 1 $ was arbitrary, we conclude that  $\sigma_*\leq 2\kappa+1 $. 
    
    Arguing by contradiction, let us assume that $\sigma_* > 0$. Defining the set 
   \begin{equation}
       S=\left\{ -\kappa-1\leq i-ct\leq \kappa+1 \right\},
   \end{equation}
 we now claim that  \begin{equation}\label{eqn:trapped_entire_sol:inf(w-v_sigma)=0}
       \inf_S \ (\omega-v^{\sigma_*}) = 0. 
   \end{equation}
 Assume to the contrary that  $\inf_S(w-{v^\sigma}^*) = K>0$. Then, using the global Lipschitz continuity of $\omega$, there exists a constant $M>0$ such that
  \begin{equation}
        \begin{aligned}
            \omega_{i,j}(t) - v^{\sigma_*-\mu}_{i,j}(t)  =  \omega_{i,j}(t) -  v^{\sigma_*}_{i,j}(t) +
            v^{\sigma_*}_{i,j}(t) -
            v^{\sigma_*-\mu}_{i,j}(t)  \geq K-M\mu
        \end{aligned}
   \end{equation}
 holds for every $\mu \geq 0$ and $(i,j,t) \in S$. Hence, there exists $\mu_0 \in (0, \sigma_*)$ such that $v^{\sigma_*-\mu}\leq \omega$ on $S$, for all $\mu \in [0, \mu_0] $.  Item~\textit{(i)} in Lemma~\ref{lemma:trapped_entire_sol:first_lemma} implies that $v^{\sigma_* - \mu}\leq \omega$ for  $i-ct \geq \kappa $ and for all $\mu \in [0, \mu_0]$. Furthermore, since $\sigma_*-\mu \geq 0$, we have $v^{\sigma_* - \mu}\leq \delta $ for $i-ct\leq -\kappa$. Since also $v^{\sigma_* - \mu}\leq \omega$ for  $-\kappa-1 \leq i-ct \leq -\kappa$, item   \textit{(ii)} of Lemma~\ref{lemma:trapped_entire_sol:first_lemma} implies that $v^{\sigma_* - \mu}\leq \omega$ also holds on $i-ct \leq -\kappa$. All together, we have $v^{\sigma_* - \mu}\leq \omega$ on $\Z^2\times \R$, which contradicts the minimality of $\sigma_*$ and yields \eqref{eqn:trapped_entire_sol:inf(w-v_sigma)=0}. 
   
 We can hence find a sequence  $(i_n, j_n, t_n)$  in $S$ such that 
    \begin{equation}
    \omega_{i_n, j_n}(t_n) - v^{\sigma_*}_{i_n, j_n}(t_n) \to 0 \qquad \text{ as } n\to\infty .    
    \end{equation}
Since $i_n-ct_n$ is bounded, we can assume that $ i_n - \lceil ct_n\rceil$ is equal to a constant, which we denote by $L$. As before, we obtain the convergence
    \begin{equation}
    \lim_{n\to\infty} \omega_{i+\lceil ct_n \rceil, j+j_n}(t+ t_n) = \omega^{\infty}_{i, j}(t), 
    \end{equation}
    where $\omega^\infty$ is also an entire solution of the LDE~\eqref{eqn:main_results:discrete AC}.
    Hence, the function $z=z_{i,j}(t)$ defined as
    \begin{equation}
        z_{i,j}(t):= \omega^\infty_{i,j}(t)   - \omega^\infty_{i+I,j+J}(t+\dfrac{I}{c} + \dfrac{\sigma_*}{c})
    \end{equation}
    satisfies 
    \begin{equation}
    z_{i,j}\geq 0 \text{ for all } (i,j, t) \in \Z^2\times\R
    \end{equation}
and $ z_{ L,0}(0) = 0 $. Using an argument similar to the one in the proof of Lemma~\ref{lemma:trapped_entire_sol:first_lemma}, it follows that $ z_{i,j}(0) = 0  $ for all $ (i,j)\in \Z^2 $. We then obtain $ z\equiv 0 $ by the uniqueness of bounded solutions for  \eqref{eqn:main_results:discrete AC}.

In particular, we have $ \omega^\infty_{0,0}(0) = \omega^\infty_{kI, kJ}(kI/c+ k\sigma_*/c) $ for all $ k\in \Z $. However, 
we also have the limits
\begin{equation}
\lim_{k \to - \infty} \omega^\infty_{kI, kJ}(kI/c+ k\sigma_*/c) =  1,
\qquad \qquad 
\lim_{k \to \infty}
\omega^\infty_{kI, kJ}(kI/c+ k\sigma_*/c) =  0,
\end{equation}
since $  \omega^\infty $ is trapped between two traveling waves as well.  We have hence reached a contradiction and conclude $ \sigma_* \leq 0 $.
\end{proof}

\begin{proof}[Proof of Proposition~\ref{thm:trapped_entire_sol:every_trapped_entire_solution_is_a_traveling_wave}]
From Lemma~\ref{lemma:trapped_entire_sol:sigma leq 0}, we know that 
    \begin{equation}
        \omega_{i,j}(t)\geq \omega_{i+I,j+J}(t+\dfrac{I}{c}) \quad \text{on }\Z^2\times\R, 
    \end{equation}
for arbitrary $(I,J)\in \Z^2$. Hence, the function $\omega$ depends only on the value of $i-ct$. 
More precisely, there exists a function $\psi$ such that $\omega_{i,j}(t) = \psi(i-ct)$. The result now follows directly from the fact that solutions to the travelling wave problem \eqref{eqn:main_results:boundary:cond:for:phi}-\eqref{eqn:main_results:MFDE} for $\theta=0$ and $c \neq 0$ 
are unique up to translation.   
\end{proof}

\section{Large time behaviour of $u$}
\label{sec:large_time_behaviour}       

The main goal of this section is to study the qualitative
large time behaviour of the solution $u$ to our main initial value problem.
In particular, we 
connect this behaviour to the dynamics
of the phase $\gamma$
defined in \eqref{eqn:main_results:def_of_gamma}
and thereby establish Theorem~\ref{thm:main results:gamma_approximates_u}.
In addition, we provide an asymptotic flatness result
for this phase.

Our first main result concerns the
large-time behaviour
of the interfacial region 
\begin{equation}
    I_t =\left\{(i,j) \in \Z^2: \Phi(-2) \leq u_{i,j}(t) \leq  \Phi(2) \right\} 
\end{equation}
where $u$ takes values close to $1/2$.
For fixed $j$ and $t$, we establish that the horizontal coordinate $i$
can not jump in and out from the interface region, which is non-empty.
In particular,
once the map $i \mapsto u_{ij}(t)$ enters the interval $[\Phi(-2), \Phi(2)]$
from below, it cannot exit throughout the lower boundary. 
In addition, it is strictly increasing in $i$
and cannot reenter the interval 
once
it has left through the upper boundary.
\begin{prop}\label{lemma:phase gamma:monotonicity_in_the_bounded_region} Suppose that the assumptions (Hg), (H$\Phi$)  and  (H0) are satisfied and let $u$ be a solution of the discrete Allen-Cahn equation \eqref{eqn:main_results:discrete AC} with the initial condition~\eqref{eqn:main_results:initial condition}. Then there exists a constant $T>0$ so that the following statements are satisfied.
\begin{enumerate}[(i)]
    \item For each $t\geq T$ and $j\in \Z$ there exists $i\in \Z$ for which      
 \begin{equation}\label{eqn:phase gamma:nonempty set around 1/2}
      \Phi(-2) < u_{i,j}(t) \leq \frac{1}{2}.  
  \end{equation}
\item We have the inequality	\begin{equation}\label{eqn:phase gamma:derivative bounded from below}
	\begin{aligned} 
		&\inf_{t \ge T, \,(i, j)\in I_t} u_{i+1, j}(t) - u_{i,j} (t) > 0. \\	
	\end{aligned}
	\end{equation}
\item Consider any $t\geq T$ and $(i,j) \in \Z^2$ for which
$u_{i,j}(t) \leq \Phi(-2)$ holds.
Then we also have $u_{i-1,j}(t) \leq \Phi(-2)$.

\item  Consider any $t\geq T$ and $(i,j) \in \Z^2$
for which $u_{i,j}(t) \geq \Phi(2)$
holds. Then we also have
$u_{i+1,j}(t) \geq \Phi(2)$.

\end{enumerate}
\end{prop}

Our second main result shows that
the discrete derivative of the phase with respect to $j$
tends to zero. This will turn out to be crucial in order to keep the mean curvature flow under control. We emphasize that this does not necessarily mean that the phase tends to a constant; see \eqref{eq:int:kappa:pm:limits}.
\begin{prop}\label{prp:zero_level_surface:smallnes of derivative of gamma}
Consider the setting of Proposition~\ref{lemma:phase gamma:monotonicity_in_the_bounded_region}
and recall the phase $ \gamma: [T, \infty) \to \ell^\infty(\Z) $
defined in \eqref{eqn:main_results:def_of_gamma}. Then 
we have the limit
	$$\lim_{t\to\infty} \sup _{j\in\Z} \big|\gamma_{j+1}(t) - \gamma_j (t)\big| = 0.$$
\end{prop}

\begin{proof}[Proof of Proposition~\ref{prp:main_results_existence of gamma}]
 The statement follows directly from Proposition~\ref{lemma:phase gamma:monotonicity_in_the_bounded_region}.
\end{proof}

\subsection{Proof of Proposition~\ref{lemma:phase gamma:monotonicity_in_the_bounded_region} and Theorem~\ref{thm:main results:gamma_approximates_u}}
The key towards establishing
Proposition~\ref{lemma:phase gamma:monotonicity_in_the_bounded_region} is to obtain strict monotonicity properties in
compact regions that move with the wavespeed $c$. This is achieved
in the following result, which leverages the travelling wave identification obtained
in Proposition~\ref{thm:trapped_entire_sol:every_trapped_entire_solution_is_a_traveling_wave}.
 \begin{lemma}\label{lemma:phase gamma:derivative bounded from below}
 Consider the setting of Proposition \ref{lemma:phase gamma:monotonicity_in_the_bounded_region}
 and pick a constant $R > 0$.
 Then there exists a constant $T>0$ such that
\begin{equation}
   \inf_{j\in \mathbb{Z}, \,  |i-ct|\leq R,\, t\geq T} u_{i+1, j}(t) - u_{i,j}(t) > 0.
\end{equation} 
 \end{lemma}
\begin{proof}
    Arguing by contradiction, let us assume that there exists a constant $ R>0 $
    so that
    $$\inf_{j\in \mathbb{Z}, \, |i-ct|\leq R, \, t\geq T} u_{i+1, j}(t) - u_{i,j}(t) \leq 0$$
    holds for every $ T>0 $. We can then find a sequence  $(t_n, i_n, j_n)\in (0, \infty)\times \Z^2$ with $ 0<t_1<t_2<\dots \to\infty$
    for which we have the inequalities 
    \begin{equation}
        |i_n-ct_n|\leq R , \qquad \qquad  u_{i_n+1, j_n}(t_n) - u_{i_n, j_n}(t_n) \leq 1/n .
    \end{equation} 
    In particular, we may assume that the bounded
    sequence of integers $ i_n - \lceil c t_n\rceil $ is identically equal
    to some constant $L \in \mathbb{Z}$.
    Applying Proposition~\ref{lemma:omega_limit_points:construction of limit points} we obtain the convergence
    \begin{equation}
    u_{i+\lceil c t_n\rceil, j+j_n} (t+t_n) \to \omega_{i j} (t)
    \end{equation}
    as $n \to \infty$, in which $ \omega $ is an $ \omega $-limit point of the function $ u $.
    In view of Proposition \ref{thm:trapped_entire_sol:every_trapped_entire_solution_is_a_traveling_wave}
    we have   $\omega_{i,j}(t)=\Phi(i-ct - \theta_0) $ for some $\theta_0 \in \R$,
    which allow us to write
$$\begin{aligned}
1/n & \geq u_{i_n+1, j_n}(t_n) - u_{i_n, j_n}(t_n) \\[0.2cm]
 &= u_{L + \lceil c t_n \rceil +1, j_n }(t_n) - u_{L + \lceil c t_n \rceil, j_n }(t_n) \\[0.2cm] & \to \omega_{L +1,0}(0) -  \omega_{ L,0}(0) 
 \\[0.2cm]
 & = \Phi(L + 1 - \theta_0) - \Phi(L - \theta_0)
\end{aligned} $$
 for $n \to \infty$. This violates the strict monotonicity $\Phi' > 0$
 and hence yields the desired contradiction.
\end{proof}

 
\begin{proof}[Proof of Proposition \ref{lemma:phase gamma:monotonicity_in_the_bounded_region}]
We first prove item \textit{(iii)}. Assuming
that this statement fails, 
we can find a sequence $(t_k, i_k, j_k)$ for which we have $0<t_1 < t_2 < \ldots \to \infty$ together with the inequalities
\begin{equation}\label{eqn:phase_gamma:assume_contrary inequality}
u_{i_k, j_k}(t_k) \leq \Phi(-2), 
\qquad \qquad 
u_{i_k-1, j_k}(t_k) >  \Phi(-2).
\end{equation}
It follows from Lemma~\ref{lemma:super_and_sub_solutions_for_the_discrete_AC:upper and lower bounds} that the sequence $i_k-ct_k$ is bounded. 
Arguing as in the proof of Lemma \ref{lemma:phase gamma:derivative bounded from below}, we can hence again assume that there exists $L \in \mathbb{Z}$
for which we have $L = i_k - \lceil ct_k \rceil $. In addition,
we obtain the limits
\begin{equation}
    u_{i_k, j_k}(t_k) \to \omega_{L, 0}(0) \leq \Phi(-2) ,
    \qquad \qquad
    u_{i_k-1, j_k}(t_k) \to \omega_{L-1, 0}(0) \geq \Phi(-2). 
\end{equation}
Here $\omega$ is an $\omega$-limit point for $u$, which must be a
travelling wave by Proposition \ref{thm:trapped_entire_sol:every_trapped_entire_solution_is_a_traveling_wave}.
This again violates the strict monotonicity of $\Phi$.
Item \textit{(iv)} follows analogously.

Turning to (i), we assume that there exists
a sequence  $(t_k, i_k, j_k)$ with $T\leq t_1 < t_2 < ...\to \infty$
together with 
\begin{equation}\label{eqn:phase_gamma:assume_contrary_2}
    u_{i_k, j_k}(t_k) \leq  \Phi(-2), 
    \qquad \qquad 
    u_{i_k + 1, j_k}(t_k) > \dfrac{1}{2} = \Phi(0)  
\end{equation}
and seek a contradiction. Arguing as above, we can find $L \in \mathbb{Z}$
together with an $\omega$-limit point $\omega$ for $u$ with
\begin{equation}
    \omega_{L,0} \le \Phi(-2), \qquad \qquad \omega_{L+1, 0}(0)\geq \Phi(0),
\end{equation}
which violates
Proposition \ref{thm:trapped_entire_sol:every_trapped_entire_solution_is_a_traveling_wave}.

It remains to establish (ii). Picking
$t \ge T$ and $(i,j) \in I_t$,
it follows from
Lemma~\ref{lemma:super_and_sub_solutions_for_the_discrete_AC:upper and lower bounds} that $i-ct$ is bounded by some constant $R$ that depends 
only on $T$. Increasing $T$ if necessary, we can apply
Lemma~\ref{lemma:phase gamma:derivative bounded from below}
to obtain the desired bound
\eqref{eqn:phase gamma:derivative bounded from below}.
%
\end{proof}
\begin{lemma}\label{lemma:phase_gamma:gamma-ct_bounded}
   Consider the setting of Proposition~\ref{lemma:phase gamma:monotonicity_in_the_bounded_region}
and recall the phase $ \gamma: [T, \infty) \to \ell^\infty(\Z) $
defined in \eqref{eqn:main_results:def_of_gamma}. Then there exists $T_*\geq T$ and $M>0$ such that for every $t\geq T_*$ we have 
\begin{equation}\label{eqn:phase_gamma:gamma-ct_bounded}
    \norm{\gamma(t) - ct}_{\ell^\infty} \leq M .
\end{equation}
\end{lemma}
\begin{proof}
In view of the definition~\eqref{eqn:main_results:def_of_gamma}
it suffices to show that $i_* - ct $ is bounded.
Combining Lemma~\ref{lemma:super_and_sub_solutions_for_the_discrete_AC:upper and lower bounds} and \eqref{eqn:main_results:i(j,t))} and possibly
increasing $T > 0$, we see that
\begin{equation}
 \Phi(0) \le u_{i_*(j,t) +1, j}(t) \le
    \Phi\Big(i_*(j,t) +1 + \theta_0 - ct + cT + Cq_0(1-e^{-\mu(t-T)}) \Big) + q_0 e^{-\mu(t-T)} 
\end{equation}
for all $t\geq T$. Choosing $T_* \ge T$ in such a way that 
$$\Phi(0) - q_0 e^{-\mu(T_*-T)}\geq \Phi(-1), $$
we conclude that
\begin{equation}
i_*(j,t) + 1 + \theta_0 - ct + cT + Cq_0(1-e^{-\mu(t-T)}) > -1,
\qquad \qquad 
t \ge T_*.
\end{equation}
Hence, $i_* - ct$ is bounded from below. An upper bound can be obtained in a similar way. 
\end{proof}

\begin{proof}[Proof of Theorem \ref{thm:main results:gamma_approximates_u}] 
    Arguing by contradiction once more, let us assume
    that there exist $\delta>0  $ together with sequences
    $(i_k, j_k) \in \Z^2$ and $ T\leq t_1 < t_2 < \dots \to \infty $
    for which
	\begin{equation}\label{eqn:phase_gamma:assume_contrary_in_proof_that_gamma_approx_u}
	 | \Delta_k | := |u_{i_k,j_k}(t_k) - \Phi\big(i_k - \gamma_{j_k}(t_k)\big)|\geq \delta.
	\end{equation}
	We first claim that the sequence $ i_k - ct_k  $ is bounded.
	To see this, we first use Lemma~\ref{lemma:phase_gamma:gamma-ct_bounded} to conclude that $ \gamma_{j_k}(t_k) - ct_k $ is bounded. 
	Using  \eqref{eqn:super_and_sub_solutions_for_the_discrete_AC:lower_bound}
	we subsequently find
	$$\begin{aligned}
	\Delta_k 
	&\geq  
	\Phi\big(i_k-ct_k + \alpha_k \big) - q_1 e^{-\mu (t_k-T)} 
	-\Phi\big(i_k - ct_k + \beta_k \big), \\
	\end{aligned}$$
	in which 
	\begin{equation}
	    \alpha_k =  cT - \theta_1 - C q_1(1-e^{-\mu(t_k-T)}),
	    \qquad \qquad
	    \beta_k = ct_k - \gamma_{j_k}(t_k) 
	\end{equation}
	are two bounded sequences. In particular, if 
	 $ i_k - ct_k  $ is unbounded we can use the exponential decay
	 of $\Phi$ to achieve $\Delta_k \ge -\delta$ for all large $k$.
	 A similar argument using \eqref{eqn:super_and_sub_solutions_for_the_discrete_AC:upper_bound}
	 yields $\Delta_k \le \delta$, which contradicts
	 \eqref{eqn:phase_gamma:assume_contrary_in_proof_that_gamma_approx_u}
	 and hence establishes our claim.
	 
	 In particular, we can extract a constant subsequence
	 $ i_k- \lceil ct_k\rceil=:L\in \Z$. Passing to a further subsequence,
	 we may also assume that $ i_*(j_k, t_k) - \lceil ct_k \rceil =:\tilde{L}\in \Z$.
	 The definition~\eqref{eqn:main_results:def_of_gamma} allows us to write
	 $$\begin{aligned}
	\Phi\big(i_k - \gamma_{j_k}(t_k)\big)  &= \Phi\Big(i_k - i_*(j_k, t_k) + \Phi^{-1}\left(u_{ i_*(j_k, t_k), j_k}(t_k)\right)\Big) \\
	&=\Phi\Big(i_k - \lceil ct_k \rceil - i_*(j_k, t_k)+ \lceil c t_k\rceil + \Phi^{-1}\left(u_{i_*(j_k, t_k), j_k}(t_k)\right)\Big) \\
	&=\Phi\left(L  - \tilde{L}  + \Phi^{-1}\big(u_{ \tilde{L} + \lceil c t_k \rceil , j_k}(t_k)\big)\right). \\
	\end{aligned}$$
	Applying  Proposition~\ref{lemma:omega_limit_points:construction of limit points}, we see that there exists an $\omega$-limit point $\omega$
	for $u$ for which the limits
	\begin{equation}
	    u_{i_k, j_k} (t_k)  \to \omega_{L, 0} (0),
	    \qquad \qquad  u_{\tilde{L} + \lceil c t_k \rceil, j_k} (t_k)  \to \omega_{\tilde{L}, 0} (0)
	\end{equation}
	hold as $k \to \infty$. Writing $\omega_{i,j}(t) = \Phi( i - ct - x_0)$
	in view of Proposition~\ref{thm:trapped_entire_sol:every_trapped_entire_solution_is_a_traveling_wave}, we hence find
	\begin{equation}
	    \Delta_k \to \Phi(L - x_0) - \Phi\Big( L - \tilde{L} + \Phi^{-1}\big( \Phi(\tilde{L} -x_0) \big) \Big)
	    = 0
	\end{equation}
	as $k \to \infty$, which clearly contradicts \eqref{eqn:phase_gamma:assume_contrary_in_proof_that_gamma_approx_u}.
	\end{proof}

\subsection{Phase asymptotics}

In this subsection we shift 
our attention to vertical differences of the phase $\gamma$,
in order to establish Proposition~\ref{prp:zero_level_surface:smallnes of derivative of gamma}.
Our first result resembles
Lemma \ref{lemma:phase gamma:derivative bounded from below}
in the sense that we study the
interfacial region of the wave,
but in this case we get a flatness
result. This can subsequently
be used to obtain 
a bound on the vertical differences
of the function $i_*$
defined in \eqref{eqn:main_results:i(j,t))},
which in view of \eqref{eqn:main_results:def_of_gamma} allows us to analyze
the phase $\gamma$.

\begin{lemma} \label{lemma:phase_gamma:decay of y-derivatives}
  Consider the setting of Proposition~\ref{lemma:phase gamma:monotonicity_in_the_bounded_region} and pick a constant
  $R > 0$. Then we have the limit 
	$$\lim_{t\to\infty} \sup_{j\in \mathbb{Z}, \, |i-ct|\leq R} |u_{i,j+1}(t) - u_{i,j}(t)| = 0.$$
\end{lemma}

\begin{proof}\normalfont
	Assume to the contrary that there exist constants $ R> 0 $ and $ \delta > 0$
	together with sequences $ (i_k, j_k) \in \Z^2 $ and $ 0<t_1<t_2<...\to\infty $
	that satisfy the inequalities
	\begin{equation}\label{eq:ph:contr:perp}
	|i_k - ct_k| \leq R,
	\qquad \qquad
	   	|u_{i_k,j_k+1}(t_k) - u_{i_k, j_k}(t_k)|\geq \delta.
	\end{equation}
	 As in the proof of the Lemma~\ref{lemma:phase gamma:derivative bounded from below}, we may assume that  $ i_k - \lceil c t_k \rceil = L \in \Z$
	 and use Proposition~\ref{lemma:omega_limit_points:construction of limit points}
    to conclude the convergence
	 \begin{equation}
	   \begin{array}{lcl}
	     u_{i_k, j_k+1} (t+t_k) - u_{i_k, j_n} (t_k) 
	      &= &u_{L+\lceil ct_k\rceil, j_k+1} (t_k) - u_{L+\lceil ct_k\rceil, j_k}
	      (t_k)
	     \\[0.2cm]
	     & \to & \omega_{L, 1}(0) - \omega_{L, 0}(0)
	     \\[0.2cm]
	     & = & 0 ,
	   \end{array}
	 \end{equation}
	 in which $\omega$ is an $\omega$-limit point of the function $u$.
	 The last identity follows from Proposition~\ref{thm:trapped_entire_sol:every_trapped_entire_solution_is_a_traveling_wave}, which states that $\omega$ is a planar wave travelling
	 in the horizontal direction. This obviously contradicts
	 \eqref{eq:ph:contr:perp} and hence concludes the proof.
%
	 \end{proof}


\begin{lemma} \label{lemma:phase_gamma:i(j+1)-i(j)}
Consider the setting of Proposition~\ref{lemma:phase gamma:monotonicity_in_the_bounded_region}  
and recall the function $i_*$ defined by \eqref{eqn:main_results:i(j,t))}.
	Then there exists $ \tilde{T} > T $ so that 
	\begin{equation}
	   | i_*(j+1,t) - i_*(j,t)| \leq 1 
	\end{equation}
	holds for all $ j\in \Z $ and all $ t\geq \tilde{T} $.
\end{lemma}

\begin{proof} \normalfont
	If the above claim does not hold, we can find sequences $ (i_k, \tilde{i}_k, j_k) \in \Z^3$ and $ T< t_1 <t_2 < \dots \to \infty $
	for which the inequality $ |i_k-\tilde{i}_k|> 1 $
	holds, together with 
	\begin{equation}\label{eqn:phase_gamma:cases1}
	\begin{cases}
	u_{i_k, j_k}(t_k) \leq 1/2,\\
	u_{i_k+1, j_k}(t_k) > 1/2, 
	\end{cases}	\qquad \qquad \qquad \qquad \qquad
	\begin{cases}
	u_{\tilde{i}_k, j_k+1}(t_k)  \leq 1/2,\\
	u_{\tilde{i}_k+1, j_k+1}(t_k) >   1/2 .
	\end{cases}    
	\end{equation}
	
As before, we can assume that $ i_k  + \lceil c t_k\rceil = L \in \Z $ and $ \tilde{i}_k  + \lceil c t_k\rceil = \tilde{L} \in \Z$.
In addition, we can use Proposition~\ref{lemma:omega_limit_points:construction of limit points} 
to construct an $\omega$-limit point $\omega$ for $u$ that satisfies
the inequalities 
\begin{equation}
	\begin{cases}
	\omega_{L, 0}(0) \leq 1/2,\\
	\omega_{L+1, 0}(0) > 1/2,
	\end{cases}	\qquad \qquad \qquad \qquad \qquad \quad
	\begin{cases}
	\omega_{\tilde{L}, 1}(0) \leq 1/2,  \\
	\omega_{\tilde{L}+1, 1}(0) > 1/2 
	\end{cases}
\end{equation}
on account of
\eqref{eqn:phase_gamma:cases1}. 
Therefore,  Proposition~\ref{thm:trapped_entire_sol:every_trapped_entire_solution_is_a_traveling_wave}  
shows that the bounds
$$\begin{aligned}
L \leq \theta_0 \leq L + 1,
\qquad \qquad
\tilde{L} \leq \theta_0 \leq \tilde{L}+1
\end{aligned}$$
hold for some $\theta_0\in \R $.
This allows us to conclude that $ |L- \tilde{L}|\leq 1$ and
obtain the contradiction $ |i_k - \tilde{i}_k|\leq 1 $.
\end{proof}
\begin{proof}[Proof of Proposition~\ref{prp:zero_level_surface:smallnes of derivative of gamma}]
	Assume to the contrary that there exists $ \delta > 0 $ 
	together with subsequences $(j_k) \in \Z$ and 
	$ T\leq t_1<t_2<\dots \to \infty $
	for which
	\begin{equation}\label{eqn:phase_gamma:assumption_gamma_not_flat}
    	|\gamma_{j_k+1}(t_k) - \gamma_{j_k}(t_k)| \geq \delta.	    
	\end{equation}
	We now claim that it is possible to pass to a subsequence that has $ i(j_k+1,t_k) \neq i(j_k, t_k)  $. Indeed, if actually
	$ i(j_k+1,t_k) = i(j_k, t_k) = i_k$ holds for all large $ k $,
	then we can use  Lemma~\ref{lemma:phase_gamma:decay of y-derivatives}
	to obtain the contradiction
	$$\begin{aligned}
	\delta \leq |\gamma_{j_k+1}(t_k) - \gamma_{j_k}(t_k)| &=  |\Phi^{-1}(u_{i_k, j_k}) - \Phi^{-1}(u_{i_k, j_k+1})| \\ & \leq C|u_{i_k, j_k+1} - u_{i_k,j_k}| \to 0   \quad \text{ as } k\to \infty.
	\end{aligned}$$
     In particular, Lemma \ref{lemma:phase_gamma:i(j+1)-i(j)} allows us to assume that $ i(j_k+1, t_k) = i(j_k, t_k) +1 $ without loss of generality. 
     Using the shorthand $ i_k =i(j_k, t_k)  $, we find
	$$|\gamma_{j_k+1}(t) - \gamma_{j_k}(t)| = |1 + \Phi^{-1}(u_{i_k, j_k}) - \Phi^{-1}(u_{i_k+1, j_k+1})|, $$
	together with the inequalities
	\vspace{-\baselineskip}
	\begin{multicols}{2}
		\begin{equation*}
		\begin{cases}
		u_{i_k,j_k}(t_k) \leq 1/2, \\
		u_{i_k+1,j_k}(t_k) > 1/2, \\ 
		\end{cases}	
		\end{equation*} \
		\begin{equation*} 
		\begin{cases}
		u_{i_k+1,j_k+1}(t) \leq  1/2,\\
		u_{i_k+2,j_k+1}(t) > 1/2.
		\end{cases}\end{equation*}
	\end{multicols}
	
	We now proceed in a similar fashion as in the proof of Lemma \ref{lemma:phase_gamma:i(j+1)-i(j)}. In particular, we may 
	assume that $i_k  + \lceil c t_k\rceil = L \in \Z$ 
	and use Proposition~\ref{lemma:omega_limit_points:construction of limit points}
	to construct an $\omega$-limit point $\omega$ for $u$ that satisfies
    the inequalities 
	\vspace{-\baselineskip}
	\begin{multicols}{2}
		\begin{equation*}
		\begin{cases}
	    \omega_{L,0}(0)\leq 1/2, \\
		\omega_{ L+1,0}(0) \geq 1/2, \\ 
		\end{cases}	
		\end{equation*} \
		\begin{equation*} 
		\begin{cases}
		\omega_{L+1,1}(0) \leq  1/2,\\
		\omega_{L+2,1}(0) \geq 1/2.
		\end{cases}\end{equation*}
	\end{multicols} 
	\noindent 
	Again, Proposition~\ref{thm:trapped_entire_sol:every_trapped_entire_solution_is_a_traveling_wave} implies that
	$ \omega_{i,j}(t) = \Phi(i-ct - x_0 ) $, for some $ x_0\in \R $.
	The independence with respect to $j$ implies that 
	$\omega_{L+1,0}(0) = \omega_{L+1, 1}(0) = \frac{1}{2}$ and consequently
	$ x_0=L+1 $. In particular, we find
	\begin{equation}
	    u_{i_k, j_k}(t_k) \to \omega_{L,0}(0) = \Phi(-1),
	    \qquad
	    u_{i_k+1, j_k+1}(t_k) \to \omega_{L+1,1}(0) = \Phi(0),
	\end{equation}
	and hence
	$$|\gamma_{j_k+1}(t) - \gamma_{j_k}(t)| \to |1 + \Phi^{-1}\big(\Phi(-1)\big) - \Phi^{-1}\big(\Phi(0)\big)| = 0 $$
	as $k \to \infty$, which leads to the desired contradiction
	with \eqref{eqn:phase_gamma:assumption_gamma_not_flat}.

\end{proof}
\section{Discrete heat equation}
\label{sec:dht}

In this section we obtain
several preliminary estimates for the Cauchy problem
\begin{numcases}{}
\dot{h}_j(t)=h_{j+1}(t)  + h_{j-1}(t) - 2h_j(t) , \label{eqn:discrete_heat_eq:DHK} \\
h_j(0)=h_j^0\label{eqn:discrete_heat_eq:initial_cond_for_dhk}
\end{numcases}
associated to the discrete heat equation. These estimates will underpin our analysis of the discrete curvature flow, using a nonlinear Cole-Hopf transformation to pass to 
a suitable intermediate system.

To set the stage,
we recall the well-known fact that the one-dimensional continuous heat equation
\begin{equation}
\begin{cases}
    H_t = H_{yy}, \quad &y\in \R, \ t>0,\\
    H(y,0) = H_0(y),  &y\in \R,     
\end{cases}
\end{equation}
admits the explicit solution
\begin{equation}\label{sol of heat eq}
    H(y,t) = \dfrac{1}{\sqrt{4\pi t}} \int_\R e^{-\frac{(y-x)^2}{4t}} H_0(x) dx. 
\end{equation}
Taking derivatives,
one readily obtains the estimates
\begin{align}
\sup_{y\in \R} |H_y(y,t)| &\leq \min\{C\norm{H_0}_{L^\infty} t^{-\frac{1}{2}}, \norm{H_{0, y}}_{L^\infty}\},\label{derivative of cont sol} \\
\sup_{y\in \R} |H_{yy}(y,t)| &\leq \min\{C\norm{H_0}_{L^\infty} t^{-1}, \norm{H_{0, yy}}_{L^\infty}\}. \label{eqn:discrete_heat_eq:laplace of cont sol}
\end{align}  

The main result of this section  transfers these estimates
to the discrete setting
\eqref{eqn:discrete_heat_eq:DHK}.
This generalization is actually surprisingly delicate, caused by the fact that supremum norms cannot be readily transferred to Fourier space.

\begin{prop}\label{prp:decay_estimates:decay_estimates_of_discerete_heat}
There exists a constant $K > 0$ so that for any $h^0 \in \ell^\infty(\Z)$,
the solution $h \in C^1\big([0,\infty);\ell^\infty(\Z) \big)$ to the initial value problem \eqref{eqn:discrete_heat_eq:DHK} satisfies
the first-difference bound
\begin{equation}\label{delta+ v}
		\norm{\partial^+h(t)}_{\ell^\infty} \leq \min\left\{\norm{\partial^+ h^0}_{\ell^\infty}, K \norm{h^0}_{\ell^\infty} \dfrac{1}{\sqrt{t}} \right\},
		\end{equation}
together with the second-difference estimate
\begin{equation}\label{Laplace v}
		\norm{\partial^{(2)} h(t)}_{\ell^\infty} \leq \min\left\{\norm{\partial^{(2)} h^0}_{\ell^\infty}, K \norm{h^0}_{\ell^\infty} \dfrac{1}{t} \right\}
		\end{equation}
for all $t > 0$.
%
%
%
\end{prop}

Using a suitable Cole-Hopf transformation
the linear heat equation
\eqref{eqn:discrete_heat_eq:DHK}
can be transformed
to the nonlinear initial 
value problem 
\begin{equation}\label{eq:ht:nonl:exp}
 \left\{
    \begin{array}{rl}
\dot {V} =& \dfrac{1}{d}\left(e^{d\partial^+ V} - 2 + e^{-d\partial^- V} \right) + c, \quad t> 0 \\[0.2cm]  
	V(0) =& V^0,\\  
\end{array} \right.
\end{equation}
which will serve as a useful proxy for the discrete curvature flow.
In order to exploit the fact that this equation is invariant under spatially homogeneous perturbations, we introduce the deviation seminorm
\begin{equation}\label{eqn:discrete_heat_eq:seminorm_dev}
    [V]_{\mathrm{dev}}:= \norm{V - V_0}_{\ell^\infty} 
\end{equation}
for sequences $V \in \ell^\infty(\Z)$.

\begin{cor}\label{cor:interface_evolution:cole hopf}
Fix two constants $c$, $d \in \R$ with $d\neq 0$. Then there exist  positive constants $M_{\mathrm{ht}}$ and $\kappa$  so that 
for any $V^0 \in \ell^\infty(\Z^2)$, the solution $V: \left[0, \infty \right) \to \ell^\infty(\Z^2)$ to the initial value problem
\eqref{eq:ht:nonl:exp}
satisfies the estimates
\begin{align}
 	\norm{\partial^+ V(t)}_{\infty} &\leq M_{\mathrm{ht}} e^{\kappa[V^0]_{\mathrm{dev}}} \min\left\{\norm{\partial^+ V^0}_{\ell^\infty}, \dfrac{1}{\sqrt{t}}\right\}, \label{eqn:interface_evolution:cole hopf:eqn1} \\
 	\norm{\partial^{(2)} V(t)}_{\infty} &\leq M_{\mathrm{ht}} e^{\kappa[V^0]_{\mathrm{dev}}} \min \left\{\norm{\partial^+ V^0}_{\ell^\infty} , \dfrac{1}{t} \right\} \label{eqn:interface_evolution:cole hopf:eqn2}.
\end{align}
\end{cor}

\subsection{Discrete heat kernel}

The discrete heat kernel $G: [0, \infty) \to \ell^\infty(\Z)$ is the fundamental solution  of the discrete heat equation, in the sense that 
the function $h=G$ satisfies~\eqref{eqn:discrete_heat_eq:DHK}-\eqref{eqn:discrete_heat_eq:initial_cond_for_dhk} with the initial condition $$h^0_0 = 1 \quad \text{and} \quad h_j^0 = 0 \text{ for } j\neq 0.$$
We now recall the characterization
\begin{equation}\label{eqn:decay_estimates:integral_characterization}
    I_k (t) = \dfrac{1}{\pi} \int_0^{\pi} e^{t \cos \omega} \cos (k \omega) d\omega,
    \qquad
    \qquad
    k \in \Z,
\end{equation}
for the family of modified Bessel 
functions of the first kind;
see e.g. the classical work
by Watson  \cite{watson}.
By passing to the Fourier domain,
one can readily confirm the 
well-known identity
\begin{equation}\label{eqn:discrete_heat_eqn:discrete_heat_kernel}
    G_j(t) = e^{-2t}I_j(2t).
\end{equation}

We may now formally write
\begin{equation}\label{eqn:DHK:formula_for_solution_of_DHK}
    h_j(t) =
    \sum_{k \in \Z} G_k(t) h^0_{j-k}
    =
    e^{-2t} \sum_{k \in \Z}  I_{k}(2t)  h^0_{j-k}
\end{equation}
for the solution
to the general initial value problem
\eqref{eqn:discrete_heat_eq:DHK}-\eqref{eqn:discrete_heat_eq:initial_cond_for_dhk}. In order to see
that this is well-defined 
for $h^0 \in \ell^\infty(\Z)$,
one can use the generating function
\begin{equation}\label{eqn:bessel_functions:generating_function}
    e^{\frac{t}{2}(x+ x^{-1})} = \sum_{k=-\infty}^{\infty} I_k(t) x^n
\end{equation}
together with the bound $I_k(t) \ge 0$ to conclude that $G(t) \in \ell^1(\Z)$. Further useful
properties of the functions $I_k$
can be found in the result below.

\begin{lemma}\label{lemma:Bessel_functions:bounds}
There exists a constant $C > 0$
so that for any integer $k \ge 0$
we have the bound
\begin{equation}\label{eqn:Bessel_functions:asymptotics_of_I_n}
I_k(t) \leq C \dfrac{e^t}{\sqrt{t}}, \qquad t>0 ,
\end{equation}
together with 
\begin{equation}\label{eqn:Bessel_functions:asymptotics of delta I_n}
0 < I_{k}(t) - I_{k+1}(t) \leq  C \dfrac{e^{t}}{t}\qquad t>0.
\end{equation}
\end{lemma}
\begin{proof}
The proof of \eqref{eqn:Bessel_functions:asymptotics_of_I_n} can be found in \cite{watson}, while the lower bound in \eqref{eqn:Bessel_functions:asymptotics of delta I_n} is established in \cite{Soni};
see also \cite[Eq. (16)]{amos1974computation}. Turning
to the upper bound
in \eqref{eqn:Bessel_functions:asymptotics of delta I_n},
we remark that
$ \cos \omega$ is negative for $ \omega \in (\pi/2, \pi )  $, which allows us to write
\begin{align*}
	I_{k}(t) - I_{k+1}(t) &= \dfrac{2}{\pi}\int_{0}^{\pi}e^{t\cos \omega} \sin(\dfrac{2k+1}{2}\omega)\sin(\dfrac{ \omega}{2}) d\omega \\
	&= \dfrac{2}{\pi}\int_{0}^{\pi/2}e^{t\cos \omega} \sin(\dfrac{2k+1}{2}\omega)\sin(\dfrac{ \omega}{2}) d\omega + O(1)
\end{align*}
as $t \to \infty$.
Substituting $ u=2\sqrt{t} \sin(\omega/2) $ we find
\begin{align*}
		I_{k}(t) - I_{k+1}(t) = \dfrac{1}{t\pi}e^t \int_{0}^{\sqrt{2t}} e^{-u^2/2} \sin\left(\left(2k+1\right)\sin^{-1}\left(\dfrac{u}{2\sqrt{t}}\right)\right) \dfrac{u}{\sqrt{1-\dfrac{u^2}{4t}}} du + O(1)
\end{align*}
as $t \to \infty$. The desired bound now follows from the fact that 
the integral can be uniformly bounded
in $t$ and $k$.
\end{proof}

In order to obtain the bounds
in Proposition 
\ref{prp:decay_estimates:decay_estimates_of_discerete_heat},
the convolution \eqref{eqn:DHK:formula_for_solution_of_DHK}
indicates that we need to
control the $\ell^1$-norm of the
first and second differences of $G$.
The following two results
provide the crucial ingredients
to achieve this, exploiting
telescoping sums. To our surprise,
we were unable to find these
bounds directly in the literature.

\begin{lemma}\label{lemma:bessel_functions:series_of_discrete_derivatives}
There exists a constant $C>0$ so that the bound
        \begin{equation}\label{eqn:bessel_functions:discrete_derivative_series}
        \sum_{k\in \Z}\left|I_{k+1}(t) - I_k(t)\right|  \leq C\dfrac{e^t}{\sqrt{t}} 
        \end{equation}
holds for all $t > 0$.      
\end{lemma}
\begin{proof}
  We first note that the characterization \eqref{eqn:decay_estimates:integral_characterization} implies $I_k(z)=I_{-k}(z)$ for all $k\in \Z$. Using \eqref{eqn:Bessel_functions:asymptotics of delta I_n}, we can hence
  use a telescoping series to compute
  \begin{equation}
\begin{aligned}
\sum_{k\in \Z}\left|I_{k+1}(t) - I_k(t)\right| &= 2 \sum_{k\geq 0}I_{k}(t) - I_{k+1}(t) \\
&=2 I_0(t) - 2\lim_{N\to\infty} I_N(t).
\end{aligned}
\end{equation}
The result now follows 
from \eqref{eqn:Bessel_functions:asymptotics_of_I_n}
together with the limit
$I_N(t) \to 0$ as $N \to \infty$.
\end{proof}

\begin{lemma}\label{lemma:bessel_functions:series}
There exists a constant $C>0$ so that the bound
\begin{equation}\label{eqn:bessel_functions:discrete_laplace_series}
        \sum_{k\in \Z}\left|I_{k+1}(t) - 2I_k(t)+ I_{k-1}(t) \right| \leq C\dfrac{e^t}{t} 
    \end{equation}
holds for all $t > 0$.
\end{lemma}
\begin{proof}
We claim that for every $t>0$ the function 
\begin{equation}
\Z_{\ge 0} \ni k\mapsto \nu^{(2)}_k(t) :=  I_{k+1}(t) - 2I_k(t) + I_{k-1}(t)
\end{equation}
changes sign exactly once. Note that this allows us to obtain
the desired bound \eqref{eqn:bessel_functions:discrete_laplace_series}
from~\eqref{eqn:Bessel_functions:asymptotics of delta I_n} by applying a telescoping argument 
similar to the one used in the proof
of Lemma \ref{lemma:bessel_functions:series_of_discrete_derivatives}.

Turning to the claim,
we recall the notation
\begin{equation}
    a_k(t) = \dfrac{tI_k'(t)}{I_k(t)}
     = k + t\dfrac{I_{k+1}(t)}{I_k(t)}
\end{equation}
from~\cite{Paltsev1999}
and
use the identity $I_{k+1}(t) + I_{k-1}(t) = 2I_k'(t) $
to compute
\begin{align*}
	\nu^{(2)}_k(t) 
	&= 2I_k'(t) - 2I_k(t)
	=\dfrac{2I_k(t)}{t}\left(a_k(t) - t\right).
	\end{align*}	
The inequality (15) in \cite{neuman1992inequalities} directly implies that $a_{k}(t) < a_{k+1}(t)$, for every $t>0$ and $k\geq 0$. In addition, 
the lower bound in \eqref{eqn:Bessel_functions:asymptotics of delta I_n}
implies that
\begin{equation}
    a_0(t) - t =  t \big(\frac{I_1(t)}{I_0(t)} - 1 \big) < 0,
\end{equation}
while for $k > t$ we easily conclude 
$a_k(t) - t \ge k - t > 0$.
In particular, $k \mapsto a_k(t) -t$  changes sign precisely once. The claim now follows from
the strict positivity $I_k(t) > 0$
for $t >0$ and $k \ge 0$.
\end{proof}

\subsection{Gradient bounds}
Using the representation
\eqref{eqn:DHK:formula_for_solution_of_DHK}
and the 
bounds for the discrete heat kernel obtained above,
we are now ready 
to establish
Proposition \ref{prp:decay_estimates:decay_estimates_of_discerete_heat}
and Corollary
\ref{cor:interface_evolution:cole hopf}.

\begin{proof}[Proof of Proposition~\ref{prp:decay_estimates:decay_estimates_of_discerete_heat}]
In order to establish \eqref{delta+ v}, we apply a discrete
derivative 
to
\eqref{eqn:DHK:formula_for_solution_of_DHK}, which yields 
\begin{equation*}
\begin{aligned}
(\partial^+h)_j(t) 
&=e^{-2t} \sum_{k\in\Z}  \big(I_{k+1}(2t) - I_{k}(2t)\big) h_{j-k}^0 .
\end{aligned}
\end{equation*}
Applying \eqref{eqn:bessel_functions:discrete_derivative_series},
we hence find
\begin{equation}\label{delta}
    \left|(\partial^+h)_j(t)\right| \leq e^{-2t}\norm{h^0}_{\ell^\infty}\sum_{k\in\Z} |I_{k+1}(2t) - I_k(2t)|
    \leq C \norm{h^0}_{\ell^\infty}\dfrac{1}{\sqrt{t}}.
\end{equation}
On the other hand,
the inequality
$\norm{\partial^+ h(t)}_{\ell^\infty} \le 
\norm{ \partial^+ h^0}_{\ell^\infty}$
follows directly from the comparison principle, since
$\partial^+ h$ satisfies the
discrete heat equation
with initial value $\partial^+ h^0$.
The second-order bound
\eqref{Laplace v} can be obtained
in a similar fashion by exploiting
the estimate \eqref{eqn:bessel_functions:discrete_laplace_series}.
\end{proof}

\begin{proof}[Proof of Corollary \ref{cor:interface_evolution:cole hopf}] \normalfont
   Since the function $\tilde{V} = V - V_0^0$ also satisfies
   the first line of \eqref{eq:ht:nonl:exp},
   we may assume without loss
   that $V^0_0 = 0$ and hence $[V^0]_{\mathrm{dev}} = \norm{V^0}_{\ell^\infty} $.
   Upon writing
   \begin{equation}
       h_j(t)=e^{d\left(V_j(t) - c t\right)},
   \end{equation}
   straightforward calculations
   show that $h$ satisfies
   \eqref{eqn:discrete_heat_eq:DHK}
   with the initial condition
   \begin{equation}
       h_j(0) = e^{d V_{j}^0},
   \end{equation}
   which using the comparison principle
   implies that
   \begin{equation}
     \label{eq:dh:lw:bnd:hj}
       h_j(t) \ge e^{- |d| \norm{V^0}_{\infty}}, \qquad \qquad t \ge 0.
   \end{equation}
   For any $j \in \Z$,
   the intermediate value theorem
   allows us to find
   $ h^*_1, h_2^*\in [h_j, h_{j+1}]$ and  $ h_3^* \in [h_{j-1}, h_{j}]$
   for which we have
 	 \begin{align}\label{eqn:discrete_heat_eq:d+V_in_terms_of_d+h}
 	    &\partial^+ V_j = \dfrac{1}{d }\dfrac{ \partial^+ h_j}{h^*_1}, 
 	    \quad \qquad
 	    \partial^{(2)} V_j =  \dfrac{1}{d} \left(\dfrac{\partial^{(2)} h_j}{h_j}  - \dfrac{\left( \partial^+ h_j\right)^2}{2(h_2^*)^2}- \dfrac{\left( \partial^- h_j\right)^2}{2(h_3^*)^2}\right).
 	 \end{align}
 	 In particular, \eqref{eq:dh:lw:bnd:hj}
 	 yields the bounds
         \begin{align}
    \norm{\partial^+ V}_{\ell^\infty} &\leq\dfrac{1}{|d|}e^{|d|\norm{V^0}_{\ell^\infty}} \norm{\partial^+ h}_{\ell^\infty},\\
    \norm{\partial^{(2)} V}_{\ell^\infty} &\leq\dfrac{1}{|d|}\left(e^{|d|\norm{V^0}_{\ell^\infty}} \norm{\partial^{(2)} h}_{\ell^\infty} + e^{2|d|\norm{V^0}_{\ell^\infty}}\norm{\partial^+ h}_{\ell^\infty}^2\right). 
     \end{align}
     In a similar fashion, we obtain
     \begin{equation}\label{eqn:discrete_heat_eq:d+h_in_terms_of_d+V}
 	    \norm{\partial^+ h^0}_{\ell^\infty} \leq |d| e^{|d|\norm{V^0}_{\ell^\infty}}\norm{ \partial^+ V^0}_{\ell^\infty}.
 \end{equation}
 Using $\norm{\partial^{(2)} h^0}_{\ell^\infty} \le 2 \norm{ \partial^+ h^0}_{\ell^\infty}$, the desired estimates
 \eqref{eqn:interface_evolution:cole hopf:eqn1}-\eqref{eqn:interface_evolution:cole hopf:eqn2}
 can now be established by 
 applying Proposition \ref{prp:decay_estimates:decay_estimates_of_discerete_heat}.
  \end{proof}

\section{Construction of super- and sub-solutions}
\label{sec:sub:sup}
In this section we construct 
refined sub- and super-solutions of \eqref{eqn:main_results:discrete AC} 
that use the solution $V$
of the nonlinear system
\eqref{eq:ht:nonl:exp} as a type
of phase. In particular, we add
a transverse $j$-dependence
to the planar sub- and super-solutions
\eqref{eqn:analysis_of_u:supersolution}-\eqref{eqn:analysis_of_u:subsolution},
which requires some substantial modifications
to account for the slowly-decaying resonances
that arise in the residuals.

As a preparation, we introduce
the linear operator
$\mathcal{L}_{\mathrm{tw}}: H^1 \to L^2$
associated to the linearization
of the travelling wave MFDE
\eqref{eqn:main_results:MFDE}, which acts 
as
\begin{equation}
\left(\mathcal{L}_{\mathrm{tw}}v\right)(\xi) = cv'(\xi) + v(\xi+1) -2 v(\xi) + v(\xi-1) + g'\big(\Phi(\xi)\big)v(\xi).
\end{equation}
In addition, we introduce the formal adjoint
$\mathcal{L}^{\mathrm{adj}}_{\mathrm{tw}}: H^1 \to L^2$ that acts as
\begin{equation}
    \left(\mathcal{L}^{\mathrm{adj}}_{\mathrm{tw}} w \right)(\xi) = - c w'(\xi) + w(\xi+1) -2 w(\xi) + w(\xi-1) + g'\big(\Phi(\xi)\big)w(\xi).
\end{equation}
In view of the requirement $c \neq 0$ in
$(H\Phi)$, the results in \cite{Mallet-Paret1999}
show that there exists a strictly positive function $\psi \in C^1(\R,\R)$ for which we have
\begin{equation}
\label{eq:subsup:char:ker:range}
    \mathrm{Ker}\, \mathcal{L}^{\mathrm{adj}}_{\mathrm{tw}}
    = \mathrm{span} \{ \psi \},
    \qquad \qquad
    \mathrm{Range}\, \mathcal{L}_{\mathrm{tw}}
    = \{ f \in L^2: \langle \psi, f \rangle = 0 \},
\end{equation}
together with the normalization
$\langle \psi, \Phi_*' \rangle = 1$.

We now fix the parameter $d$
in the LDE \eqref{eq:ht:nonl:exp}
by writing
\begin{equation}\label{eqn:super_and_sub_sols2:d}
	d=-\langle \Phi'', \psi\rangle.
\end{equation}
The characterization \eqref{eq:subsup:char:ker:range}
implies that
we can find a 
solution $r \in H^1$
to the MFDE
\begin{equation}
\label{eq:sub:def:r}
    \mathcal{L}_{\mathrm{tw}} r  + d \Phi' = - \Phi''
\end{equation}
that becomes unique upon imposing
the normalization $\langle \psi, r \rangle = 0$.
Multiplying this residual function by the square gradients 
\begin{equation}
[\alpha_V]_j = [\beta_V]^2_j -1  = \dfrac{(V_{j+1} - V_j)^2}{2} + \dfrac{(V_{j-1} - V_j)^2}{2}
\end{equation}
gives us the  correction terms we need
to control the resonances discussed above.
In order to account for the possibility that $d = 0$, the actual LDE that we use here 
is given by
\begin{equation}
\label{eq:sub:sup:lde:V}
\dot {V} =
\left\{
    \begin{array}{lcl}
    \dfrac{1}{d}\left(e^{d\partial^+ V} - 2 + e^{-d\partial^- V} \right) + c
      & & d \neq 0, \\[0.3cm]
  \partial^{(2)} V + c, & & d = 0 .
\end{array} \right.
\end{equation}

\begin{prop}\label{prp:construction_of_super_and_subsolutions:supersolution} 
Fix $R > 0$ and suppose that
the assumptions (H$g$) and (H$\Phi$)
both hold. Then for any $\epsilon > 0$, there exist constants $\delta > 0$, $\nu > 0$ and $C^1$-smooth functions
\begin{equation}
    p: [0, \infty) \to \R,
    \qquad
    q: [0, \infty) \to \R
\end{equation}
so that for any $V^0 \in \ell^\infty(\Z)$ with
\begin{equation}
[V^0]_{\mathrm{dev}} < R,
\qquad \qquad
\norm{\partial^+ V^0 }_{\ell^\infty}  < \delta
\end{equation}
the following holds true.
\begin{enumerate}[(i)]
\item Writing  $V: [0, \infty) \to \ell^\infty(\Z)$ for the solution to
\eqref{eq:sub:sup:lde:V}
with the initial condition $V(0) = V^0$,
the functions $u^+$ and $u^-$ defined by
	\begin{equation}
	  \label{eq:sps:def:u:plus:minus}
	  \begin{array}{lcl}
		u^+_{i,j}(t) & := &\Phi\big(i-V_j(t)+ q(t)\big) + r\big(i-V_j(t) + q(t)\big)[\alpha_V]_j + p(t),
	\\[0.2cm]
		u^-_{i,j}(t) &:=&\Phi\big(i-V_j(t)- q(t)\big) + r\big(i-V_j(t) - q(t)\big)[\alpha_V]_j - p(t) 
	\end{array}
	\end{equation}
are a super- respectively sub-solution of \eqref{eqn:main_results:discrete AC}. %
	\item We have $q(0) = 0$ together with the bound $0\leq q(t) \leq \epsilon$ for all $t \ge 0$.
	\item We have the bound $0\leq p(t) \leq \epsilon $ for all $t \ge 0$,
	together with the initial inequality
		\begin{equation}
	  \label{eq:int:sub:sup:bnd:init:p:zero}
	    p(0) - \norm{r}_{L^\infty} \delta^2  > \nu > 0 .
	\end{equation}
	\item The asymptotic behaviour
	$p (t) = O ( t^{-\frac{3}{2}} ) $
	holds for $t \to \infty$.
\end{enumerate}
In addition, the constants $\nu = \nu(\epsilon)$
satisfy $\lim_{\epsilon\downarrow 0} \nu(\epsilon) = 0$.
\end{prop}
In the remainder of this section
we set out to establish this result
for $u^+$,
which requires us to understand
the residual $\mathcal{J}[u^+]$
introduced in
\eqref{eqn:omega_limit_points:residual}.
Upon introducing the notation
\begin{equation}
    \xi_{i,j}(t) = i - V_j(t),
\end{equation}
a short computation allows
us to obtain the splitting
\begin{equation}
	\mathcal{J}[u^+] = \mathcal{J}_{\text{glb}} + \mathcal{J}_\Phi + \mathcal{J}_r,
\end{equation}
in which the two expressions
\begin{equation}
   \label{eq:sps:def:j:phi:j:r}
    \begin{array}{lcl}
    [\mathcal{J}_{\Phi}]_{i,j} &=& -\Phi'\left(\xi_{i,j}  + q\right) \dot{V}_j  - \Phi(\xi_{i, j+1} + q)  - \Phi(\xi_{i+1, j} + q)  \\[0.2cm]
	& &\qquad - \Phi(\xi_{i+1, j} + q)  - \Phi(\xi_{i-1, j} + q) + 4\Phi(\xi_{i, j} + q) - g\big(\Phi(\xi_{i,j}+q)\big), \\[0.2cm]
	[\mathcal{J}_{r}]_{i,j} &= & -r'\left(\xi_{i,j} + q\right) \dot{V}_j [\alpha_V]_j \\[0.2cm]
	& & \qquad - r\left(\xi_{i, j+1} + q\right)[\alpha_V]_{j+1} - r\left(\xi_{i, j-1} + q\right)[\alpha_V]_{j-1} 
	\\[0.2cm]
	& & \qquad 
	  - r\left(\xi_{i+1, j} +q\right)[\alpha_V]_j -  r\left(\xi_{i-1, j} +q\right)[\alpha_V]_j  + 4r\left(\xi_{i, j} +q\right)[\alpha_V]_j
	\\[0.2cm]
	& & \qquad +r\left(\xi_{i,j} + q \right)\big(\partial^+ V_j \partial^+ \dot{V}_j +\partial^- V_j \partial^- \dot{V}_j \big)
    \end{array}
\end{equation}
are naturally related
to the defining equations
for $\Phi$, $r$ and $V$,
while 
\begin{equation}
	\mathcal{J}_{\text{glb}} 
	=  \dot{q}\big(\Phi'(\xi + q) + r'(\xi + q)\alpha_V \big)  - g(u^+) + g\big(\Phi(\xi+q)\big)  + \dot{p} 
\end{equation}
reflects the contributions
associated to the dynamics of $p$ and $q$.

In order to control the
quantities \eqref{eq:sps:def:j:phi:j:r},
we introduce the two simplified expressions \begin{equation}
\begin{array}{lcl}
 \mathcal{J}_{\Phi;\mathrm{apx}}
 & = & - \big(d\Phi'(\xi +q) + \Phi''(\xi + q)\big)\alpha_V ,
\\[0.2cm]
\mathcal{J}_{r;\mathrm{apx}}
  & = & \Big(d\Phi'(\xi + q) +  \Phi''(\xi + q) + g'\big(\Phi(\xi+q)\big)r(\xi+q)  \Big)\alpha_V
\end{array}
\end{equation}
which will turn out to be useful
approximations. Indeed,
the two results below
provide bounds for the associated remainder terms
\begin{equation}
    \mathcal{J}_{\Phi} 
     = \mathcal{J}_{\Phi;\mathrm{apx}} + \mathcal{R}_{\Phi},
\qquad
\qquad
\mathcal{J}_{r} 
     = \mathcal{J}_{r;\mathrm{apx}} + \mathcal{R}_{r}.
\end{equation}

\begin{lemma}\label{lemma:super_and_sub_solutions_for_the_discrete_AC2:RPhi}
Fix $R >0$ and suppose that (H$g$) and (H$\Phi$)
are satisfied.
Then there exists a constant $M > 0$
so that for any $V \in C^1([0,\infty);\ell^\infty)$
that satisfies the LDE~\eqref{eq:sub:sup:lde:V}
with 
$[V(0)]_{\mathrm{dev}} < R$
and any pair of  functions 
$p,q \in C([0,\infty);\R)$,
we have the estimate
\begin{equation}
    \norm{\mathcal{R}_\Phi(t)}_{\ell^\infty} \le M \min\left\{\norm{\partial^+ V(0)}_{\ell^\infty}, t^{-\frac{3}{2}}\right\},
    \qquad 
    \qquad
    t > 0.
\end{equation}
\end{lemma}
\begin{proof}
Expanding $\Phi(\xi_{i, j\pm1} + q)$
to third order
around $\xi_{i, j} + q$
and evaluating
the travelling wave MFDE \eqref{eqn:main_results:MFDE}
at this point, we find
%
\begin{align*}
[\mathcal{J}_{\Phi}]_{i,j} &=  \Phi'\left(\xi_{i,j}  + q\right) \left(- \dot{V}_j + \partial^{(2)}V_j  + c\right)  - \Phi''(\xi_{i, j} + q)[\alpha_V]_j\\  
&\quad  -\dfrac{1}{2}\int_{\xi_{i, j} + q}^{\xi_{i, j+1} + q} \Phi'''(s)(\xi_{i, j+1}+q-s)^2 ds  - \dfrac{1}{2}\int_{\xi_{i, j}+q}^{\xi_{i, j-1}+q} \Phi'''(s)(\xi_{i, j-1}+q-s)^2 ds. 
\end{align*}
Substituting the LDE~\eqref{eq:sub:sup:lde:V}
and expanding $e^{d\partial^+ V}$ 
and $e^{-d\partial^- V}$ 
to third order, we compute
\begin{align*}
[\mathcal{J}_{\Phi}]_{i,j} &=  -d\Phi'\left(\xi_{i,j}  + q\right)[\alpha_V]_j  - \Phi''(\xi_{i, j} + q)[\alpha_V]_j\\  
&\quad  -\dfrac{1}{2}\int_{\xi_{i, j}-q}^{\xi_{i, j+1}-q} \Phi'''(s)(\xi_{i, j+1}+q-s)^2 ds  - \dfrac{1}{2}\int_{\xi_{i, j}+q}^{\xi_{i, j-1}+q} \Phi'''(s)(\xi_{i, j-1}+q-s)^2 ds \\
&\quad - \dfrac{1}{2d}\int_{0}^{d\partial^+V}e^s\left(d\partial^+V - s \right)^2 ds + \dfrac{1}{2d}\int_{-d\partial^-V}^{0}e^s\left(d\partial^-V + s \right)^2ds.
\end{align*}
Since the first line
of this expression
corresponds with 
$\mathcal{J}_{\Phi;\mathrm{apx}}$,
the desired estimate
follows from
Corollary~\ref{cor:interface_evolution:cole hopf}.
\end{proof}

\begin{lemma}\label{lemma:super_and_sub_solutions_for_the_discrete_AC2:Rr}
Fix $R >0$ and suppose that (H$g$) and (H$\Phi$)
are satisfied.
Then there exists a constant $M > 0$
so that for any $V \in C^1([0,\infty);\ell^\infty)$
that satisfies the LDE~\eqref{eq:sub:sup:lde:V}
with 
$[V(0)]_{\mathrm{dev}} < R$
and any pair of  functions 
$p,q \in C([0,\infty);\R)$,
we have the estimate
\begin{equation}\label{eqn:super_and_sub_solutions_for_the_discrete_AC2:residual_r}
    \norm{\mathcal{R}_r(t)}_{\ell^\infty} \le M \min\left\{\norm{\partial^+ V(0)}_{\ell^\infty}, t^{-\frac{3}{2}}\right\},
    \qquad 
    \qquad
    t > 0.
\end{equation}
\end{lemma}
\begin{proof}
Expanding $r(\xi_{i,j+1} + q)$ and $r(\xi_{i,j-1} + q)$ around $\xi_{i,j} + q$
and evaluating \eqref{eq:sub:def:r}
at this point,
we find
\begin{equation}
    \begin{array}{lcl}
[\mathcal{J}_{r}]_{i,j} &= &
r'\left(\xi_{i,j} + q\right)\big(- \dot{V}_j + c \big)[\alpha_V]_j
- [\alpha_V]_{j+1} \int_{\xi_{i,j} +q}^{\xi_{i,j+1} + q} r'(s)ds 
-  [\alpha_V]_{j-1}\int_{\xi_{i,j} +q}^{\xi_{i,j-1} + q} r'(s)ds
\\[0.2cm]
& & \qquad  - r\left(\xi_{i,j} + q \right)
\Big( [\alpha_V]_{j+1} + [\alpha_V]_{j-1} - 2 [\alpha_V]_j \Big) \\[0.2cm]
& &\qquad
+r\left(\xi_{i,j} + q \right)\big(\partial^+ V_j \partial^+ \dot{V}_j +\partial^- V_j \partial^- \dot{V}_j 
 \big)  
 \\[0.2cm]
 & & \qquad 
 + \Big(d\Phi'(\xi_{i,j} + q) + \Phi''(\xi_{i,j} + q) +  g'\big(\Phi(\xi_{i,j}+q)\big)r(\xi_{i,j}+q) \Big)[\alpha_V]_j .
 \end{array}
\end{equation}
 In order to estimate the terms in the second line above, 
 we compute
\begin{equation}\label{eqn:super_and_sub_solutions2:dy(dy2)}
	\begin{aligned}
	{[}\alpha_V]_{j+1} - [\alpha_V]_{j} &= \dfrac{1}{2}\left(V_{j+2} - V_{j+1} + V_{j} - V_{j-1}\right)\left(V_{j+2} - V_{j+1} - V_{j} + V_{j-1}\right) \\
	&= \dfrac{1}{2}\left(V_{j+2} - V_{j+1} + V_{j} - V_{j-1}\right)\left(\partial^{(2)}V_{j+1} + \partial^{(2)}V_{j} \right),
	\end{aligned}
\end{equation}
which can be thought of as a
discrete analogue of the identity
$\partial_y(\partial_{y}^2) = 2\partial_y\partial_{yy} $. 
Substituting~\eqref{eq:sub:sup:lde:V}
and expanding $e^{d\partial^+ V}$ and $e^{-d\partial^- V}$ up to second order,
we can again apply
Corollary~\ref{cor:interface_evolution:cole hopf}
to obtain the desired estimate.
%
%
\end{proof}

We are now ready to introduce
our final approximation
\begin{equation}
    \label{eq:sub:splt:j}
\mathcal{J} = \mathcal{J}_{\mathrm{apx}} + \mathcal{R}
\end{equation}
by writing
\begin{equation}\label{eqn:super_and_sub_solutions_for_the_discrete_AC2:Japx}
\begin{array}{lcl}
    \mathcal{J}_{\mathrm{apx}} 
     & = & 	  \dot{q}\big(\Phi'(\xi + q)  + r'(\xi + q) \alpha_V \big) + \dot{p}	\\[0.2cm]
			& & \qquad
			-p \int_0^1 g'\Big(\Phi(\xi + q) + \tau\big(p + r(\xi +q)\alpha_V\big)\Big) d\tau 
	\\[0.2cm]
	& & \qquad
	  -p r(\xi+q) \alpha_V \int_0^1 \int_0^\tau g''\Big(\Phi(\xi + q)
	   + s\big(p + r(\xi + q)\alpha_V\big) \Big)
	   d s \, d\tau.
\end{array}
\end{equation}
We show below that
the residual $\mathcal{R}$ satisfies
the same bound as $\mathcal{R}_{\Phi}$
and $\mathcal{R}_{r}$.
This will allow us to construct
appropriate functions $p$ and $q$
and establish
Proposition \ref{prp:construction_of_super_and_subsolutions:supersolution}.

\begin{lemma}\label{lemma:super_and_sub_solutions_for_the_discrete_AC2:R}
Fix $R >0$ and suppose that (H$g$) and (H$\Phi$)
are satisfied.
Then there exists a constant $M > 0$
so that for any $V \in C^1([0,\infty);\ell^\infty)$
that satisfies the LDE~\eqref{eq:sub:sup:lde:V}
with 
$[V(0)]_{\mathrm{dev}} < R$
and any pair of  functions 
$p,q \in C([0,\infty);\R)$,
we have the estimate
\begin{equation}\label{eqn:super_and_sub_solutions_for_the_discrete_AC2:residual}
    \norm{\mathcal{R}(t)}_{\ell^\infty} \le M \min\left\{\norm{\partial^+ V(0)}_{\ell^\infty}, t^{-\frac{3}{2}}\right\},
    \qquad \qquad
    t > 0.
\end{equation}
\end{lemma}
\begin{proof}
Writing
\begin{equation}
    \mathcal{J}_{\mathrm{apx};I}
    = \mathcal{J}_{\mathrm{glb}}  + \mathcal{J}_{\Phi;\mathrm{apx}} 
 + \mathcal{J}_{r;\mathrm{apx}},
\end{equation}
together with
\begin{equation}
    \mathcal{I}_g = g\big(\Phi(\xi+q)\big)- g(u^+) + g'\big(\Phi(\xi+q)\big)r(\xi+q)\alpha_V  ,
\end{equation}
we have 
\begin{equation}
\mathcal{J}_{\mathrm{apx};I} 
=\dot{q}\big(\Phi'(\xi + q) + r'(\xi + q)\alpha_V \big) + \dot{p} + \mathcal{I}_g .
\end{equation}
Upon rewriting $\mathcal{I}_g $ in the form
\begin{equation}
    \begin{array}{lcl}
        \mathcal{I}_g   & = & -\big(p + r(\xi + q)\alpha_V\big)\int_0^1 g'\Big(\Phi(\xi + q) + \tau\big(p + r(\xi+q)\alpha_V\big)\Big)d\tau
        \\[0.2cm]
        & & \qquad 
          + g'\big(\Phi(\xi + q)\big)r(\xi+q)\alpha_V \\ [0.2cm] 
         & = & -p \int_0^1 g'\Big(\Phi(\xi + q) + \tau\big(p + r(\xi + q)\alpha_V\big)\Big) d\tau   \\ [0.2cm]
         &  & \quad -\alpha_V r(\xi+q)\big(p+r(\xi+q)\alpha_V\big)\int_0^1 \int_0^\tau g''\Big(\Phi + s\big(p+r(\xi+q)\alpha_V\big)\Big)ds \, d\tau ,
    \end{array}
\end{equation}
we obtain
the splitting \eqref{eq:sub:splt:j}
with the residual
\begin{equation}
\begin{array}{lcl}
\mathcal{R} &= &\mathcal{R}_\Phi + \mathcal{R}_r - r(\xi+q)^2\alpha_V^2\int_0^1 \int_0^\tau g''\Big(\Phi + s\big(p+r(\xi + q)\alpha_V\big)\Big)ds \, d\tau . \\[0.2cm]
\end{array}
\end{equation}
As before, the desired bound now
follows from Corollary
\ref{cor:interface_evolution:cole hopf}.
\end{proof}

\begin{proof}[Proof of Proposition \ref{prp:construction_of_super_and_subsolutions:supersolution}] 
Without loss, we assume
that the constant $M$
from Lemma~\ref{lemma:super_and_sub_solutions_for_the_discrete_AC2:R}
satisfies 
\begin{equation}
    M\geq\max\{1,  \norm{r}_{L^\infty}, \norm{r'}_{L^\infty},  \sup_{-1 \le s \le 2} |g'(s)|, \sup_{-1 \le s \le 2} |g''(s)|,
    M_{\mathrm{ht}} e^{\kappa R}\}.
\end{equation}
We first pick a constant $m \in ( 0, 1] $ in such a way that
$$-g'(s) \geq 2m > 0, \text{ for } s\in [-\epsilon, 3\epsilon] \cup [1-2\epsilon, 1+2\epsilon],$$
reducing $\epsilon$ if needed. Next, we define the positive constants
$$C_{\epsilon} := \max\{1, \dfrac{2m + 
M
}{\min_{\Phi\in [\epsilon, 1-\epsilon]}\Phi'}\}, \quad \qquad \delta_{\epsilon} := \dfrac{\epsilon^3m^3}{6^3 M^3 C_{\epsilon}^3},
\qquad \qquad
\nu_{\epsilon} :=
\dfrac{\epsilon^3 m^2 }{2 \cdot 6^3 M^2C_{\epsilon}^3} = \dfrac{M\delta_{\epsilon}}{2m}
$$
together with the positive function
\begin{equation}
    \begin{array}{lcl}
     K_{\epsilon}:[0,\infty)\to \R,
& & t \mapsto  M \min\left\{\delta_{\epsilon}, t^{-\frac{3}{2}} \right\}.
    \end{array}
\end{equation}

We now choose functions $p, q\in C^\infty\left[0, \infty\right)$ that satisfy 
$$K_{\epsilon}(t)\leq m p(t) \leq 2K_{\epsilon}(t), \quad m|\dot{p}(t)|\leq 2\tilde{K}_{\epsilon}(t), \quad q(t) = C_{\epsilon}\int_0^t p(s) ds, $$ where $\tilde{K}_{\epsilon}$ is defined by
\begin{equation}
    \tilde{K}_{\epsilon}(t) = \begin{cases}
    0, \quad &t\leq \delta_{\epsilon}^{-\frac{2}{3}} \\
    \frac{3}{2}M t^{-\frac{5}{2}},  &t> \delta_{\epsilon}^{-\frac{2}{3}},
    \end{cases}
\end{equation}
which we recognize as the absolute value of weak derivative of the function $K_{\epsilon}$. 
The functions $p$ and $q$ are clearly nonnegative, with
$$p(0) - \norm{r}_{L^\infty} \delta_{\epsilon}^2\geq \dfrac{M\delta_{\epsilon}}{m} - M\delta_{\epsilon}^2 
= \dfrac{M\delta_{\epsilon}}{m}\left(1 -\delta_{\epsilon} m \right) \geq \dfrac{M\delta_{\epsilon}}{m}\left(1-\dfrac{1}{6^3} \right) > \nu_{\epsilon}.$$
Furthermore, we have $p(t) \leq \dfrac{2M\delta_{\epsilon}}{m}\leq \epsilon$,
together with
$$q(t)\leq \dfrac{2C_{\epsilon}}{m}\int_0^\infty K_{\epsilon}(s) ds \leq \dfrac{6C_{\epsilon}}{m} M \delta_{\epsilon}^{\frac{1}{3}} = \epsilon.$$ In particular, items (ii)-(iv) are satisfied.
In addition,
using $|\alpha_V| \le M^2 \delta_{\epsilon}^2$
we obtain the a-priori bound
\begin{equation}
\label{eq:sub:sup:a:priori:bnd}
     |p(t) + r\big(\xi_{ij}(t) + q(t) \big) [\alpha_V]_j(t) |
    \le \frac{2M \delta_{\epsilon}}{m} + M^3 \delta_{\epsilon}^2
    \le \epsilon
\end{equation}
for all $t \ge 0$ and $(i,j) \in \Z^2$.

Turning to (i),
Lemma~\ref{lemma:super_and_sub_solutions_for_the_discrete_AC2:R}
implies that it suffices
to show that the approximate
residual \eqref{eqn:super_and_sub_solutions_for_the_discrete_AC2:Japx}
satisfies
$\mathcal{J}_{\mathrm{apx}} \geq K_{\epsilon}(t)$.
Introducing the notation
\begin{equation}
   \mathcal{I}_A = \frac{\dot{q}}{p} \Phi'(\xi + q),
   \qquad \qquad
   \mathcal{I}_B = \frac{\dot{q}}{p} r'(\xi + q) \alpha_V,
   \qquad \qquad
    \mathcal{I}_C = \frac{\dot{p}}{p},
\end{equation}
together with the integral expressions
\begin{equation}
    \begin{array}{lcl}
     \mathcal{I}_D & = & 
     -\int_0^1 g'\Big(\Phi(\xi + q) + \tau\big(p + r(\xi +q)\alpha_V\big)\Big) d\tau 
     \\[0.2cm]
     \mathcal{I}_E & = & 
     -r(\xi +q) \alpha_V
         \int_0^1 \int_0^\tau g''\Big(\Phi(\xi + q)
	   + s\big(p + r(\xi + q)\alpha_V\big) \Big)
	   d s \, d\tau,
	\end{array}
\end{equation}
we see that
\begin{equation}
    \mathcal{J}_{\mathrm{apx}}
      =  
     p \big( \mathcal{I}_A 
      + \mathcal{I}_B + \mathcal{I}_C + \mathcal{I}_D + \mathcal{I}_E \big).
\end{equation}

Using the observation 
\begin{equation}
    \frac{|\dot{p}(t)|}{p(t)} 
    \le \begin{cases} 0, \quad &t\leq \delta_{\epsilon}^{-\frac{2}{3}} \\
    3t^{-1}, &t>\delta_{\epsilon}^{-\frac{2}{3}},
    \end{cases}
  \end{equation}
we obtain the global bounds
\begin{equation}
     \begin{array}{lclcl}
   |\mathcal{I}_B|
       & \leq & C_{\epsilon}  M^3 \delta_{\epsilon}^2
       & \leq & \dfrac{m}{3} ,
     \\[0.2cm]
     |\mathcal{I}_C|  &\leq& 3\delta_{\epsilon}^{\frac{2}{3}} &\leq& \dfrac{m}{3} ,
     \\[0.2cm]
     \big|\mathcal{I}_E \big|
	    & \le & M^2 \delta_{\epsilon} ^2
       & \le & \dfrac{m}{3}.
     \end{array}
 \end{equation}

 When  $\Phi(\xi +q) \in (0,\epsilon] \cup[1-\epsilon, 1) $, we
 may use \eqref{eq:sub:sup:a:priori:bnd} to obtain the lower bound
 \begin{equation}
    \mathcal{I}_D
     \ge  2m.
 \end{equation}
 Together with $\mathcal{I}_A \ge 0$,
 this allows us
 to conclude
 \begin{equation}
 \label{eq:sub:sup:fin:bnd:j:apx}
     \mathcal{J}_{\mathrm{apx}}  \ge
    mp(t) \geq K_{\epsilon}(t).
 \end{equation}
 On the other hand,
 when 
 $\Phi(\xi + q) \in [ \epsilon, 1-\epsilon]$, we 
 have 
 \begin{equation}
     | \mathcal{I}_A | 
     \ge
     C_{\epsilon} \dfrac{2m + M}{C_{\epsilon}}
     \ge 2m + M,
     \qquad \qquad
     | \mathcal{I}_D | 
     \le M,
 \end{equation}
 which again yields \eqref{eq:sub:sup:fin:bnd:j:apx}.
 \end{proof}

\section{Phase approximation}
\label{sec:asymp}

In this section we discuss the relation between the interface $\gamma$ defined in \eqref{eqn:main_results:def_of_gamma}, solutions of the
discrete mean curvature flow
\begin{equation}\label{eqn:approximation_of_MCF:discrete_mean_curvature_eqn}
\dot{\Gamma} = \dfrac{\partial^{(2)}\Gamma}{\beta_\Gamma^2} + 2d \beta_\Gamma + c-2d, \end{equation}
and solutions of the 
(nonlinear) heat LDE
\eqref{eq:sub:sup:lde:V},
both with $d = - \langle \Phi'', \psi \rangle$.
In particular, we establish 
Theorem \ref{thm:main_results:approx_of_gamma} in two main steps. 

The first step is to show that $\gamma$ can be well-approximated by $V$ after allowing sufficient time
for the interface to 'flatten'.
This is achieved using
the sub- and super-solutions
constructed in {\S}\ref{sec:sub:sup}.

\begin{prop}\label{prop:interface_evolution:approx_of_gamma_by_V}
Assume that (H$g$), (H$\Phi$) and (H0) all hold and let $u$ be a solution of \eqref{eqn:main_results:discrete AC} with the initial condition \eqref{eqn:main_results:initial condition}. Then for every $\epsilon >0$, there exists a constant $\tau_\epsilon > 0$ so that 
for any $\tau \ge \tau_{\epsilon}$, the solution $V$ of the
LDE \eqref{eq:sub:sup:lde:V}
with the initial value 
$V(0) = \gamma(\tau)$
satisfies 
\begin{equation}
\label{eq:phase:gamma:apx:by:v}
    \norm{\gamma(t) - V(t-\tau)}_{\ell^\infty} \leq \epsilon, 
    \qquad \qquad 
    t\geq \tau. 
\end{equation}
\end{prop}

The second step compares the 
dynamics of \eqref{eqn:approximation_of_MCF:discrete_mean_curvature_eqn}
and \eqref{eq:sub:sup:lde:V}
 and shows
that the solutions $V$ and $\Gamma$
closely track each other.
This is achieved by
developing a local comparison principle for \eqref{eqn:approximation_of_MCF:discrete_mean_curvature_eqn} that is valid
as long as $\Gamma$ is sufficiently flat.

\begin{prop}\label{thm:analysis of mean curvature eqn:Approximation of a mean curvature flow}
    Fix $R>0$. Then for any $\epsilon>0$
    there exists  $\delta>0$ so that 
    any pair $\Gamma$, $V\in C^1\big([0, \infty), \ell^\infty(\Z, \R)\big)$ that satisfies the assumptions
    \begin{enumerate}[(a)]
        \item $\Gamma$ satisfies the mean curvature LDE 
        \eqref{eqn:approximation_of_MCF:discrete_mean_curvature_eqn} on $(0, \infty)\times \Z$;
        \item $V$ satisfies the heat LDE \eqref{eq:sub:sup:lde:V} on $(0, \infty)\times \Z$;
        \item $\Gamma(0) = V(0)$,
         with $\norm{\partial^+ V(0)}_{\ell^\infty} \leq \delta$
         and $[V(0)]_{\mathrm{dev}} \le R$,
        \end{enumerate}
        must in fact have
         \begin{equation}
            \norm{\Gamma(t) - V(t)}_{\ell^\infty} \leq \epsilon, \quad\text{for all }t\geq 0. 
        \end{equation}
\end{prop}


\subsection{Approximating $\gamma$ by $V$}

The main idea for our proof of
Proposition \ref{prop:interface_evolution:approx_of_gamma_by_V} is to
compare the information on $\gamma$
resulting from the asymptotic description \eqref{eqn:main_results:front}
with the phase information that can be derived
from \eqref{eq:sps:def:u:plus:minus}. In particular,
we capture the solution $u$
between the sub- and super-solutions
constructed in {\S}\ref{sec:sub:sup}
and exploit the monotonicity properties of $\Phi$.

\begin{lemma}
\label{lem:asymp:comp:V:gamma:prlm}
Assume that (H$g$), (H$\Phi$) and (H0) all hold and let $u$ be a solution of \eqref{eqn:main_results:discrete AC} with the initial condition \eqref{eqn:main_results:initial condition}. Then for every $\epsilon >0$, there exists a constant $\tau_\epsilon > 0$ so that 
for any $\tau \ge \tau_{\epsilon}$ the solution $V$ of the
LDE \eqref{eq:sub:sup:lde:V}
with the initial value 
$V(0) = \gamma(\tau)$
satisfies 
\begin{equation}
\label{eqn:asymp:phi:bnds}
    \Phi\big(i-\gamma_j(t) \big)  
    \le \Phi\big(i -V_j(t-\tau)
    \big) + \epsilon \\
\end{equation}    
for all $(i,j) \in \Z^2$ and $t \ge \tau$.
\end{lemma}
\begin{proof}
Without loss, we assume that $0 < \epsilon < 1$. Recalling the constant $\nu_{\epsilon}$
from Proposition~\ref{prp:construction_of_super_and_subsolutions:supersolution}, 
 Theorem~\ref{thm:main results:gamma_approximates_u}
 and Lemma 
 \ref{lemma:phase_gamma:gamma-ct_bounded}
 allow us to find 
 $\tau_\epsilon>0$ and $R >0$ for which
 the bounds
\begin{equation}\label{eqn:asymp:comp:init:bnd:u:i:j}
    \left| u_{i,j}(t) - \Phi\big(i-\gamma_j(t)\big) \right| \leq \frac{1}{2}\nu_{\epsilon}, 
    \qquad \qquad
    [\gamma(t)]_{\mathrm{dev}} \le R
\end{equation}
hold for all $(i,j)\in \Z^2$
and $t \ge \tau_{\epsilon}$. 
We now recall the constant $\delta>0$ and the functions $p$ and $q$ 
that arise by
applying Proposition~\ref{prp:construction_of_super_and_subsolutions:supersolution} with our pair $(\epsilon, R)$. Decreasing $\delta$ if necessary, we may assume that $\epsilon > \delta$. After possibly increasing
$\tau_{\epsilon}$,
we may use Proposition~\ref{prp:zero_level_surface:smallnes of derivative of gamma} to obtain 
\begin{equation}
\norm{\partial^+ \gamma (\tau)}_{\ell^\infty}\leq \delta.
\end{equation}

We now recall the super-solution
$u^+$ defined in \eqref{eq:sps:def:u:plus:minus}.
Our choice for $V$ together with 
the
bounds \eqref{eq:int:sub:sup:bnd:init:p:zero}
and \eqref{eqn:asymp:comp:init:bnd:u:i:j}
imply that
\begin{equation}
    u_{i,j}(\tau) \leq
     \Phi\big(i-\gamma_j(\tau)\big) + r\big(i-\gamma_j(\tau)\big)
      [\alpha_{\gamma}]_j(\tau) + p(0)
     =u^+_{i,j}(0) 
     .
\end{equation}    
In particular,
the comparison principle
for
LDE~\eqref{eqn:main_results:discrete AC} together with
the bound \eqref{eqn:asymp:comp:init:bnd:u:i:j}
implies that
\begin{equation}
 \Phi\big(i - \gamma_j(t) \big)
 \le 
 u_{i,j}(t) + \frac{1}{2} \nu(\epsilon) \leq
    u^+_{i,j}(t - \tau)
    + \frac{1}{2} \nu_{\epsilon},
    \qquad 
    \qquad
    t \ge \tau.
\end{equation}
Corollary
\ref{cor:interface_evolution:cole hopf} allows us to obtain the uniform bound
$\norm{\alpha_{V}}_{\ell^\infty} \le C_1 \delta^2
\le C_1 \epsilon^2$ for some $C_1 > 0$. Recalling
items (ii) and (iii)
of Proposition \ref{prp:construction_of_super_and_subsolutions:supersolution},
we obtain the bound
\begin{equation}
    u^+_{i,j}(t)
    - \Phi\big( i - V_j(t) \big)
    \le C_2 \epsilon,
    \qquad
    \qquad
    t \ge 0
\end{equation}
for some $C_2 > 0$. In particular, we see that
\begin{equation}
    \Phi\big( i  - \gamma_j(t) \big)
     \le \Phi\big( i - V_j(t - \tau) \big)
     + \frac{1}{2} \nu_{\epsilon}
     + C_2 \epsilon,
     \qquad \qquad
     t \ge \tau,
\end{equation}
from which the statement can
readily be obtained.
\end{proof}

\begin{proof}[Proof of Proposition \ref{prop:interface_evolution:approx_of_gamma_by_V}]
For convenience, we write
\begin{equation}
    \sigma = \min_{\xi \in [0,3]} \Phi'(\xi) > 0.
\end{equation}
Recalling the constant $\tau_{\epsilon} > 0$ defined
in Lemma \ref{lem:asymp:comp:V:gamma:prlm}
and picking $\tau \ge \tau_{\epsilon}$, we set out to show that
\begin{equation}
    V_j(t - \tau)  - \gamma_j(t) \leq \sigma^{-1} \epsilon,
    \qquad \qquad
    t \ge \tau.
\end{equation}
Arguing by contradiction,
we plug $i= \lceil V_j(t - \tau)  \rceil$ into
\eqref{eqn:asymp:phi:bnds}
and obtain
\begin{equation}
\label{eq:asymp:ineqs:phi:phi:one}
    \Phi\big( \lceil V_j(t - \tau) \rceil - \gamma_j(t)\big) 
    \leq 
    \Phi\big( \lceil V_j(t - \tau) \rceil - V_j(t - \tau)\big) + \epsilon
    \leq \Phi(1) + \epsilon \leq \Phi(2),
\end{equation}
possibly after restricting the size of $\epsilon > 0$.
This implies $\sigma^{-1} \epsilon < V_j(t - \tau) - \gamma_j(t) \le 2$. The first inequality
in \eqref{eq:asymp:ineqs:phi:phi:one} now yields the contradiction
\begin{equation}
\begin{array}{lcl}
   \epsilon
& < & 
    \sigma  
      \big( V_j(t - \tau) - \gamma_j(t) \big)
\\[0.2cm]
    & \le &
      \Phi\big( \lceil V_j(t - \tau) \rceil - \gamma_j(t)\big) 
      -\Phi\big( \lceil V_j(t - \tau) \rceil - V_j(t - \tau)\big)
\\[0.2cm]
    & \le & \epsilon,
\end{array}
\end{equation}
since both arguments of $\Phi$
are contained in the interval $[0,3]$.
An $O(\epsilon)$ lower bound
for $V_j(t - \tau) - \gamma_j(t)$
can be obtained in a similar fashion, which allows the proof to be completed.
\end{proof}

\subsection{Tracking $V$ with $\Gamma$}

In this subsection we set out to establish Proposition \ref{thm:analysis of mean curvature eqn:Approximation of a mean curvature flow}. The main idea
to establish this approximation result is to apply a local comparison principle
to the discrete curvature  LDE~\eqref{eqn:approximation_of_MCF:discrete_mean_curvature_eqn}.
To this end, 
we define the residual
\begin{equation}\label{eqn:approximation_of_MCF:residual_J}
\mathcal{J}_{\mathrm{dc}}[\Gamma] = \dot{\Gamma} - \dfrac{\partial^{(2)}\Gamma}{\beta_\Gamma^2} - 2d \beta_\Gamma - c + 2d
\end{equation}
for any $\Gamma \in C^1\big([0,\infty);\ell^\infty(\Z) \big)$. As usual, we
say that $\Gamma$ is a super- or sub-solution 
for \eqref{eqn:approximation_of_MCF:discrete_mean_curvature_eqn} if the inequality
$\mathcal{J}_{\mathrm{dc}}[\Gamma]_j(t) \geq 0$
respectively 
$\mathcal{J}_{\mathrm{dc}}[\Gamma]_j(t) \leq 0$
holds for all $j \in \Z$
and $t \ge 0$.

\begin{lemma}[Comparison principle]\label{lemma:analysis of mean curvature eqn: comparison principle}
Pick a sufficiently small $\delta>0$ and consider a pair of functions  $\Gamma^-$, $\Gamma^+\in C^1\big([0, \infty), \ell^\infty(\Z, \R)\big)$ 
that satisfy the  following assumptions: 
\begin{enumerate} [(a)]
    \item $\Gamma^-$ is a subsolution of the LDE \eqref{eqn:approximation_of_MCF:discrete_mean_curvature_eqn};
    \item $\Gamma^+$ is a supersolution of the LDE \eqref{eqn:approximation_of_MCF:discrete_mean_curvature_eqn};
    \item The inequalities $\norm{\partial^+ \Gamma^- (t)}_{\ell^\infty}\leq \delta,$
    and $\norm{\partial^+ \Gamma^+ (t)}_{\ell^\infty}\leq \delta$
    hold for every $t\geq 0$;
    \item $\Gamma^-_{j}(0) \leq \Gamma^+_{j}(0)$ holds for every $j\in \Z$.
\end{enumerate}
 Then for every $j \in \Z$ and $t \ge 0$ we have the bound 
 $$ \Gamma^-_{j}(t)\leq \Gamma^+_{j}(t) .
 $$
\end{lemma}
\begin{proof}
    Define the function $W:[0, \infty)\to \ell^\infty(\Z)$ by $W(t)= \Gamma^+(t) - \Gamma^-(t)$. Then $W$ satisfies the differential inequality
    $$\dot{W}_j \geq  (W_{j+1} - W_j) F(\Gamma^+, \Gamma^-)_j + (W_{j-1}-W_j) G(\Gamma^+,\Gamma^-)_j,
    $$
    in which the functions $F$ and $G$ are defined by
    \begin{align*}
        F(\Gamma^-,\Gamma^+) &= \dfrac{1}{\beta_{\Gamma^+}^2} + \left(\partial^+ \Gamma^- + \partial^+ \Gamma^+\right)\left(\dfrac{d}{ \beta_{\Gamma^-} + \beta_{\Gamma^+}} - \dfrac{\partial^{(2)} \Gamma^-}{2\beta_{\Gamma^-}^2\beta_{\Gamma^+}^2 }\right), \\
        G(\Gamma^-,\Gamma^+) &= \dfrac{1}{\beta_{\Gamma^+}^2} - \left(\partial^- \Gamma^- + \partial^- \Gamma^+\right)\left(\dfrac{d}{\beta_{\Gamma^-} + \beta_{\Gamma^+}} - \dfrac{\partial^{(2)} \Gamma^-}{2\beta_{\Gamma^-}^2 \beta_{\Gamma^+}^2 }\right) .
    \end{align*}
Pick $\delta > 0$ in such a way that $\dfrac{1}{1+\delta^2} >  \delta\left(|d| +2\delta\right) + \frac{1}{2} $. Notice that 
this choice and assumption \textit{(c)} imply that both $\beta_{\Gamma^-}$ and   $\beta_{\Gamma^+}$  are bounded by $\sqrt{1+\delta^2}$, which in turn  implies
\begin{equation}
F(\Gamma^-,\Gamma^+)> \frac{1}{2},
\qquad \qquad  G(\Gamma^-,\Gamma^+)>\frac{1}{2}.
\end{equation}

In order to prove that $W\geq 0$, we assume to the contrary that there exist  $j_*\in Z$ and $t_*$ such that $W_{j_*}(t_*) = -\vartheta < 0$. Picking $\epsilon > 0$ and $K>0$ in such a way that $\vartheta = \epsilon e^{2Kt_*}$, we can define
\begin{equation}
    T:=\sup\left\{t\geq 0: W_{j}(t) > -\epsilon e^{2Kt} \ \text{for all}\ j\in \Z\right\}.
\end{equation}
Since $W\in C^1\big(\left[0, \infty\right); \ell^\infty(\Z)\big)$ we must have $T\leq t^*$ and 
\begin{equation}
\inf_{j\in \Z} W_j(T) = -\epsilon e^{2KT}.
\end{equation}
Without loss of generality we may assume that $W_{0}(T) < -\frac{7}{8}\epsilon e^{2KT}$. 

We now choose a sequence $z\in \ell^\infty(\Z, \R)$ with the properties \begin{equation}
    z_0 = 1, \quad \lim_{|j|\to\infty} z_j = 3, \quad  1\leq z_j\leq 3 \quad \text {and} \quad   \norm{\partial ^+ z}_{\ell^\infty}\leq 1.
\end{equation} 
This allows us to define the function \begin{equation}
    W^{-}_j(t;\alpha) = -\epsilon(\dfrac{3}{4} +\alpha z_{j})e^{2Kt},
\end{equation}
in which $\alpha > 0$ is a parameter. We denote by $\alpha^* \in (\frac{1}{8}, \frac{1}{4}]$ the minimal value of $\alpha$ for which $W_j(t) \geq W_j^-(t;\alpha)$ for all $(j,t)\in \Z\times [0,T]$. 
In view of the limiting behaviour
\begin{equation}
    \lim_{|j|\to\infty} W_j^-(t;\alpha^*) = -\epsilon [\dfrac{3}{4} + 3\alpha^*]e^{2Kt} < -\dfrac{9}{8}\epsilon e^{2Kt},
\end{equation}
the minimality of $\alpha^*$
allows us to
conclude that there exist $j_0\in \Z$ and $0<t_0\leq T$ such that $W_{j_0}(t_0) = W^-_{j_0}(t_0;\alpha^*)$. 
As a consequence, we must have
\begin{equation}
    \dot{W}_{j_0}(t_0) \leq \dot{W}^-_{j_0}(t_0;\alpha^*). 
\end{equation}

In addition, the definitions of $ W^-$ and $\alpha^*$ directly yield
the inequalities
\begin{align}
    W_{j_0+1}(t_0) - W_{j_0}(t_0)  \geq  W^-_{j_0+1}(t_0;\alpha^*) - W^-_{j_0}(t_0;\alpha^*), \\
    W_{j_0-1}(t_0) - W_{j_0}(t_0)  \geq  W^-_{j_0-1}(t_0;\alpha^*) - W^-_{j_0}(t_0;\alpha^*).
\end{align}
Together with the bounds
\begin{equation}
    \dot{W}^-_{j_0}(t_0;\alpha^*)
    \le -\frac{7}{4}\epsilon K e^{2K t_0},
    \qquad \qquad
    \norm{ \partial^\pm 
    W^-(t_0;\alpha^*) }_{\ell^\infty}
    \le  \epsilon e^{2 K t_0},
\end{equation}
this allows us to compute
\begin{equation}
\begin{aligned}
-\frac{7}{4}\epsilon K e^{2Kt_0} 
&\geq (\partial^+ W)_{j_0}(t_0) F(U, V)_{j_0} - (\partial^- W)_{j_0}(t_0) G(U,V)_{j_0} \\
&\geq \frac{1}{2} (\partial^+ W^-)_{j_0}(t_0;\alpha^*)  - \frac{1}{2} (\partial^- W^-)_{j_0}(t_0;\alpha^*) \\ 
&\geq -\epsilon e^{2Kt_0}.
\end{aligned}
\end{equation}
This leads to the desired contradiction upon choosing $K>1$ to be sufficiently large.
\end{proof}

In order to use the comparison principle above to compare
$V$ and $\Gamma$,
we need to obtain uniform bounds on the discrete derivatives
$\partial^+V$ and $\partial^+\Gamma$.
Corollary~\ref{cor:interface_evolution:cole hopf} provides such bounds
for $\partial^+V$, but
the corresponding estimates
for $\partial^+ \Gamma$ require
some additional technical work.

We pursue this in the results below,
establishing a second comparison principle directly
for the function
$\Upsilon:=\partial^+ \Gamma$.
Indeed, upon introducing the shorthand
 $$\Pi[\Upsilon]_j = \sqrt{1+ \left(\Upsilon_{j+1}^2 + \Upsilon_{j}^2\right)/2} $$
 and differentiating
 \eqref{eqn:approximation_of_MCF:discrete_mean_curvature_eqn},
 a short computation shows that $\Upsilon$ satisfies the LDE
 \begin{equation} \label{eqn:analysis_of_discrete_mean_curvature_flow:delta of mean curvatre flow}
\dot{\Upsilon}_j =\dfrac{\partial^+ \Upsilon_j}{\Pi[\Upsilon]_{j}^2} - \dfrac{\partial^- \Upsilon_j}{\Pi[\Upsilon]_{j-1}^2} + 2d \big(\Pi[\Upsilon]_j - \Pi[\Upsilon]_{j-1}\big). \end{equation}

\begin{lemma}\label{lemma:approximation_of_MCF:comp_principle_for_gradient_equation}
Pick a sufficiently small $\delta>0$ and consider a pair of functions $\Upsilon^-$, $\Upsilon^+\in C^1\big([0, \infty),\ell^\infty(\Z)\big)$ that satisfy the following assumptions:
\begin{enumerate}[(a)]
    \item $\Upsilon^-$ is a subsolution of the LDE \eqref{eqn:analysis_of_discrete_mean_curvature_flow:delta of mean curvatre flow};
    \item $\Upsilon^+$ is a supersolution of the  LDE \eqref{eqn:analysis_of_discrete_mean_curvature_flow:delta of mean curvatre flow};
    \item The inequalities $ \norm{\Upsilon^-(t)}_{\ell^\infty} \leq \delta$ and $ \norm{\Upsilon^+ (t)}_{\ell^\infty} \leq \delta$ hold for every $t\geq 0$;
    \item $\Upsilon^-_j(0) \leq \Upsilon^+_j(0)$ holds for every $j\in\Z$.
    \end{enumerate}
%
Then for every $j \in \Z$ and $t \ge 0$ we have the inequality
\begin{equation}
    \Upsilon^-_{j}(t)\leq \Upsilon^+_{j}(t).
\end{equation}
\end{lemma}

\begin{proof}
Defining $Z_j(t)=\Upsilon^+_j(t)-\Upsilon^-_j(t)$, we see that $Z_j(0)\geq 0$ for every $j\in \Z$. Moreover,  $Z$ satisfies the differential inequality
\begin{equation}
    \dot{Z}_j \geq F(\Upsilon^-, \Upsilon^+)_j(Z_{j+1} - Z_j) + G(\Upsilon^-, \Upsilon^+)_{j-1}(Z_{j-1} - Z_j) + H(\Upsilon^-, \Upsilon^+)_j Z_j,
\end{equation}
in which the functions $F$, $G$ and $H$ are defined by 
\begin{align*}
    F(\Upsilon^-, \Upsilon^+)_j &=  \dfrac{1}{\Pi[\Upsilon^+]_j^2} +\left( \dfrac{d}{\Pi[\Upsilon^+]_j + \Pi[\Upsilon^-]_j} - \dfrac{\partial^+ \Upsilon^-_j}{2\Pi[\Upsilon^+]_j^2\Pi[\Upsilon^-]_j^2} \right)\left(\Upsilon^+_{j+1} + \Upsilon^-_{j+1}\right), \\
    G(\Upsilon^-, \Upsilon^+)_j &=  \dfrac{1}{\Pi[\Upsilon^+]_{j-1}^2} +\left(\dfrac{\partial^- \Upsilon^-_j}{2\Pi[\Upsilon^+]_{j-1}^2\Pi[\Upsilon^-]_{j-1}^2} - \dfrac{d}{\Pi[\Upsilon^+]_{j-1} + \Pi[\Upsilon^-]_{j-1}}\right)\left(\Upsilon_{j-1}^+ + \Upsilon^-_{j-1}\right), \\
    H(\Upsilon^-, \Upsilon^+)_j &= \left( \dfrac{d}{\Pi[\Upsilon^+]_j + \Pi[\Upsilon^-]_j} - \dfrac{\partial^+ \Upsilon^-_j}{2\Pi[\Upsilon^+]_j^2\Pi[\Upsilon^-]_j^2} \right)\left(\Upsilon^+_{j+1}  + \Upsilon^-_{j+1} + \Upsilon^+_{j}  + \Upsilon^-_{j}\right) - \\
    &\quad \quad + \left(\dfrac{\partial^- \Upsilon^-_{j}}{2\Pi[\Upsilon^+]_{j-1}^2\Pi[\Upsilon^-]_{j-1}^2} - \dfrac{d}{\Pi[\Upsilon^+]_{j-1} + \Pi[\Upsilon^-]_{j-1}}\right)\left(\Upsilon^+_{j}  + \Upsilon^-_{j} + \Upsilon^+_{j-1}  + \Upsilon^-_{j-1}\right).
\end{align*}
We again pick $\delta > 0$ in such a way that $\dfrac{1}{1+\delta^2} >  \delta\left(|d| +2\delta\right) + \frac{1}{2} $. Notice that 
this choice and assumption \textit{(c)} imply that both $\Pi[\Upsilon^-]$ and   $\Pi[\Upsilon^+]$  are bounded by $\sqrt{1+\delta^2}$. This in turn  yields the bounds 
\begin{equation}
    F(\Upsilon^-, \Upsilon^+) > 1/2 ,
    \qquad
    G(\Upsilon^-, \Upsilon^+)> 1/2,
    \qquad
    |H(\Upsilon^-, \Upsilon^+)|\leq 4\delta(2\delta + |d|).  
\end{equation} 

Applying a similar procedure as in the proof of Lemma~\ref{lemma:analysis of mean curvature eqn: comparison principle} allows us to conclude that $Z_j(t)\geq 0$ for every $j\in \Z$.   
\end{proof}
\begin{lemma}\label{lemma:approximation_of_MCF:uniform_bound_on_gradient_of_MCF_T}
Fix $T > 0$ and pick a sufficiently small 
$\delta_0>0$. Then for any
$\Gamma^0 \in \ell^\infty(\Z)$
with $\norm{\partial^+ \Gamma^0}_{\ell^\infty} \le \delta_0$, the solution
$\Gamma \in C^1\big([0,T],\ell^\infty(\Z)\big)$
to the mean curvature LDE
\eqref{eqn:approximation_of_MCF:discrete_mean_curvature_eqn}
with $\Gamma(0) = \Gamma^0$
satisfies
\begin{equation}\label{eqn:approximation_of_MCF:uniform_bound_on_gradient_of_MCF}
    \norm{\partial^+ \Gamma(t)}_{\ell^\infty}\leq \delta_0, \text{ for all } t\in [0, T]. 
\end{equation}
%
\end{lemma}
\begin{proof}
Writing $\Upsilon = \partial^+ \Gamma$, we can apply Grönwall's inequality to~\eqref{eqn:analysis_of_discrete_mean_curvature_flow:delta of mean curvatre flow} to find
\begin{equation}\label{eqn:analysis of mean curvature eqn:Gronwall}
    \norm{\Upsilon(t)}_{\ell^\infty} \leq K\norm{\Upsilon(0)}_{\ell^\infty} e^{bt},
\end{equation}
for some constants $K\geq 1$ and $b>0$ that are independent of $T$.
Recalling the constant $\delta > 0$ from Lemma~\ref{lemma:approximation_of_MCF:comp_principle_for_gradient_equation}, we now choose $\delta_0>0$ in such a way 
that $\delta_0 Ke^{bT}\leq \delta$. 
Applying the comparison principle from Lemma~\ref{lemma:approximation_of_MCF:comp_principle_for_gradient_equation}, we conclude that
$\norm{\Upsilon(t)}_{\ell^\infty}\leq\norm{\Upsilon(0)}_{\ell^\infty}$
holds for $t\in [0,T]$.
Indeed, the constant function $\norm{\Upsilon(0)}_{\ell^\infty} $ also satisfies LDE~\eqref{eqn:analysis_of_discrete_mean_curvature_flow:delta of mean curvatre flow}.
\end{proof}

\begin{cor}\label{lemma:approximation_of_MCF:ex_and_uq_of_MCF}
Pick $\Gamma^0\in\ell^\infty(\Z)$. Then there exists an unique solution $\Gamma\in C^1\big([0, \infty),\ell^\infty(\Z)\big)$ of the mean curvature LDE~\eqref{eqn:approximation_of_MCF:discrete_mean_curvature_eqn}. Moreover, there exists $\delta>0$ such that the initial bound $\norm{\partial^+\Gamma^0}_{\ell^\infty} \leq \delta$, 
implies that also \begin{equation}\label{lemma:approximation_of_MCF:uniform_bound_on_gradient_of_MCF_infty}
    \norm{\partial^+\Gamma(t)}_{\ell^\infty}\leq \delta, \quad \text{for all } t\geq 0.
\end{equation}
\end{cor}
\begin{proof}
 Existence and uniqueness follows from standard arguments. Applying an iterative argument involving Lemma~\ref{lemma:approximation_of_MCF:uniform_bound_on_gradient_of_MCF_T} leads to the uniform bound~\eqref{lemma:approximation_of_MCF:uniform_bound_on_gradient_of_MCF_infty}. 
\end{proof}

\begin{proof}[Proof of Proposition \ref{thm:analysis of mean curvature eqn:Approximation of a mean curvature flow}]
Using the fact that $V$ satisfies the LDE~\eqref{eq:sub:sup:lde:V}, we compute
\begin{align}
    \mathcal{J}_{\mathrm{dc}}[V] &= \dot{V} - \dfrac{\partial^{2}V}{\beta_V^2} - 2d\beta_V  - c + 2d \\
    &=\partial^{(2)}V \big(1-\dfrac{1}{\beta_V^2}\big) + 2d\big(1+\dfrac{1}{2}\alpha_V - \beta_V \big) \\
    & \ + \dfrac{1}{2d}\int_0^{d\partial^+V} e^s(d\partial^+ V - s)^2 ds + \dfrac{1}{2d}\int_0^{-d\partial^-V} e^s (d\partial^- V + s)^2 ds.
\end{align}
Expanding $\beta_V$ around $0$ and using Corollary~\ref{cor:interface_evolution:cole hopf}, we find a constant $M>0$ for which
\begin{equation}
    \norm{\mathcal{J}_{\mathrm{dc}}[V]}_{\ell^\infty} \leq M\min\left\{\norm{\partial^+V(0)}_{\ell^\infty}, t^{-\frac{3}{2}} \right\}.
\end{equation}
We define the constant $\delta>0$ and the function $K:[0, \infty) \to \R$ by
\begin{equation}
    \delta = \dfrac{\epsilon^3}{M^36^3}, \qquad \qquad K(t) = M\min\left\{\delta, t^{-\frac{3}{2}}\right\}.
\end{equation}
Possibly reducing $\epsilon > 0$,
we many assume that $\delta > 0$ is sufficiently small to satisfy the requirements of
Lemma~\ref{lemma:analysis of mean curvature eqn: comparison principle} and Corollary~\ref{lemma:approximation_of_MCF:ex_and_uq_of_MCF}.

 Next, we pick a smooth function $q:[0, \infty)$ that satisfies
\begin{equation}
   K(t) \leq q(t) \leq 2K(t)
\end{equation}
 and introduce the integral  $p(t) = \int_0^t q(s) ds$. It is straightforward to check that $0\leq p(t) \leq \epsilon$ for every $t\geq 0$.
By spatial homogeneity, we have
$\mathcal{J}_{\mathrm{dc}}[V + p] = \mathcal{J}_{\mathrm{dc}}[V] + \dot{p}$ and hence 
\begin{equation}
    \norm{\mathcal{J}_{\mathrm{dc}}[V+p]}_{\ell^\infty} \geq 0. 
\end{equation}
In particular, the function
$V+p$ is a supersolution of the LDE~\eqref{eq:sub:sup:lde:V}.  Lemma~\ref{lemma:analysis of mean curvature eqn: comparison principle} hence implies
\begin{equation}
    \Gamma(t) \leq V(t) + p(t) \leq V(t) + \epsilon .
\end{equation}
The inequality $V(t) - \epsilon \leq \Gamma(t)$ follows similarly by constructing an appropriate sub-solution for the LDE~\eqref{eqn:approximation_of_MCF:discrete_mean_curvature_eqn}.
\end{proof}

\subsection{Proof of Theorem~ \ref{thm:main_results:approx_of_gamma}}

As a final step, we need
to link the parameter $d = - \langle \Psi'', \psi \rangle$ used 
here and in {\S}\ref{sec:sub:sup}
to the expressions
in \eqref{eq:mr:discrete:MCF} that involve the wavespeed
$c$ and its angular derivatives.
To this end, we recall the identity
\begin{equation}(\partial^2_\theta c_\theta)_{|_{\theta=0}} = \langle \Phi'(\cdot +1) - \Phi'(\cdot-1)  - 2\Phi'', \psi \rangle
\end{equation}
that was obtained in \cite{hupkes2019travelling}.
As expected, this expression
vanishes in the continuum limit
since
$$\lim_{h\to 0} \dfrac{\Phi'(\cdot + h)- \Phi'(\cdot - h)}{h} - 2\Phi'' = 0.$$

\begin{lemma}\label{lemma:proof_of_thm2.3:d}
Suppose that $(Hg)$ and $(H\Phi)$
both hold. Then the parameter
$d$ defined in
\eqref{eqn:super_and_sub_sols2:d} 
satisfies the identity
\begin{equation}
\label{eq:asymp:id:for:d:with:c}
    d = 
    \dfrac{c}{2} + \dfrac{(\partial^2_\theta c_\theta)|_{\theta=0}}{2}.
\end{equation}
\end{lemma}
\begin{proof}
Comparing \eqref{eq:asymp:id:for:d:with:c}
with \eqref{eqn:super_and_sub_sols2:d} and recalling the characterization
\eqref{eq:subsup:char:ker:range}
together with
the normalization $\langle \Phi', \psi \rangle = 1$,
it suffices to show that
the function
\begin{equation}
h(\xi) = \Phi'(\xi +1) - \Phi'(\xi-1)  + c\Phi'(\xi)
\end{equation}
satisfies $h \in \mathrm{Range} \big( \mathcal{L}_{\mathrm{tw}} \big)$.
To achieve this, we write
$\varphi(\xi) = \xi\Phi'(\xi)$
and recall the travelling wave
MFDE
\eqref{eqn:main_results:MFDE}
to compute
\begin{equation}
    \begin{array}{lcl}
\mathcal{L}_{\mathrm{tw}}\varphi(\xi) &= &c \varphi'(\xi) + \varphi(\xi+1) + \varphi(\xi-1) -2\varphi(\xi)
+ g'\big(\Phi(\xi)\big)\varphi(\xi) \\[0.2cm]
&=& c\Phi'(\xi) +  c\xi\Phi'' (\xi) + \xi\Phi'(\xi+1) + \Phi'(\xi+1) + \xi\Phi'(\xi-1) - \Phi'(\xi-1)  -2\xi\Phi'(\xi) 
\\[0.2cm]
& & \qquad 
+ \xi g'\big(\Phi(\xi)\big)\Phi'(\xi) \\[0.2cm]
&=& h(\xi) + \xi\dfrac{d}{d\xi}\Big(c \Phi'(\xi) + \Phi(\xi+1) + \Phi(\xi-1) -2\Phi(\xi) + g\big(\Phi(\xi)\big)\Big) \\[0.2cm]
&= & h(\xi), \\
\end{array}
\end{equation}
as desired.
\end{proof}

\begin{proof}[Proof of Theorem~\ref{thm:main_results:approx_of_gamma}]
The statements follow directly from Propositions~\ref{prop:interface_evolution:approx_of_gamma_by_V}-\ref{thm:analysis of mean curvature eqn:Approximation of a mean curvature flow} 
and Lemma~\ref{lemma:proof_of_thm2.3:d}. 
\end{proof}

\section{Stability results}
\label{sec:per}

Our goal here is to establish 
Theorem \ref{thm:mr:periodicity+decay:stability},
our final main result. In particular, we consider
the two solutions 
\begin{equation}
\label{eq:per:sol:u:u:per}
    u: [0 , \infty) \to \ell^\infty(\Z^2),
    \qquad \qquad
    u^{\mathrm{per}}: [0, \infty) \to \ell^\infty(\Z^2)
\end{equation}
to the Allen-Cahn LDE \eqref{eqn:main_results:discrete AC}
with the respective  initial conditions
\begin{equation}
    u(0) = u^0, \qquad \qquad u^{\mathrm{per}}(0) = u^{0;\mathrm{per}},
\end{equation}
together with their associated phases
\begin{equation}
\label{eq:per:def:phases}
    \gamma: [T , \infty) \to \ell^\infty(\Z),
    \qquad \qquad
    \gamma^{\mathrm{per}}: [T, \infty) \to \ell^\infty(\Z)
\end{equation}
that are defined by \eqref{eqn:main_results:def_of_gamma}
for some sufficiently large $T \gg 1$.
Since the LDE \eqref{eqn:main_results:discrete AC} is
autonomous, the uniqueness of solutions imply  that $u^{\mathrm{per}}$ and hence the phase $\gamma^{\mathrm{per}}$ inherit the $j$-periodicity
\begin{equation}
    u^{\mathrm{per}}_{i,j+P}(t) = 
    u^{\mathrm{per}}_{i,j}(t),
    \qquad \qquad
    \gamma^{\mathrm{per}}_{j+P}(t) = 
    \gamma^{\mathrm{per}}_{ij}(t)
\end{equation}
for $ t\ge 0$ respectively $t \ge T$.

It is natural to expect that
$u_{\cdot, j }(t)$ converges to $u^{\mathrm{per}}_{\cdot, j}(t)$ as $|j| \to \infty$,
which we confirm below in  {\S}\ref{sec:per:spt:asymp}.
However, one cannot expect the corresponding result to hold for the phases \eqref{eq:per:def:phases},
on account of the discontinuities that occur. In fact,
we obtain the following asymptotic 'almost-convergence' result.

\begin{prop}
\label{prp:per:phase:cmp}
Consider the setting of Theorem \ref{thm:mr:periodicity+decay:stability} and recall
the two phase functions \eqref{eq:per:def:phases}.
Then for every $\epsilon > 0$ there exists
a constant $T_{\epsilon}$ together with a function
\begin{equation}
    J_{\epsilon}: [T_\epsilon , \infty) \to \Z_{\ge 0}
\end{equation}
so that we have the bound
\begin{equation}
    |\gamma_j(t) - \gamma^{\mathrm{per}}_j(t)| \leq \epsilon
\end{equation}
for every $t \ge T_{\epsilon}$
and $|j| \ge J_{\epsilon}(t)$.
\end{prop}

In order to explore the consequences
of the approximation result Proposition \ref{prop:interface_evolution:approx_of_gamma_by_V},
we hence need to understand the evolution
of asymptotically almost-periodic initial conditions
under \eqref{eq:sub:sup:lde:V}. This is achieved
in our second main result here. We emphasize that in the special case $P=1$, the asymptotic phase $\mu$ is equal
to the value taken by the constant sequence $V^{0;\mathrm{per}}$.

\begin{prop}
\label{prp:per:beh:alm:per:nl:heat}
Suppose that the assumptions  (Hg)
and (H$\Phi$) both hold, fix two constants
$R > 0$ and $P \in \Z_{>0}$ and pick a sufficiently large $K > 0$. Then for any $\epsilon > 0$ and $J \in \Z_{\ge 0}$, there exists a time $T_{\epsilon,J} > 0$
so that the following holds true.

Consider any pair $(V^0, V^{0;\mathrm{per}}) \in \ell^\infty(\Z)^2$ that satisfies the conditions
\begin{itemize}
    \item[(a)] For all $|j| \ge J$ we have
    $|V^0_{j} - V^{0;\mathrm{per}}_j| \le \epsilon $.
    \item[(b)] The periodicity 
    $V^{0;\mathrm{per}}_{j+P} = V^{0;\mathrm{per}}_{j}$
    holds for all $j \in \Z$.
    \item[(c)] We have the deviation bounds
    \begin{equation}
    [V^{0;\mathrm{per}}]_{\mathrm{dev}}\leq R, \qquad
    \qquad [V^0]_{\mathrm{dev}}\leq R.
    \end{equation}
\end{itemize}
Then there exists an asymptotic phase $\mu \in \R$
so that the solution $V: [0, \infty) \to \ell^\infty(\Z)$
to the LDE \eqref{eq:sub:sup:lde:V} with
the initial condition $V(0) = V^0$
satisfies the bound
\begin{equation}
\label{eq:per:asymp:phase:mu:bnd}
 \norm{V(t) - ct - \mu}_{\ell^\infty(\Z)}\leq K\epsilon,
 \qquad \qquad
 t \ge T_{\epsilon, J}.
\end{equation}
\end{prop}

\begin{proof}[Proof of Theorem~\ref{thm:mr:periodicity+decay:stability}]
Pick $\epsilon>0$. Recalling
the terminology of
of 
Propositions 
\ref{prop:interface_evolution:approx_of_gamma_by_V} and \ref{prp:per:phase:cmp},
we introduce the constants $\overline{\tau}_{\epsilon} = \max\{ \tau_{\epsilon}, T_{\epsilon} \}$ and $\overline{J}_{\epsilon} = J_{\epsilon}(\overline{\tau}_{\epsilon})$
and write
$V^{(\epsilon)}$ for the solution to the LDE
\eqref{eq:sub:sup:lde:V} with the
initial condition $V^{(\epsilon)}(0) = \gamma(\overline{\tau}_{\epsilon})$.
Writing $\mu_\epsilon$ for the phase
defined in Proposition~\ref{prp:per:beh:alm:per:nl:heat},
we combine
\eqref{eq:phase:gamma:apx:by:v}
with
\eqref{eq:per:asymp:phase:mu:bnd}
to obtain
\begin{equation}
\label{eq:per:lim:gamma:phase}
\begin{array}{lcl}
    \norm{\gamma(t + \overline{\tau}_{\epsilon}) - ct - \mu_{\epsilon}}_{\ell^\infty}
     & \leq &
    \norm{\gamma(t+  \overline{\tau}_{\epsilon}) -  V^{(\epsilon)}(t) }_{\ell^\infty}
    + \norm{ V^{(\epsilon)}(t) 
      - ct  - \mu_{\epsilon} )
    }_{\ell^\infty}
\\[0.2cm]
& \le & (K + 1) \epsilon,  
\end{array}
\end{equation}
for all $t \ge T_{\epsilon, \overline{J}_{\epsilon}}$.

We now claim that there exists $\mu \in \R$
for which we have the limit
\begin{equation}
\lim_{\epsilon \downarrow 0} \big( \mu_{\epsilon} - c \overline{\tau}_\epsilon \big)  = \mu.
\end{equation}
Indeed, the uniform bound on $\gamma(t) -c t$ obtained
in Lemma \ref{lemma:phase_gamma:gamma-ct_bounded} allows us to find a convergent subsequence, which using
\eqref{eq:per:lim:gamma:phase}
can be transferred to the full set.
Sending $\epsilon\downarrow 0$ we hence 
obtain
\begin{equation}
    \lim_{t\to\infty}\norm{\gamma(t) - ct -\mu}_{\ell^\infty(\Z)} 
    = 0,
\end{equation}
which leads to the desired convergence in view of Theorem~\ref{thm:main results:gamma_approximates_u}.
\end{proof}

\subsection{Spatial asymptotics}
\label{sec:per:spt:asymp}

In this subsection we establish Proposition \ref{prp:per:phase:cmp}. As a preparation,
we compare the $j$-asymptotic behaviour of the two solutions \eqref{eq:per:sol:u:u:per}. 
We remark that the arguments in Lemma \ref{lemma:periodicity+decay:u_convg_to_uper}
below
remain valid upon replacing the limits
in \eqref{eq:mr:lim:per:j:pm:infty}
and \eqref{eq:per:lim:per:j:pm:infty}
by their two counterparts $|i| \pm j \to \infty$,
which are one-sided in $j$. This validates the comments
in {\S}\ref{sec:int} concerning the limit \eqref{eq:int:kappa:pm:limits}.

\begin{lemma}\label{lemma:stability:continuity_wtr_initial_cond}
Assume that (Hg) is satisfied and consider any $u^0_A \in \ell^\infty(\Z^2)$.
Then for any $\epsilon  > 0$ and time $T > 0$, there exists $\delta > 0$ so that for any $u^0_B \in \ell^\infty(\Z^2)$
that satisfies
\begin{equation}
    \norm{u^0_A - u^0_B}_{\ell^\infty(\Z^2)} \leq \delta,
\end{equation}
the solutions $u_A$ and $u_B$ of the Allen-Cahn LDE
\eqref{eqn:main_results:discrete AC} 
with the initial conditions $u_A(0) = u^0_A$
and $u_B(0) = u^0_B$ satisfy
\begin{equation}
    \norm{u_A(t) - u_B(t)}_{\ell^\infty(\Z^2)} \leq \epsilon,
    \qquad \qquad t \in [0, T].
\end{equation}
\end{lemma}
\begin{proof}
This is a standard consequence of the well-posedness
of \eqref{eqn:main_results:discrete AC} in $\ell^\infty(\Z^2)$.
\end{proof}

\begin{lemma}\label{lemma:periodicity+decay:u_convg_to_uper}
Consider the setting of Theorem \ref{thm:mr:periodicity+decay:stability} and recall the 
two solutions \eqref{eq:per:sol:u:u:per}.
Then for every $\tau > 0$ we have the spatial limit
\begin{equation}
\label{eq:per:lim:per:j:pm:infty}
     u_{i,j}(\tau)  - u^{\mathrm{per}}_{i,j}(\tau) \to 0, 
     \qquad \text{as}\quad  |i|+|j|\to\infty .
\end{equation}
\end{lemma}

\begin{proof}
In view of symmetry considerations, it suffices
to establish the claim for the limit $i + j \to -\infty$.
To this end, we fix an arbitrary $\epsilon > 0$. 
We write $\tilde{u}^{\mathrm{per}}$ for the solution to the LDE \eqref{eq:sub:sup:lde:V} with the initial condition
$\tilde{u}^{\mathrm{per}}(0) = u^{0; \mathrm{per}} + \delta$,
using Lemma~\ref{lemma:stability:continuity_wtr_initial_cond} to pick $\delta > 0$ in such a way that
\begin{equation}
    \tilde{u}^{\mathrm{per}}(\tau) \le u^{\mathrm{per}}(\tau) + \frac{\epsilon}{2}.
\end{equation}
We subsequently pick $M > 0$ in such a way that
\begin{equation}
u^{0}_{i,j} \leq 
\tilde{u}^{0;\mathrm{per}}_{i,j}(0) + \delta + Me^{|c|(i + j)}
=
\tilde{u}^{\mathrm{per}}_{i,j}(0) + Me^{|c|(i + j)}
\end{equation}
holds for every $(i,j)\in \Z^2$.

On account of (H0) and the comparison principle,
we can pick $A \ge 1$ in such a way that
\begin{equation}
-A \le  \tilde{u}^{\mathrm{per}}(t) \le A
\end{equation}
holds for all $t \in [0, \tau]$. We now write 
\begin{equation}
    K = \max\{ g'( s) : -A \le s \le A \} > 0
\end{equation}
and observe that  (Hg) implies that
\begin{equation}
\label{eq:per:choice:K}
    g( s + \beta ) \le g(s) + K \beta
\end{equation}
for any $-A \le s \le A$ and $\beta \ge 0$.

We now pick  $\alpha > 0$ in such a way that
\begin{equation}
\label{eq:per:choice:alpha}
    \alpha|c| - \dfrac{c^4}{6} \cosh |c|  > K
\end{equation}
and claim that the function %
\begin{equation}
\label{eqn:periodocity:w}
    w_{i,j} (t) = \tilde{u}^{\text{per}}_{i,j}(t) + Me^{|c|(i+j+2|c|t + \alpha t)}
\end{equation}
is a super-solution to \eqref{eqn:main_results:discrete AC}. Indeed, recalling the residual
\eqref{eqn:omega_limit_points:residual},
a short computation yields
\begin{equation}
    \begin{array}{lcl}
\mathcal{J}[w]_{i,j}(t) &=& g\big(\tilde{u}^{\text{per}}_{i,j}(t) \big) 
- g\big(w_{i,j}(t)\big) + Me^{|c|(i+j+2|c|t + \alpha t)}\left(2c^2 + \alpha|c| - 2e^{|c|} - 2e^{-|c|} + 4\right) \\[0.2cm]
&= & g\big(\tilde{u}^{\text{per}}_{i,j}(t) \big) - g\big(w_{i,j}(t)\big) + \big(w_{i,j}(t)-\tilde{u}^{\text{per}}_{i,j}(t) \big)\big(\alpha |c| - \dfrac{c^4}{6} \cosh \tilde{c}\big)
\\[0.2cm]
\end{array}
\end{equation}
for some $\tilde{c}\in [0, |c|]$,
which using \eqref{eq:per:choice:K} and   \eqref{eq:per:choice:alpha}
implies
\begin{equation}
    \begin{array}{lcl}
\mathcal{J}[w]_{i,j} 
& \ge &
(w_{i,j}-\tilde{u}^{\text{per}}_{i,j})\big(\alpha |c| - \dfrac{c^4}{6} \cosh \tilde{c} - K \big)
\\[0.2cm]
& \ge & 0.
\end{array}
\end{equation}
In particular, the comparison principles allows us to conclude that
\begin{equation}
    u_{i,j}(\tau) \leq u^{\mathrm{per}}_{i,j}(\tau) + \dfrac{\epsilon}{2} + Me^{|c|(i + j+2 |c|\tau + \alpha \tau)}, 
\end{equation}
which implies that there exists $L_{\epsilon} \gg 1$ so that
\begin{equation}
    u_{i,j}(\tau) \le u^{\mathrm{per}}_{i,j}(\tau) +\epsilon
\end{equation}
for $i + j \le - L_{\epsilon}$.
An analogous lower bound can be obtained by exploiting similar sub-solutions, which completes the proof.
\end{proof}

\begin{proof}[Proof of Proposition \ref{prp:per:phase:cmp}]
For any sufficiently large $t \ge 1$ and $(i,j) \in \Z^2$ we may estimate
\begin{equation}
    \begin{array}{lcl}
      \Phi\big(i - \gamma_j^{\mathrm{per}}(t) \big)
       - \Phi\big(i - \gamma_j(t) \big)
       & \le &
         |\Phi\big(i - \gamma_j^{\mathrm{per}}(t) \big)
         - u^{\mathrm{per}}_{i,j}(t) |
         + | u_{i,j}(t) - \Phi\big(i - \gamma_j(t) \big) |
         \\[0.2cm]
         & & \qquad 
           + | u^{\mathrm{per}}_{i,j}(t) - 
             u_{i,j}(t) |.
    \end{array}
\end{equation}
Applying Theorem \ref{thm:main results:gamma_approximates_u} and Lemma 
\ref{lemma:periodicity+decay:u_convg_to_uper},
we find a constant  $T_{\epsilon} > 0$
and a function $J_{\epsilon}: [T_{\epsilon}, \infty) \to \Z_{\ge 0}$ 
for which
we have
\begin{equation}
\label{eq:per:phase:diff:phi}
    \begin{array}{lcl}
      \Phi\big(i - \gamma_j^{\mathrm{per}}(t) \big)
       - \Phi\big(i - \gamma_j(t) \big)
       & \le & 3 \epsilon
    \end{array}
\end{equation}
for all $t \ge T_{\epsilon}$ and $|j| \ge J_{\epsilon}(t)$.
Recalling the constant $M > 0$ from Lemma \ref{lemma:phase gamma:monotonicity_in_the_bounded_region} and writing
\begin{equation}
    \nu =\mathrm{min} \{ \Phi'(\xi) :  |\xi| \le M+1   \} >0,
\end{equation}
we may substitute $i = \lceil ct \rceil$ into
\eqref{eq:per:phase:diff:phi} to obtain
\begin{equation}
        \nu | \gamma_j^{\mathrm{per}}(t)
     - \gamma_j(t) |
     \le 
      \Phi\big( \lceil ct \rceil - \gamma_j^{\mathrm{per}}(t) \big)
       - \Phi\big(\lceil ct \rceil - \gamma_j(t) \big)
        \le  3 \epsilon
\end{equation}
for all $t \ge T_{\epsilon}$ and $|j| \ge J_{\epsilon}(t)$.
This yields the desired result after some minor relabelling.
\end{proof}

\subsection{Phase asymptotics}

It remains to establish Proposition \ref{prp:per:beh:alm:per:nl:heat}.
We accomplish this by using the Cole-Hopf transformation
discussed in {\S}\ref{sec:dht} to transform
\eqref{eq:sub:sup:lde:V} into the linear heat LDE \eqref{eqn:discrete_heat_eq:DHK}. The bounds
in {\S}\ref{sec:dht} readily allow us to analyze
solutions with initial conditions that are asymptotically
`almost-periodic'.

\begin{lemma}
\label{lem:per:heat:eqn}
Pick an integer $P \ge 1$ and let 
$h\in  C^1\big([0, \infty); \ell^\infty(\Z)\big) $ be  a solution to the discrete heat equation \eqref{eqn:discrete_heat_eq:DHK} 
with an initial condition $h^0 \in \ell^\infty(\Z)$ 
that satisfies $h^0_{j+P} = h^0_j$ for all $j \in \Z$.
Then upon introducing the average
\begin{equation}
    \overline{h} = \frac{1}{P} \sum_{j=0}^{P-1} h^0_j,
\end{equation}
we have the limit
\begin{equation}
		\lim_{t\to\infty}\norm{h(t)-\overline{h}}_{\ell^\infty} = 0. 
	\end{equation}
\end{lemma}
\begin{proof}
Since $h$ inherits the periodicity of $h^0$, the function
\begin{equation}
H_j(t) = \frac{1}{P} \sum_{k=0}^{P-1} h_{j+k}(t)
\end{equation}
is constant with respect to $j$. Since it also satisfies
\eqref{eqn:discrete_heat_eq:DHK}, we must have $H_j(t) = \overline{h}$.
The result now follows from the fact that $\norm{\partial^+ h(t)}_{\ell^\infty} \to 0$
as $t \to \infty$; see \eqref{delta+ v}.
\end{proof}

\begin{proof}[Proof of Proposition \ref{prp:per:beh:alm:per:nl:heat}]
We first treat the case $d\neq 0$
and write $V^{\mathrm{per}}$ for the
solution to the nonlinear LDE \eqref{eq:sub:sup:lde:V}
with initial condition $V^{\mathrm{per}}(0) = V^{0;\mathrm{per}}$.
Without loss, we may assume that $V^{0;\mathrm{per}}_0 = 0$.
Inspired by the proof of Corollary
\ref{cor:interface_evolution:cole hopf},
we introduce the functions
\begin{equation}
\label{eq:per:alm:per:def:h:h:per}
h^{\mathrm{per}}(t) = e^{d (V^{\mathrm{per}}(t) -c t)},
\qquad
h = e^{d (V(t) - ct)},
\qquad
q
(t) = e^{ d V(t) - d V^{\mathrm{per}}(t) } - 1
\end{equation}
and note that $h^{\mathrm{per}}$
and $h$ both satisfy the linear heat LDE 
\eqref{eqn:discrete_heat_eq:DHK}. By construction,
we have
\begin{equation}
    h(0) = h^{\mathrm{per}}(0) + h^{\mathrm{per}}(0) q(0), 
\end{equation}
which allows us to write
\begin{equation}
\label{eq:per:sol:h:vs:per}
    h_j(t) - h^{\mathrm{per}}_j(t)
     =  \sum_{k \in \Z} G_k(t) h^{\mathrm{per}}_{j-k}(0) q_{j-k}(0) . 
\end{equation}

Assuming $0 < \epsilon <1$ and $R \ge 1$, we see that
\begin{equation}
    |V^0_0| \le R + |V_J| \le R + \epsilon + V^0_J
    \le 2R + 1 \le 3R
\end{equation}
and hence $\norm{V^0}_{\ell^\infty}\le 4R$.
This allows us to obtain the global bounds
\begin{equation}
    \norm{h^{\mathrm{per}}(0)}_{\ell^\infty}
      \le e^{|d| R},
      \qquad \qquad
     \norm{q(0)}_{\ell^\infty} 
      \le e^{5 |d| R} + 1,
\end{equation}
together with the tail bound
\begin{equation}
    |q_j(0)| \le e^{|d| \epsilon} - 1, \qquad \qquad j \ge |J|. 
\end{equation}
Using \eqref{eq:per:sol:h:vs:per}, these bounds allow us to
obtain the estimate
\begin{equation}
\begin{array}{lcl}
    \norm{h(t) - h^{\mathrm{per}}(t)}_{\ell^\infty}
    & \le &
    \sum_{|j-k| \ge J} |G_k(t)| (e^{ |d| \epsilon} - 1) 
    + \sum_{|j-k| < J} |G_k(t)|  e^{|d| R} ( e^{5 |d| R} + 1 )
    \\[0.2cm]
    & \le & 
        (e^{ |d| \epsilon} - 1)  \norm{G(t)}_{\ell^1} 
        + (2 J -1 )  e^{ |d| R} ( e^{5 |d| R} + 1 ) \norm{G(t)}_{\ell^\infty}.
\end{array}
\end{equation}
Since $\norm{G(t)}_{\ell^1} =1$
on account of \eqref{eqn:bessel_functions:generating_function}
and $\norm{G(t)}_{\ell^\infty} \le C t^{-1/2}$ on account of \eqref{eqn:Bessel_functions:asymptotics_of_I_n}, we can find
a time $T = T(\epsilon,J, d, R)$ 
so that
\begin{equation}
    \norm{h(t) - h^{\mathrm{per}}(t)}_{\ell^\infty} \le 
    2 (e^{ |d| \epsilon} - 1)
\end{equation}
for all $t \ge T$. 

After possibly increasing $T$, we can use
Lemma \ref{lem:per:heat:eqn} to conclude
\begin{equation}
    \norm{h(t) - \overline{h}}_{\ell^\infty} \le 
    4 (e^{ |d| \epsilon} - 1), \qquad \qquad t \ge T,
\end{equation}
for some $\overline{h} \in [0, e^{ |d| R}]$. Inverting
the transformation \eqref{eq:per:alm:per:def:h:h:per} hence leads
to the desired bound on $V$
with $\mu = \frac{\ln \overline{h}}{d}$.
The remaining case $d = 0$ can be treated in the same fashion as above,
but now one does not need to use the nonlinear coordinate transformation.
\end{proof}

\bibliographystyle{klunumHJ}
\bibliography{refs}

\end{document}